\documentclass[11pt]{amsart}
\usepackage{amssymb,amsmath,mathabx,mathtools,bm}
\usepackage{a4wide}
\newtheorem{theorem}{Theorem}[section]
\newtheorem{lemma}[theorem]{Lemma}
\newtheorem{proposition}[theorem]{Proposition}
\newtheorem{corollary}[theorem]{Corollary}
\newtheorem{claim}{Claim}[section]
\theoremstyle{remark}
\newtheorem{remark}{Remark}[section]
\numberwithin{equation}{section}
\newcommand{\R}{\mathbb{R}}
\newcommand{\C}{\mathbb{C}}

\newcommand{\Z}{\mathbb{Z}}
\newcommand{\T}{\mathbb{T}}
\newcommand{\la}{\langle}
\newcommand{\ra}{\rangle}
\newcommand{\pd}{\partial}
\renewcommand{\a}{\alpha}
\newcommand{\eps}{\varepsilon}
\newcommand{\wP}{\widetilde{P}}
\newcommand{\wR}{\widetilde{R}}
\newcommand{\wS}{\widetilde{S}}
\newcommand{\cC}{\mathcal{C}}
\newcommand{\wC}{\widetilde{\mathcal{C}}}
\newcommand{\mF}{\mathcal{F}}
\newcommand{\mL}{\mathcal{L}}

\newcommand{\bS}{\bar{S}}
\newcommand{\bM}{\mathbb{M}}
\newcommand{\cN}{\mathcal{N}}
\newcommand{\fN}{\mathfrak{N}}

\newcommand{\ta}{\tilde{a}}
\newcommand{\tb}{\tilde{b}}
\newcommand{\tc}{\tilde{c}}
\newcommand{\td}{\tilde{d}}

\newcommand{\tk}{\tilde{k}}

\newcommand{\tv}{\tilde{v}}
\newcommand{\tx}{\tilde{x}}
\newcommand{\bb}{\mathbf{b}}
\newcommand{\bd}{\mathbf{d}}
\newcommand{\tbd}{\mathbf{\tilde{d}}}
\newcommand{\tpsi}{\tilde{\psi}}
\newcommand{\oc}[1]{\overset{\circ}{#1}}
\newcommand{\obu}[1]{\overset{\bullet}{#1}}
\DeclareMathOperator{\sech}{sech}

\DeclareMathOperator{\supp}{supp}
\DeclareMathOperator{\diag}{diag}

\begin{document}
\title
{The phase shift of line solitons for the KP-II equation}
\author{Tetsu Mizumachi}
\thanks{Department of Mathematics, Hiroshima University,
1-7-1 Kagamiyama 739-8521, Japan\\
Email: tetsum@hiroshima-u.ac.jp}
\keywords{KP-II, line soliton, phase shift}
\subjclass[2010]{Primary 35B35, 37K40; Secondary 35Q35}
\begin{abstract}
The KP-II equation was derived by Kadmotsev and Petviashvili \cite{KP}
to explain stability of line solitary waves of shallow water.
Stability of line solitons has been proved by \cite{Miz15,Miz17}
and it turns out the local phase shift of modulating line solitons
are not uniform in the transverse direction.
In this paper, we obtain the $L^\infty$-bound for the local phase
shift of modulating line solitons for polynomially localized perturbations.
\end{abstract}

\maketitle
\tableofcontents

\section{Introduction}
\label{sec:intro}
The KP-II equation
\begin{equation}\label{eq:KPII}
\partial_x(\pd_tu+\pd_x^3u+3\pd_x(u^2))+3\sigma\partial_y^2u=0
\quad\text{for $t>0$ and $(x,y)\in \R^2$,}
\end{equation}
where $\sigma=1$,
is a generalization to two spatial dimensions of the KdV equation
\begin{equation}
  \label{eq:KdV}
\pd_tu+\pd_x^3u+3\pd_x(u^2)=0\,,
\end{equation}
and has been derived as a model to explain the transverse stability
of solitary wave solutions to the KdV equation with respect to two dimensional perturbation when the surface tension
is weak or absent.  See \cite{KP} for the derivation of \eqref{eq:KPII}.
\par
The global well-posedness of \eqref{eq:KPII} in $H^s(\R^2)$ ($s\ge0$)
on the background of line solitons has been studied by Molinet, Saut
and Tzvetkov \cite{MST} whose proof is based on the work of Bourgain
\cite{Bourgain}.  For the other contributions on the Cauchy problem of
the KP-II equation, see e.g.
\cite{GPS,Hadac,HHK,IM,Tak,TT,Tz,Ukai} and the references therein.
\par
Let
$$\varphi_c(x)\equiv c\sech^2\Big(\sqrt{\frac{c}{2}}\,x\Big),\quad c>0.
$$
Then $\varphi_c(x-2ct)$ is a solitary wave solution of the KdV equation
\eqref{eq:KdV} and a line soliton solution of \eqref{eq:KPII} as well.
Transverse linear stability of line solitons for the KP-II equation 
was studied by Burtsev (\cite{Burtsev}).
See also \cite{APS} for the spectral stability of KP line solitons.
Recently, transverse spectral and linear stability of periodic waves
for the KP-II equation has been studied in \cite{Haragus,HLP,JZ}.

If $\sigma=-1$, then \eqref{eq:KPII} is called KP-I which is a model for
long waves in a media with positive dispersion, e.g. water waves with large
surface tension. The KP-I equation has a stable ground state
(\cite{dBS}) and line solitons are unstable for the KP-I equation
except for thin domains in $\R^2$ where the two dimensional nature
of the equation is negligible (see \cite{RT1, RT2, RT3, Z}).
\par
Nonlinear stability of line solitons for the KP-II equation has been proved 
for localized perturbations as well as for perturbations which have
$0$-mean along all the lines parallel to the $x$-axis (\cite{Miz15, Miz17}).

\begin{theorem}\emph{(\cite[Theorem~1.1]{Miz17})}
  \label{thm:stability0}
Let $c_0>0$ and $u(t,x,y)$ be a solution of \eqref{eq:KPII} satisfying
$u(0,x,y)=\varphi_{c_0}(x)+v_0(x,y)$.
There exist positive constants $\eps_0$ and $C$ satisfying the following:
if $v_0\in \pd_xL^2(\R^2)$ and
$\|v_0\|_{L^2(\R^2)}+\||D_x|^{1/2}v_0\|_{L^2(\R^2)}+\||D_x|^{-1/2}|D_y|^{1/2}v_0\|_{L^2(\R^2)} <\eps_0$
then there exist $C^1$-functions $c(t,y)$ and $x(t,y)$
such that for every $t\ge0$ and $k\ge0$,
\begin{align}
\label{OS}
 & \|u(t,x,y)-\varphi_{c(t,y)}(x-x(t,y))\|_{L^2(\R^2)}\le C\|v_0\|_{L^2}\,,
 \\  & 
\left\|c(t,\cdot)-c_0\right\|_{H^k(\R)}+\left\|\pd_yx(t,\cdot)\right\|_{H^k(\R)}
+\|x_t(t,\cdot)-2c(t,\cdot)\|_{H^k(\R)}\le C\|v_0\|_{L^2}\,,
\\  &  \label{phase2}
\lim_{t\to\infty}\left(\left\|\pd_yc(t,\cdot)\right\|_{H^k(\R)}+
\left\|\pd_y^2x(t,\cdot)\right\|_{H^k(\R)}\right)=0\,,
\end{align}
and for any $R>0$,
\begin{equation}
\label{AS}
\lim_{t\to\infty}
\left\|u(t,x+x(t,y),y)-\varphi_{c(t,y)}(x)\right\|_{L^2((x>-R)\times\R_y)}=0\,.
\end{equation}
\end{theorem}
\begin{theorem}\emph{(\cite[Theorem~1.2]{Miz17})}
  \label{thm:stability}
Let $c_0>0$ and $s>1$. Suppose that $u$ is a solutions of 
\eqref{eq:KPII} satisfying  $u(0,x,y)=\varphi_{c_0}(x)+v_0(x,y)$.
Then there exist positive constants $\eps_0$ and $C$ such that 
if $\|\la x\ra^sv_0\|_{H^1(\R^2)} <\eps_0$,
there exist $c(t,y)$ and $x(t,y)$ satisfying \eqref{phase2}, \eqref{AS} and
\begin{align}
\label{OS'} &
\|u(t,x,y)-\varphi_{c(t,y)}(x-x(t,y))\|_{L^2(\R^2)}
\le C\|\la x\ra^sv_0\|_{H^1(\R^2)}\,,\\
\label{phase-sup'} & 
\left\|c(t,\cdot)-c_0\right\|_{H^k(\R)}+\left\|\pd_yx(t,\cdot)\right\|_{H^k(\R)}
+\|x_t(t,\cdot)-2c(t,\cdot)\|_
{H^k(\R)}\le C\|\la x\ra^sv_0\|_{H^1(\R^2)}
\end{align}
for every $t\ge0$ and $k\ge0$.
\end{theorem}
\begin{remark}
The parameters $c(t_0,y_0)$ and $x(t_0,y_0)$ represent the local amplitude
and the local phase shift of the modulating
line soliton $\varphi_{c(t,y)}(x-x(t,y))$ at time $t_0$ along the line $y=y_0$
and that $x_y(t,y)$  represents the local orientation of the crest of the
line soliton.    
\end{remark}
\begin{remark}
In view of Theorem~\ref{thm:stability0},
$$\lim_{t\to\infty}\sup_{y\in\R}(|c(t,y)-c_0|+|x_y(t,y)|)=0\,,$$
and as $t\to\infty$, the modulating line soliton $\varphi_{c(t,y)}(x-x(t,y))$
converges to a $y$-independent modulating line soliton
$\varphi_{c_0}(x-x(t,0))$ in $L^2(\R_x\times (|y|\le R))$ for any $R>0$.  
\end{remark}

For the KdV equation as well as for the KP-II equation posed on 
$L^2(\R_x\times \T_y)$, the dynamics of a modulating soliton
$\varphi_{c(t)}(x-x(t))$ is described by a system of ODEs
  \begin{equation*}
\dot{c}\simeq 0\,,\quad \dot{x}\simeq 2c\,.
\end{equation*}
See \cite{PW} for the KdV equation and \cite{MT} for the KP-II equation
with the $y$-periodic boundary condition.
However, to analyze transverse stability of line solitons
for localized perturbation in $\R^2$,
we need to study a system of PDEs for $c(t,y)$ and $x(t,y)$
in \cite{Miz15,Miz17} as is the case with the planar traveling waves
for the heat equations (e.g. \cite{Kapitula,Le-Xin,Xin}) 
and planar kinks for the $\phi^4$-model (\cite{Cu}).
\par
By analyzing modulation PDEs, it turns out the set of exact $1$-line solitons
$$\mathcal{K}=\{\varphi_c(x+ky-(2c+3k^2)t+\gamma)
\mid c>0\,,\, k\,,\gamma\in\R\}$$
is not stable in $L^2(\R^2)$.
\begin{theorem}\emph{(\cite[Theorem~1.4]{Miz15})}
  \label{thm:instability}
Let $c_0>0$. Then for any $\eps>0$, 
there exists a solution of \eqref{eq:KPII} such that
$\|u(0,x,y)-\varphi_{c_0}(x)\|_{L^2(\R^2)}<\eps$ and
$\liminf_{t\to\infty}t^{-1/4}\inf_{v\in\mathcal{A}}
\|u(t,\cdot)-v\|_{L^2(\R^2)}>0$.  
\end{theorem}

\begin{remark}
Theorem~\ref{thm:instability} is a consequence of
finite speed propagation of local phase shifts and
the fact that the line solitons have infinite length in the $\R^2$ case.
Indeed, the phase $x(t,y)$ has \textit{jumps} around
the points $y=\pm\sqrt{8c_0}t$.
\par
Such phenomena are observed for Boussinesq equations in the physics
literature. See e.g. \cite{Ped} and the reference therein.
\end{remark}

The following result is an improvement of \cite[Theorem~1.5]{Miz15}.
\begin{theorem}
  \label{thm:burgers}
 Let $c_0=2$ and $u(t)$ be as in Theorem~\ref{thm:stability}. 
 There exist positive constants $\eps_0$ and $C$ such that if 
$\eps:=\|\la x\ra(\la x\ra+\la y\ra) v_0\|_{H^1(\R^2)} <\eps_0$,
then there exist $C^1$-functions $c(t,y)$ and $x(t,y)$ satisfying
\eqref{OS}--\eqref{AS} and
\begin{equation}
  \label{eq:profile}
\left\|\begin{pmatrix} c(t,\cdot)-2\\ x_y(t,\cdot) \end{pmatrix}
-\begin{pmatrix}  2 & 2 \\ 1 & -1\end{pmatrix}
\begin{pmatrix}u_B^+(t,y+4t) \\ u_B^-(t,y-4t)\end{pmatrix}
\right\|_{L^2(\R)}=o(\eps t^{-1/4})
\end{equation}
as $t\to\infty$, where $u_B^\pm$ are self similar solutions of
the Burgers equation 
$$\pd_tu=2\pd_y^2u\pm 4\pd_y(u^2)$$
 such that 
$$u_B^\pm(t,y)=\frac{\pm m_\pm H_{2t}(y)}
{2\left(1+m_\pm \int_0^{y}H_{2t}(y_1)\,dy_1\right)}\,,
\quad H_t(y)=(4\pi t)^{-1/2}e^{-y^2/4t}\,,$$
and that $m_\pm$ are constants satisfying
$$\int_\R u_B^\pm(t,y)\,dy=\frac14\int_\R \left(c(0,y)-2\right)\,dy+O(\eps^2)\,.$$
\end{theorem}
\begin{remark}
 Since \eqref{eq:KPII} is invariant under the scaling
$u\mapsto \lambda^2 u(\lambda^3t,\lambda x,\lambda^2y)$,
we may assume that $c_0=2$ without loss of generality.
\end{remark}
\begin{remark}
  The linearized operator around the line soliton solution has
  resonant continuous eigenvalues near $\lambda=0$ whose corresponding
  eigenmodes grow exponentially as $x\to-\infty$. See
  \eqref{eq:resonance1}--\eqref{eq:resonance3}.  The diffraction of
  the line soliton around $y=\pm4t$ can be thought as a mechanism to
  emit energy from those resonant continuous eigenmodes.
\end{remark}

If we disregard diffractions of waves propagating along
the crest of line solitons, then time evolution of the phase shift
is approximately described by the $1$-dimensional wave equation
$$x_{tt}=8c_0x_{yy}\,.$$
It is natural to expect that $\sup_{t,y\in\R}|x(t,y)-2c_0t|$ remains
small for localized perturbations although the $L^2(\R_y)$ norm of
$x(t,y)-2c_0t$ grows as $t\to\infty$. 

Our main result in the present paper is the following.
\begin{theorem}
  \label{thm:phase}
Let $u(t,x,y)$ and $x(t,y)$ be as in Theorem~\ref{thm:stability}.
There exist positive constants $\eps_0$ and $C$ such that if
$\eps:=\|\la x\ra(\la x\ra+\la y\ra)v_0\|_{H^1(\R^2)}<\eps_0$, then
$\sup_{t\ge0\,,\,y\in\R}|x(t,y)-2c_0t|\le C\eps$.
\par
Moreover, there exists an $h\in \R$ such that for any $\delta>0$,
\begin{equation}
  \label{eq:phase-lim}
  \begin{cases}
\lim_{t\to\infty}\left\|x(t,\cdot)-2c_0t-h
\right\|_{L^\infty(|y|\le (\sqrt{8c_0}-\delta)t}=0\,,
\\    
\lim_{t\to\infty}\left\|x(t,\cdot)-2c_0t
\right\|_{L^\infty(|y|\ge (\sqrt{8c_0}+\delta)t)}=0\,.
  \end{cases}
\end{equation}
\end{theorem}
In the case where $h\ne0$ in \eqref{eq:phase-lim}, 
the $L^2(\R^2)$-distance between the solution $u$ and the set of exact
$1$-line solitons grows like $t^{1/2}$ or faster.
\begin{corollary}
  \label{cor:instability2}
Let $c_0>0$. Then for any $\eps>0$, 
there exists a solution of \eqref{eq:KPII} such that
$\left\|\la x\ra(\la x\ra+\la y\ra)\left\{u(0,x,y)-\varphi_{c_0}(x)\right\}
\right\|_{H^1(\R^2)}<\eps$ and
$\liminf_{t\to\infty}t^{-1/2}\inf_{v\in\mathcal{A}}
\|u(t,\cdot)-v\|_{L^2(\R^2)}>0$.  
\end{corollary}
To investigate the large time behavior of $x(t,y)$, we derive estimates of
fundamental solutions to the linearized equation of modulation equations
for parameters $c(t,y)$ and $x(t,y)$ which is a $1$-dimensional
damped wave equation (see Section~\ref{subsec:dampedW}).
As is the same with the $1$-dimensional wave equation,
we need integrability of the initial data of the modulation equation
to prove the boundedness of the phase shift.

In our construction of modulation parameters, we impose a secular term
condition on $c(t,y)$ and $x(t,y)$ only for $y$-frequencies in a small
interval $[-\eta_0,\eta_0]$. This facilitates the estimates of
modulation parameters because the truncation of Fourier modes turns
the modulation equations into semilinear equations.  On the other
hand, it was not clear in \cite{Miz15} whether the initial data of
modulation equations are integrable even if perturbations to line
solitons are exponentially localized.
We find that $c(0,y)$ can be decomposed into a sum of an integrable function
and a derivative of a function that belongs to $\mF^{-1}L^\infty(\R)$
for polynomially localized perturbations in $\R^2$.

The decomposition of initial data also enables us to prove
Theorem~\ref{thm:burgers} which shows the large time asymptotic
of the local amplitude and the local orientation of line solitons
in $L^2(\R)$ whereas the result in \cite{Miz15} shows
large time asymptotics in a region $y=\pm\sqrt{8c_0}t+O(\sqrt{t})$.
\par
In \cite{MizShim}, we study the $2$-dimensional linearized Benney-Luke
equation around line solitary waves in the weak surface tension case
and find that the time evolution of resonant continuous eigenmondes is
similar to \eqref{eq:phase-lim}. We except our argument presented in
this paper is useful to investigate phase shifts of modulating line
solitary waves for the $2$-dimensional Benney-Luke equation and the
other long wave models for $3$D water.
\par

Finally, let us introduce several notations. 
Let $\mathbf{1}_A$ be the characteristic function of the set $A$.
For Banach spaces $V$ and $W$, let $B(V,W)$ be the space of all the
linear continuous operators from $V$ to $W$ and
$\|T\|_{B(V,W)}=\sup_{\|x\|_V=1}\|Tu\|_W$ for $T\in B(V,W)$.
We abbreviate $B(V,V)$ as $B(V)$.
For $f\in \mathcal{S}(\R^n)$ and $m\in \mathcal{S}'(\R^n)$, let 
\begin{gather*}
(\mathcal{F}f)(\xi)=\hat{f}(\xi)
=(2\pi)^{-n/2}\int_{\R^n}f(x)e^{-ix\xi}\,dx\,,\\
(\mathcal{F}^{-1}f)(x)=\check{f}(x)=\hat{f}(-x)\,,
\end{gather*}
and $(m(D)f)(x)=(2\pi)^{-n/2}(\check{m}*f)(x)$.
\par
The symbol $\la x\ra$ denotes $\sqrt{1+x^2}$ for $x\in\R$.
We use $a\lesssim b$ and $a=O(b)$ to mean that there exists a
positive constant such that $a\le Cb$. 
Various constants will be simply denoted
by $C$ and $C_i$ ($i\in\mathbb{N}$) in the course of the
calculations. 
\bigskip

\section{Preliminaries}
\label{sec:preliminaries}

\subsection{Semigroup estimates for the linearized KP-II equation}
\label{subsec:LKP}
First, we recall decay estimates of the semigroup generated by the linearized operator
around a $1$-line soliton in exponentially weighted spaces.
\par
Let $$\varphi=\varphi_2\,,\quad
\mL=-\pd_x^3+4\pd_x-3\pd_x^{-1}\pd_y^2-6\pd_x(\varphi \cdot)\,.$$
We remark that $\mL$ generates a $C^0$-semigroup on $X:=L^2(\R^2;e^{2\a x}dxdy)$
for any $\a>0$.
\par
Let $\mL(\eta)=-\pd_x^3+4\pd_x+3\eta^2\pd_x^{-1}-6\pd_x(\varphi \cdot)$
be an operator on $L^2(\R;e^{2\a x}dx)$ with its domain $D(\mL(\eta))=e^{-\a x}H^3(\R)$.
Obviously, we have $\mL(u(x)e^{iy\eta})=e^{iy\eta}\mL(\eta)u(x)$
for any $\eta\in\R$.
If $\eta\simeq0$, then $\mL(\eta)$ has two isolated eigenvalues near $0$ 
and the rest of the spectrum is bounded away from 
the imaginary axis and lies in the stable half plane
(see \cite[Chapter~2]{Miz15}).
We remark that that $\mL(0)$ is the linearized KdV operator around $\varphi$
which has an isolated $0$ eigenvalue of multiplicity $2$
in $L^2(\R;e^{2\a x}dx)$ with $\a\in(0,2)$ (see \cite{PW}).
\par
Let 
\begin{gather}
\label{eq:resonance1}
\beta(\eta)=\sqrt{1+i\eta}\,,\quad \lambda(\eta)=4i\eta\beta(\eta)\,,
\\ \label{eq:resonance2}
g(x,\eta)=\frac{-i}{2\eta\beta(\eta)}
\pd_x^2(e^{-\beta(\eta)x}\sech x),\quad
g^*(x,\eta)=\pd_x(e^{\beta(-\eta)x}\sech x)\,.
\end{gather}
Then
\begin{equation}
  \label{eq:resonance3}
\mL(\eta)g(x,\pm\eta)=\lambda(\pm\eta)g(x,\pm\eta)\,,\quad
\mL(\eta)^*g^*(x,\pm\eta)=\lambda(\mp\eta)g^*(x,\pm\eta)\,.
\end{equation}
The continuous eigenvalues $\lambda(\eta)$ belongs to the stable half plane
$\{\lambda\in \C\mid \Re\lambda<0\}$ for $\eta\in\R\setminus\{0\}$ and
$\lambda(\eta)\to\lambda(0)=0$ as $\eta\to0$.
\par
Let $\nu(\eta):=\Re\beta(\eta)-1$ and $\eta_0$ be a small positive number. 
Since $g(x,\eta)=O(e^{\nu(\eta)|x|})$ as $x\to-\infty$
and $\nu(\eta)=O(\eta^2)$ for small $\eta$, we choose $\a$ and so that
$\a\ge \nu(\eta)$ and $g(x,\eta)\in L^2(\R;e^{2\a x}dx)$ for $\eta\in[-\eta_0,\eta_0]$. The continuous eigenmodes $g(x,\eta)e^{iy\eta}$ grow
exponentially as $x\to-\infty$. Nevertheless, they have to do with modulation of line solitons.
See \cite{Burtsev} and the references therein.
\par
The spectral projection to the continuous eigenmodes
$\{g_\pm(x,\eta)\}_{-\eta_0\le \eta\le \eta_0}$ is given by
\begin{gather*}
P_0(\eta_0)f(x,y)=\frac{1}{\sqrt{2\pi}}\sum_{k=1,\,2}
\int_{-\eta_0}^{\eta_0}a_k(\eta)g_k(x,\eta)e^{i y\eta}\,d\eta\,,
\\
 a_k(\eta)=\int_\R (\mF_yf)(x,\eta)g_k^*(x,\eta)\,d x\,,
\end{gather*}
where
\begin{gather*}
g_1(x,\eta)=2\Re g(x,\eta)\,,\quad g_2(x,\eta)=-2\eta\Im g(x,\eta)\,,\\
g_1^*(x,\eta)=\Re g^*(x,\eta)\,,\quad g_2^*(x,\eta)=-\eta^{-1}\Im g^*(x,\eta)\,.
\end{gather*}
We remark that for an $\a\in(0,2),$
\begin{align*}
& g_1(x,\eta)=\frac14\varphi'+\frac{x}{4}\varphi'+\frac{1}{2}\varphi
+O(\eta^2)\,,\quad
g_2(x,\eta)=-\frac{1}{2}\varphi'+O(\eta^2) 
\quad\text{in $L^2(\R;e^{2\a x}dx)$,}\\
& g_1^*(x,\eta)=\frac12\varphi+O(\eta^2)\,,\quad
g_2^*(x,\eta)=\int_{-\infty}^x\pd_c\varphi dx+O(\eta^2)
\quad\text{in $L^2(\R;e^{-2\a x}dx)$,}
\end{align*}
where $\pd_c\varphi=\pd_c\varphi_c|_{c=2}$.
See \cite[Chapter~3]{Miz15}.
\par
For $\eta_0$ and $M$ satisfying $0<\eta_0\le M\le \infty$, let
\begin{gather*}
P_1(\eta_0, M)u(x,y):=\frac{1}{2\pi}\int_{\eta_0\le |\eta|\le M}
\int_\R  u(x,y_1)e^{i\eta(y-y_1)}\,d y_1d\eta\,,
\\ P_2(\eta_0,M):= P_1(0,M)-P_0(\eta_0)\,.
\end{gather*}
The semigroup $e^{t\mL}$ is exponentially stable on $(I-P_0(\eta_0))X$.
\begin{proposition}\emph{(\cite[proposition~3.2 and Corollary~3.3]{Miz15})}
\label{prop:semigroup-est}
Let $\a\in (0,2)$ and $\eta_1$ be a positive number satisfying
$\nu(\eta_1)<\a$.  Then there exist positive constants
$K$ and  $b$ such that for any $\eta_0\in(0,\eta_1]$, $M\ge \eta_0$,
$f\in X$ and $t\ge 0$,
$$ \|e^{t\mL}P_2(\eta_0,M)f\|_X\le Ke^{-bt}\|f\|_X\,.$$
Moreover, there exist positive constants $K'$ and $b'$ such that for $t>0$,
\begin{gather*}
\|e^{t\mL}P_2(\eta_0,M)\pd_xf\|_X \le
K'e^{-b' t}t^{-1/2}\|e^{\a x}f\|_X\,,\\
\|e^{t\mL}P_2(\eta_0,M)\pd_xf\|_X \le
K'e^{-b' t}t^{-3/4}\|e^{\a x}f\|_{L^1_xL^2_y}\,.  
\end{gather*}
\end{proposition}

\subsection{Decay estimates for linearized modulation equations}
\label{subsec:dampedW}
Time evolution of parameters $c(t,y)$ and $x(t,y)$
of a modulating line soliton $\varphi_{c(t,y)}(x-x(t,y))$ is described
by a system of Burgers type equations.
In this subsection, we introduce linear estimates which will be used to prove
boundedness of the phase shift $x(t,y)-2c_0t$. The estimates are a substitute
of d'Alembert's formula for the $1$-dimensional wave equation.
\par

Let $\omega(\eta)=\sqrt{16+(8\mu_3-1)\eta^2}$,
$\mu_3=-\frac{\mu_1}{2}+\frac{3}{4}=\frac{1}{2}+\frac{\pi^2}{24}>1/8$,
$\lambda_*^\pm(\eta)=-2\eta^2\pm i\eta\omega(\eta)$ and
$$\mathcal{A}_*(\eta)=
\begin{pmatrix} -3\eta^2 & -8\eta^2 \\ 2+\mu_3\eta^2 & -\eta^2\end{pmatrix}\,,
\quad \mathcal{P}_*(\eta)=\frac{1}{4\eta}
\begin{pmatrix}
  8\eta & 8\eta \\ -\eta-i\omega(\eta) & -\eta+i\omega(\eta)
\end{pmatrix}\,.$$
Then $\mathcal{P}_*(\eta)^{-1}\mathcal{A}_*(\eta)\mathcal{P}_*(\eta)
=\diag(\lambda_*^+(\eta),\lambda_*^-(\eta))$ and
\begin{equation}
  \label{eq:etA}
\begin{split}
e^{t\mathcal{A}_*(\eta)}
=& e^{-2t\eta^2}
\begin{pmatrix}
  \cos t\eta\omega(\eta)-\frac{\eta}{\omega(\eta)}\sin t\eta\omega(\eta)
& -\frac{8\eta}{\omega(\eta)}\sin t\eta\omega(\eta) 
\\ \frac{\eta^2+\omega(\eta)^2}{8\eta\omega(\eta)}\sin t\eta\omega(\eta)
&   \cos t\eta\omega(\eta)+\frac{\eta}{\omega(\eta)}\sin t\eta\omega(\eta)
\end{pmatrix}\,.
\end{split}
\end{equation}
Let $\eta_0$ be a positive number and let
$\chi_1(\eta)$ be a nonnegative smooth function such that
$0\le \chi_1(\eta)\le 1$ for $\eta\in\R$,
$\chi_1(\eta)=1$ if $|\eta|\le \frac12\eta_0$ and $\chi_1(\eta)=0$ if
$|\eta|\ge \frac{3}{4}\eta_0$.
Let $\chi_2(\eta)=1-\chi_1(\eta)$.
Then 
\begin{equation}
  \label{eq:k-est,high}
\|\chi_2(D_y)e^{t\mathcal{A}_*(D_y)}\|_{B(L^2(\R))}\lesssim e^{-\eta_0^2t/2}
\quad\text{for $t\ge0$.}
\end{equation}
\par
Next, we will estimate the low frequency part of $e^{t\mathcal{A}_*(\eta)}$.
Let  
\begin{align*}
& K_1(t,y)=\frac{1}{\sqrt{2\pi}}\mathcal{F}^{-1}
\left(\chi_1(\eta)e^{-2t\eta^2}\cos t\eta\omega(\eta)\right)\,,
\\ &
K_2(t,y)=\frac{1}{\sqrt{2\pi}}\mathcal{F}^{-1} \left(
e^{-2t\eta^2}\frac{\eta\chi_1(\eta)}{\omega(\eta)}
\sin t\eta\omega(\eta)\right)\,,
\\ &
K_3(t,y)=\frac{1}{\sqrt{2\pi}}\mathcal{F}^{-1}\left(
e^{-2t\eta^2}\frac{\chi_1(\eta)\omega(\eta)}{\eta}
\sin t\eta\omega(\eta)\right)\,.  
\end{align*}
Then
\begin{equation}
  \label{eq:fundamental-low}
\chi_1(D_y)e^{t\mathcal{A}_*(D_y)}\delta=\begin{pmatrix}
 K_1(t,y)-K_2(t,y) & -8K_2(t,y)
\\ \frac{1}{8}(K_2(t,y)+K_3(t,y))
&  K_1(t,y)+K_2(t,y)
\end{pmatrix}\,.
\end{equation}
We have the following estimates for $K_1$, $K_2$ and $K_3$.
\begin{lemma} Let $j\in\Z_{\ge0}$. Then
  \label{lem:fundamental-sol}
  \begin{gather}
    \label{eq:k1-est}
\sup_{t>0}\|K_1(t,\cdot)\|_{L^1(\R)}<\infty\,,\quad \|K_1(t,\cdot)\|_{L^2(\R)}\lesssim \la t \ra^{-1/4}\,,
\\  \label{eq:k2-est1}
\|\pd_y^{j+1}K_1(t,\cdot)\|_{L^1(\R)} +\|\pd_y^jK_2(t,\cdot)\|_{L^1(\R)}+
\|\pd_y^{j+2}K_3(t,\cdot)\|_{L^1(\R)} 
\lesssim \la t \ra^{-(j+1)/2}\,,
\\
\label{eq:k2-est2}
\quad \|\pd_y^{j+1}K_1(t,\cdot)\|_{L^2(\R)}+\|\pd_y^jK_2(t,\cdot)\|_{L^2(\R)}
+\|\pd_y^{j+2}K_3(t,\cdot)\|_{L^2(\R)}\lesssim \la t \ra^{-(2j+3)/4}\,,
\\
\label{eq:k3-est'}
\sup_{t>0}\|\pd_yK_3(t,\cdot)\|_{L^1(\R)}<\infty\,,
\quad \|\pd_yK_3(t,\cdot)\|_{L^2(\R)}\lesssim \la t \ra^{-1/4}\,,
\\  \label{eq:k3-est}
\sup_{t>0}\|K_3(t,\cdot)*f\|_{L^\infty(\R)}\lesssim \|f\|_{L^1(\R)}\,.
\end{gather}
\end{lemma}
\begin{proof}
Let $\tilde{\omega}(\eta)=\omega(\eta)-4$ and
\begin{equation}
  \label{eq:def-Kpm}
K_{1,\pm}(t,y)=\frac{1}{2\sqrt{2\pi}}
\mF^{-1}\left(\chi_1(\eta)e^{-(2\eta^2\pm i\eta\tilde{\omega}(\eta))t}\right)\,.  
\end{equation}
Since $K_1(t,y)=\sum_\pm K_{1,\pm}(t,y\mp4t)$,
it suffices to show that
$\sup_{t>0}\|K_{1,\pm}(t,\cdot)\|_{L^1(\R)}<\infty$
and $\|K_{1,\pm}(t,\cdot)\|_{L^2(\R)}\lesssim \la t \ra^{-1/4}$ to prove
\eqref{eq:k1-est}. Using the Plancherel identity, we have
\begin{align*}
  \|K_{1,\pm}(t,\cdot)\|_{L^2(\R)}
\lesssim & \left\|\chi_1(\eta)e^{-2t\eta^2}\right\|_{L^2(\R)}\lesssim \la t \ra^{-1/4}\,,
\end{align*}
\begin{align*}
  \|yK_{1,\pm}(t,y)\|_{L^2(\R)} \lesssim &
\left\|\pd_\eta\left(\chi_1(\eta)e^{-(2\eta^2\pm i\tilde{\omega}(\eta))t}\right)\right\|_{L^2(\R)}
\\ \lesssim & 
\|\chi_1'(\eta)e^{-2t\eta^2}\|_{L^2(\R)}
+t(\|\eta\chi_1(\eta) e^{-2t\eta^2}\|_{L^2(\R)}
+\|\tilde{\omega}'(\eta)\chi_1(\eta) e^{-2t\eta^2}\|_{L^2(\R)})
\\ \lesssim & \la t \ra^{1/4}\,.
\end{align*}
Note that
\begin{equation}
    \label{eq:omega-approx}
 |\tilde{\omega}(\eta)|\lesssim \min\{1,\eta^2\}\,,
\quad |\tilde{\omega}'(\eta)|\lesssim \min\{1,|\eta|\}\,.
\end{equation}
Combining the above, we have
\begin{align*}
  \|K_{1,\pm}(t,\cdot)\|_{L^1(\R)} \lesssim & t^{1/4}\|K_{1,\pm}(t,\cdot)\|_{L^2(|y|\le \sqrt{t})}
+\|yK_{1,\pm}(t,y)\|_{L^2(|y|\ge\sqrt{t})}\|y^{-1}\|_{L^2(|y|\ge\sqrt{t})}
\\ =& O(1)\,.
\end{align*}
Thus we have \eqref{eq:k1-est}. We can prove
\eqref{eq:k2-est1}--\eqref{eq:k3-est'} in the same way.
\par

Now we will prove \eqref{eq:k3-est}.
Let
\begin{gather*}
K_{3,1}(t,y)=\frac{1}{2\sqrt{2\pi}}\mathcal{F}^{-1}\left(\omega(\eta)\chi_1(\eta)
e^{-2t\eta^2}\cos t\eta\tilde{\omega}(\eta)\right)\,,
\\
K_{3,2}(t,y)=\frac{1}{2\sqrt{2\pi}}\mathcal{F}^{-1}
\left(\omega(\eta)\chi_1(\eta)e^{-2t\eta^2}\frac{\sin t\eta\tilde{\omega}(\eta)}{\eta}\right)\,.
\end{gather*}
Then
\begin{equation}
  \label{eq:k3-formula}
K_3(t,y)=K_{3,1}(t,\cdot)\ast \mathbf{1}_{[-4t,4t]}+K_{3,2}(t,y+4t)+K_{3,2}(t,y-4t)\,.
\end{equation}
We can prove that
\begin{equation}
\label{eq:k31}
\sup_{t>0}\|K_{3,1}(t,\cdot)\|_{L^1(\R)}<\infty\,,
\end{equation}
and that $\|\pd_y^jK_{3,1}(t,\cdot)\|_{L^2(\R)}\lesssim \la t \ra^{-(2j+1)/4}$
for $j\ge0$ in the same way as \eqref{eq:k1-est}.
\par
Using the Schwarz inequality and the Plancherel theorem, we have
\begin{align*}
  \|K_{3,2}(t,\cdot)\|_{L^1(\R)} \lesssim & t^{1/4}
\|K_{3,2}(t,y)\|_{L^2(|y|\le \sqrt{t})}
+t^{-1/4}\|yK_{3,2}(t,y)\|_{L^2(|y|\ge \sqrt{t})}
\\ \lesssim & 
t^{1/4}\left\|\omega(\eta)\chi_1(\eta)e^{-2t\eta^2}
\frac{\sin t\eta\tilde{\omega}(\eta)}{\eta}\right\|_{L^2(\R)}
\\ & +t^{-1/4}\left\|\pd_\eta\left\{\omega(\eta)\chi_1(\eta)e^{-2t\eta^2}
\frac{\sin t\eta\tilde{\omega}(\eta)}{\eta}\right\}\right\|_{L^2(\R)}\,.
\end{align*}
Since $\left|\pd_j\left(\eta^{-1}\sin t\eta\tilde{\omega}(\eta)\right)\right|
\lesssim t\eta^{2-j}$ for $j=0$ and $1$,
it follows that
  \begin{align*}
  \|K_{3,2}(t,\cdot)\|_{L^1(\R)} \lesssim &
t^{1/4} \|t\eta^2 e^{-2t\eta^2}\|_{L^2(\R)}+t^{-1/4}\|t\eta e^{-2t\eta^2}\|_{L^2(\R)}
=O(1)\,.
\end{align*}
\par
Let $\chi_3(\eta)\in C_0^\infty(\R)$ such that
$\chi_1(\eta)=\chi_1(\eta)\chi_3(\eta)$. Then
$$K_{3,2}(t,\cdot)*f=K_{3,2}(t,\cdot)*\chi_3(D_y)f\,,
\quad \|\chi_3(D_y)f\|_{L^\infty(\R)} \lesssim  
\|\widecheck{\chi_3}\|_{L^\infty(\R)}\|f\|_{L^1(\R)}\,,$$
and
\begin{equation}
\label{eq:k32}
 \|K_{3,2}(t,\cdot\pm 4t)*f\|_{L^\infty(\R)}
\lesssim  \|f\|_{L^1(\R)}\,.
\end{equation}
Combining \eqref{eq:k3-formula}--\eqref{eq:k32}, we have \eqref{eq:k3-est}.
This completes the proof of Lemma~\ref{lem:fundamental-sol}.
\end{proof}

Let $Y$ and $Z$ be closed subspaces of $L^2(\R)$ defined by
$$Y=\mF^{-1}_\eta Z\quad\text{and}\quad
Z=\{f\in L^2(\R)\mid\supp f \subset[-\eta_0,\eta_0]\}\,,$$
and let $X_1=L^1(\R_y;L^2(R;e^{\a x}dx))$, $Y_1=\mF^{-1}_\eta Z_1$ and
$Z_1=\{f\in Z \mid\|f\|_{Z_1}:=\|f\|_{L^\infty}<\infty\}$.
\par

Let $E_1=\diag(1,0)$ and $E_2=\diag(0,1)$ and let $\chi(\eta)$ be a smooth
such that $\chi(\eta)=1$ if $\eta\in[-\frac{\eta_0}{4},\frac{\eta_0}{4}]$
and $\chi(\eta)=0$ if $\eta\not\in[-\frac{\eta_0}{2},\frac{\eta_0}{2}]$.
We will use the following estimates to investigate large time
behavior of modulation parameters.
\begin{lemma}
\label{lem:fund-sol}
For $t\ge0$ and $k\ge1$,
\begin{gather}
  \label{eq:decayE1}
  \|\chi_1(D_y)e^{t\mathcal{A}_*}E_1\|_{B(L^1;L^\infty)}=O(1)\,,\quad
\|(I-\chi(D_y))e^{t\mathcal{A}_*}E_1\|_{B(Y;L^\infty)}=O(e^{-c_1 t})\,,
\\  
\label{eq:etA-2}
\|e^{t\mathcal{A}_*}E_2\|_{B(Y;L^\infty)}\lesssim \la t\ra^{-1/4}\,,
\quad
\|e^{t\mathcal{A}_*}E_2\|_{B(Y_1;L^\infty)}\lesssim \la t\ra^{-1/2}\,,
\\ 
\label{eq:etA-3}
\|\pd_y^ke^{t\mathcal{A}_*}\|_{B(Y,L^\infty)}\lesssim \la t\ra^{-(2k-1)/4}\,,
\quad  \|\pd_y^ke^{t\mathcal{A}_*}\|_{B(Y_1;L^\infty)}\lesssim \la t\ra^{-k/2}\,,
\end{gather}
where $c_1$ is a positive constant. Moreover,
\begin{equation}
\label{eq:fund-asymp1}
\left\|e^{t\mathcal{A}_*}\begin{pmatrix}f_1 \\ f_2 \end{pmatrix}
-\frac12H_{2t}*W_{4t}*f_1\mathbf{e_2}
\right\|_{L^\infty}
\lesssim \la t \ra^{-1/2}(\|f_1\|_{Y_1}+\|f_2\|_{Y_1})\,,
\end{equation}
\begin{equation}
\label{eq:fund-asymp2}
\begin{split}
& \left\|\diag(1,\pd_y)e^{t\mathcal{A}_*}\begin{pmatrix}f_1 \\ f_2 \end{pmatrix}
-\frac14\begin{pmatrix}
  2 & 2 \\  1 & -1 \end{pmatrix}
\begin{pmatrix}
 H_{2t}(\cdot+4t)\\ H_{2t}(\cdot-4t)
\end{pmatrix}
*f_1\right\|_{L^\infty}
\\ \lesssim &
\la t \ra^{-1}(\|f_1\|_{Y_1}+\|f_2\|_{Y_1})\,,  
\end{split}
\end{equation}
\begin{equation}
\label{eq:fund-asymp3}
\begin{split}
& \left\|e^{t\mathcal{A}_*}\diag(1,\pd_y)\begin{pmatrix}f_1 \\ f_2 \end{pmatrix}
-\frac14\sum_{\pm}
 H_{2t}(\cdot\pm4t)(2f_2\pm f_1)\mathbf{e_2}\right\|_{L^\infty}
\\ \lesssim &
\la t \ra^{-1}(\|f_1\|_{Y_1}+\|f_2\|_{Y_1})\,,  
\end{split}
\end{equation}
where
$H_t(y)=(4\pi t)^{-1/2}\exp(-y^2/4t)$ and $W_t(y)=\frac12 \mathbf{1}_{[-t,t]}(y)$.
\end{lemma}
\begin{proof}
Equations~\eqref{eq:decayE1}--\eqref{eq:etA-3} follows immediately
from Lemma~\ref{lem:fundamental-sol}, \eqref{eq:k-est,high} and
\eqref{eq:fundamental-low}.
\par
In view of \eqref{eq:def-Kpm} and \eqref{eq:omega-approx},
\begin{align*}
\left\|2K_{1,\pm}(t,\cdot)*f-\chi_1(D)e^{2t\pd_y^2}f\right\|_{L^\infty}
\lesssim & \left\|\chi_1(\eta)e^{-2t\eta^2}(e^{\pm it\eta\tilde{\omega}(\eta)}-1)
\right\|_{L^1}\|\hat{f}\|_{L^\infty}
\\ \lesssim & \min\left \{t\|\eta^3e^{-2t\eta^2}\|_{L^1},\|\chi_1\|_{L^1}\right\}
\|f\|_{Y_1}
\\ \lesssim & \la t \ra^{-1}\|f\|_{Y_1}\,.
\end{align*}
Since $K_1(t,\cdot)=\frac12\sum_\pm K_{1,\pm}(t,\cdot\mp4t)$,
\begin{equation}
\label{eq:pf-fund-1}
\left\|K_1(t,\cdot)*f-
\frac12\sum_{\pm}H_{2t}(\cdot\pm 4t)*f\right\|_{L^\infty}
\lesssim \la t \ra^{-1}\|f\|_{Y_1}\,.
\end{equation}
We can prove 
\begin{gather}
\label{eq:pf-fund-2}
\left\|\pd_yK_3(t,\cdot)*f-2H_{2t}(\cdot+4t)*f+2H_{2t}(\cdot-4t)*f
\right\|_{L^\infty(\R)}
\lesssim \la t \ra^{-1}\|f\|_{Y_1}\,,
\\ \label{eq:pf-fund-3}
\|K_3(t,\cdot)*f-4H_{2t}*W_{4t}*f\|_{L^\infty(\R)}
\lesssim \la t \ra^{-1/2}\|f\|_{Y_1}\,.
  \end{gather}
in the same way.
 Combining \eqref{eq:pf-fund-1}--\eqref{eq:pf-fund-3}
with Lemma~\ref{lem:fundamental-sol} and   \eqref{eq:k-est,high},
we obtain \eqref{eq:fund-asymp1}--\eqref{eq:fund-asymp3}.
Thus we complete the proof.
\end{proof}

To investigate the large time behavior of $x(t,y)$, we need the following.
\begin{lemma}
  \label{cl:phase-lim}
Suppose that $f\in L^1(\R_+\times \R)$. Then for any $\delta>0$,
\begin{gather}
\label{eq:lim-in}
\lim_{t\to\infty}\sup_{|y|\le (4-\delta)t}
\left|\int_0^t H_{2(t-s)}*W_{4(t-s)}*f(s,\cdot)(y)\,ds
-\frac12\int_0^\infty\int_\R f(s,y)\,dyds\right|=0\,,
\\ \label{eq:lim-out}
\lim_{t\to\infty}\sup_{|y|\ge (4+\delta)t}
\left|\int_0^t H_{2(t-s)}*W_{4(t-s)}*f(s,\cdot)(y)\,ds\right|=0\,.  
\end{gather}
\end{lemma}
\begin{proof}
Let $K_t(y,y_1)=\left\{(s,y_2)\mid 0\le s< t\,,\;
\left|y_2-y+2y_1(t-s)^{1/2}\right|\le4(t-s)\right\}$.
Then
\begin{align*}
& \int_0^t H_{2(t-s)}*W_{4(t-s)}*f(s,\cdot)(y)\,ds
\\=&
\frac{1}{4\sqrt{2\pi}}\int_0^t ds\, (t-s)^{-1/2}
\int_\R dy_1\,e^{-(y-y_1)^2/8(t-s)} \int_{y_1-4(t-s)}^{y_1+4(t-s)}dy_2\,f(s,y_2)
\\ =& 
\frac{1}{2\sqrt{2\pi}}\int_\R dy_1\,e^{-y_1^2/2}
\int_{K_t(y,y_1)}dsdy_2\,f(s,y_2)\,.
\end{align*}
Since $e^{-y_1^2/2}f(s,y_2)$ is integrable on $\R_+\times \R^2$,
\begin{align*}
& \left|\int_{|y_1|\ge \delta\sqrt{t}/4} dy_1\,e^{-y_1^2/2}
\int_{K_t(y,y_1)}dsdy_2\,f(s,y_2)\right|
\\ \le  &
\|f\|_{L^1(\R_+\times \R)}\int_{|y_1|\ge\delta\sqrt{t}/4} dy_1\,e^{-y_1^2/2}
\to0\quad\text{as $t\to\infty$.}  
\end{align*}
Moreover,
\begin{align*}
& \lim_{t\to\infty}\sup_{|y|\le (4-\delta)t}
\left|\frac{1}{\sqrt{2\pi}}\int_{|y_1|\le\delta\sqrt{t}/4}
e^{-y_1^2/2}\left(\int_{K_t(y,y_1)}f(s,y_2)\,dy_2ds\right)\,dy_1
-\int_{\R_+\times\R} f(s,y)\,dyds\right|=0\,,
\\ &
\lim_{t\to\infty}\sup_{|y|\ge (4+\delta)t}
\left|\int_{|y_1|\le\delta\sqrt{t}/4}e^{-y_1^2/2}
\left(\int_{K_t(y,y_1)}f(s,y_2)\,dy_2ds\right)\,dy_1\right|=0
\end{align*}
because
\begin{align*}
& \lim_{t\to\infty}\bigcap_{|y|\le (4-\delta)t\,,\,|y_1|\le \delta\sqrt{t}/4}
K_t(y,y_1)=\R_+\times \R\,,
\quad
\lim_{t\to\infty}\bigcup_{|y|\ge (4+\delta)t\,,\,|y_1|\le \delta\sqrt{t}/4}
K_t(y,y_1)=\emptyset\,.
\end{align*}
\end{proof}
\bigskip

\section{Decomposition of solutions around $1$-line solitons}
Following \cite{Miz15,Miz17}, we decompose a solution around
a line soliton $\varphi(x-4t)$ into a sum of a modulating
line soliton and a dispersive part plus a small wave
which is caused by amplitude changes of the line soliton:
\begin{equation}
  \label{eq:decomp}
u(t,x,y)=\varphi_{c(t,y)}(z)-\psi_{c(t,y),L}(z+3t)+v(t,z,y)\,,\quad
z=x-x(t,y)\,,
\end{equation}
where
$\psi_{c,L}(x)=2(\sqrt{2c}-2)\psi(x+L)$,
$\psi(x)$ is a nonnegative function such that
$\psi(x)=0$ if  $|x|\ge1$ and that $\int_\R \psi(x)\,dx=1$
and $L>0$ is a large constant to be fixed later.
The modulation parameters $c(t_0,y_0)$ and $x(t_0,y_0)$ denote
the maximum height and the phase shift of the modulating line soliton
$\varphi_{c(t,y)}(x-x(t,y))$ along the line $y=y_0$ at the time $t=t_0$,
and $\psi_{c,L}$ is an auxiliary smooth function such that
\begin{equation}
  \label{eq:0mean}
\int_\R \psi_{c,L}(x)\,dx=\int_\R(\varphi_c(x)-\varphi(x))\,dx\,.
\end{equation}
\par
Now we further decompose $v$ into a small solution of
\eqref{eq:KPII} and an exponentially localized part
as in \cite{Miz1,MT2,Miz17}.
If $v_0(x,y)$ is polynomially localized, then as in \cite{Miz17},
we can decompose the initial data as a sum of an amplified line soliton and
a remainder part $v_*(x,y)$ that satisfies $\int_\R v_*(x,y)\,dx=0$ for every
$y\in\R$.
Let
\begin{gather}
  \label{eq:defc1}
  c_1(y)=\left\{\sqrt{c_0}+\frac{1}{2\sqrt{2}}\int_\R v_0(x,y)\,dx\right\}^2\,,
  \\ \label{eq:defv*}
  v_*(x,y)=v_0(x,y)+\varphi_{c_0}(x)-\varphi_{c_1(y)}(x)\,.
\end{gather}
Then we have the following. 
\begin{lemma}
  \label{lem:nonzeromean1}
Let $c_0>0$ and $s>1$.
There exists a positive constant $\eps_0$ such that if
$\eps:=\|\la x\ra^{s/2}(\la x\ra+\la y\ra)^{s/2}v_0\|_{H^1(\R^2)}<\eps_0$, then
\begin{align}
\label{eq:nonzero1b} 
&  \left\|\la y\ra^{s/2}(c_1-c_0)\right\|_{L^2(\R)}+\left\|\la y\ra^{s/2}\pd_yc_1\right\|_{L^2(\R)}
 \lesssim
\left\|\la x\ra^{s/2}\la y\ra^{s/2}v_0\right\|_{H^1(\R^2)}\,,
\\ \label{eq:nonzero1c}
& \left\|\la x\ra^s v_*\right\|_{L^2(\R^2)} \lesssim 
\left\|\la x\ra^sv_0\right\|_{L^2(\R^2)}\,,
\quad
\left\|\la x\ra^{s/2}\la y\ra^{s/2}v_*\right\|_{L^2(\R^2)} \lesssim 
\left\|\la x\ra^{s/2}\la y\ra^{s/2}v_0\right\|_{L^2(\R^2)}\,,
\\  \label{eq:nonzero1d}
&  \|\pd_x^{-1}v_*\|_{L^2}+ \|\pd_x^{-1}\pd_yv_*\|_{L^2}+\|v_*\|_{H^1(\R^2)} \lesssim
\|\la x\ra^sv_0\|_{H^1(\R^2)}\,.
\end{align}
Moreover, the mapping
$$ \la x\ra^{-s/2}(\la x\ra+\la y\ra)^{-s/2}H^1(\R^2)\ni v_0\mapsto
(v_*,c_1-c_0)\in H^1(\R^2)\times H^1(\R)\cap \la y\ra^{-s/2}L^2(\R)$$
is continuous.
\end{lemma}
\begin{proof} 
By \cite[(10.4)]{Miz17},
$$\sup_y\left|\sqrt{c_1(y)}-\sqrt{c_0}\right|
\lesssim  \|\la x\ra^{s/2}v_0\|_{L^2}+ \|\la x\ra^{s/2}\pd_yv_0\|_{L^2}\,.$$
Hence it follows from \eqref{eq:defc1} and \eqref{eq:defv*} that
for sufficiently small $\eps$,
\begin{align*}
  \left\|\la y\ra^{s/2}\pd_y^i(c_1-c_0)\right\|_{L^2(\R)}
\lesssim &   \left\|\la y\ra^{s/2}
\pd_y^i\left(\sqrt{c_1}-\sqrt{c_0}\right)\right\|_{L^2(\R)}
\\ \lesssim &
\left\| \la y\ra^{s/2}\int_\R \pd_y^iv_0(x,y)\,dx\right\|_{L^2(\R_y)}
\lesssim \left\| \la x\ra^{s/2}\la y\ra^{s/2} \pd_y^iv_0\right\|_{L^2(\R^2)}\,,
\end{align*}
\begin{align*}
\left\|\la x\ra^{s/2}\la y\ra^{s/2}v_*\right\|_{L^2(\R^2)}
\lesssim &
\left\|\la x\ra^{s/2}\la y\ra^{s/2}v_0\right\|_{L^2(\R^2)}
+\left\|\la x\ra^{s/2}\la y\ra^{s/2}(\varphi_{c_1(y)}-\varphi_{c_0}\right\|_{L^2(\R^2)}
\\ \lesssim & \left\|\la x\ra^{s/2}\la y\ra^{s/2}v_0\right\|_{L^2(\R^2)}\,.
\end{align*}
Using \cite[(10.2)]{Miz17}, we can prove 
$\left\|\la x\ra^sv_*\right\|_{L^2(\R^2)} 
\lesssim \left\|\la x\ra^sv_0\right\|_{L^2(\R^2)}$ in the same way.
We have \eqref{eq:nonzero1d} from \cite[Lemma~10.1]{Miz17} and its proof.
Since the continuity of the mapping $v_0\mapsto (v_*,c_1-c_0)$ can
be proved in the similar way, we omit the proof.
Thus we complete the proof.
\end{proof}

Let $\tv_1$ be a solution of
\begin{equation}
\label{eq:tv1}
\left\{\begin{aligned}
& \pd_t\tv_1+\pd_x^3\tv_1+3\pd_x(\tv_1^2)+3\pd_x^{-1}\pd_y^2\tv_1=0\,,\\
& \tv_1(0,x,y)=v_*(x,y)\,.    
  \end{aligned}\right.
\end{equation}
Since $v_*\in H^1(\R^2)$ and $\pd_x^{-1}\pd_yv_*\in L^2(\R^2)$,
we have $\tv_1(t)\in C(\R;H^1(\R^2))$ from \cite{MST}.
Applying the Strichartz estimate in \cite[Proposition~2.3]{Saut93} to
$$ \pd_x^{-1}\pd_y\tv_1(t)=e^{tS}\pd_x^{-1}\pd_yv_*-6\int_0^t
e^{(t-s)S}(\tv_1\pd_y\tv_1)(s)\,ds\,,
\quad S=-\pd_x^3-3\pd_x^{-1}\pd_y^2\,,
$$
we have $\pd_x^{-1}\pd_y\tv\in C(\R;L^2(\R))$.
Suppose that $v_0$ satisfies the assumption of Lemma~\ref{lem:nonzeromean1}
and that $u(t)$ is a solution to \eqref{eq:KPII} with
$u(0,x,y)=\varphi(x)+v_0(x,y)$, where $\varphi=\varphi_2$.
Then as \cite[Lemmas~3.1 and 3.3]{Miz17}, we can prove that
$w(t,x,y):=u(t,x+4t,y)-\varphi(x)-\tv_1(t,x+4t,y)$ belongs to an exponentially
weighted space $X=L^2(\R^2;e^{2ax}dxdy)$ for an $\a\in(0,1)$, that
$$w\in C([0,\infty);X)\,,\quad 
\pd_xw\,,\, \pd_x^{-1}\pd_yw\in L^2(0,T;X)\,,$$
and that the mapping
$$ \la x\ra^{-1}(\la x\ra+\la y\ra)^{-1}H^1(\R^2)\ni v_0
\mapsto w\in C([0,T];X)$$
is continuous for any $T>0$ by using by Lemma~\ref{lem:nonzeromean1}.
\par

Now let
\begin{equation}
  \label{eq:decomp2}
v_1(t,z,y)=\tv_1(t,z+x(t,y),y)\,,\quad v_2(t,z,y)=v(t,z,y)-v_1(t,z,y)\,.
\end{equation}
To fix the decomposition \eqref{eq:decomp},
we impose the constraint that for $k=1$, $2$,
\begin{equation}
  \label{eq:orth}
 \int_{\R^2}
v_2(t,z,y)g_k^*(z,\eta,c(t,y))e^{-iy\eta}\,dzdy=0
\quad\text{in $L^2(-\eta_0,\eta_0)$,}
\end{equation}
where 
$g_1^*(z,\eta,c)=cg_1^*(\sqrt{c/2}z,\eta)$ and
$g_2^*(z,\eta,c)=\frac{c}{2}g_2^*(\sqrt{c/2}z,\eta)$.
\par
Let
\begin{align*}
F_k[u,\tc,\gamma,L](\eta):=\mathbf{1}_{[-\eta_0,\eta_0]}&(\eta)
\int_{\R^2} \bigl\{u(x,y)+\varphi(x)- \varphi_{c(y)}(x-\gamma(y))
\\ & +\psi_{c(y),L}(x-\gamma(y))\bigr\}g_k^*(x-\gamma(y),\eta,c(y))
e^{-iy\eta}\,dxdy
\end{align*}
for $k=1$, $2$, where $c(y)=2+\tc(y)$.
Since $w(0)=\varphi_{c_1}-\varphi$ and
\begin{align*}
\|\varphi_{c_1}-\varphi\|_{X_1}\lesssim &
\left\|\la y\ra(c_1-2)\right\|_{L^2(\R_y)}
 \lesssim  \|\la x\ra\la y\ra v_0\|_{L^2(\R^2)}
\end{align*}
by Lemma~\ref{lem:nonzeromean1}, it follows from
\cite[Lemmas~5.2 and 5.4]{Miz15} that there exists
$(\tc_*,x_*)\in Y_1\times Y_1$ satisfying
$$F_1[w(0),\tc_*,x_*,L]=F_2[w(0),\tc_*,x_*,L]=0\,,$$
\begin{equation}
\label{eq:cx*bound}
\begin{split}
& \|\tc_*\|_Y+\|x_*\|_Y\lesssim \|w(0)\|_X\lesssim \|\la x\ra v_0\|_{L^2(\R^2)}\,,  
\\ & 
\|\tc_*\|_{Y_1}+\|x_*\|_{Y_1}\lesssim \|w(0)\|_{X_1}\lesssim
\|\la x\ra\la y\ra v_0\|_{L^2(\R^2)}\,,  
\end{split}
\end{equation}
provided $\|\la x\ra\la y\ra v_0\|_{L^2(\R^2)}$ is sufficiently small.
By the definitions,
\begin{equation}
  \label{eq:v2-init}
\left\{
  \begin{aligned}
& v_2(0,x,y)=v_{2,*}(x,y):=\varphi_{c_1(y)}(x)-\varphi_{c_*(y)}(x-x_*(y))
+\tpsi_{c_*(y),L}(x-x_*(y))\,,
\\ & \tc(0,y)=\tc_*(y)\,,\quad x(0,y)=x_*(y)\,,
  \end{aligned}
\right.
\end{equation}
where $c_*=2+\tc_*$
and it follows from Lemma~\ref{lem:nonzeromean1} and \eqref{eq:cx*bound}
that 
\begin{equation}
 \label{eq:v2*-bound}
 \|v_{2,*}\|_X\lesssim \|c_1-2\|_{L^2(\R)}+ \|\tc_*\|_Y
\lesssim \|\la x\ra v_0\|_{L^2(\R^2)}\,.
\end{equation}
\par

Lemma~5.2 in \cite{Miz15} implies that there exist
a $T>0$, $\tc(t,\cdot):=c(t,\cdot)-2\in C([0,T;Y)$ and $\tx(t,\cdot):=x(t,\cdot)-4t\in C([0,T];Y)$
satisfying
$$ F_1[w(t),\tc(t),\tx(t),L]=F_2[w(t),\tc(t),\tx(t),L]=0
\quad\text{for $t\in[0,T]$}$$
because $w\in C([0,\infty);X)$.
If $(v_2(t),\tc(t))$ remains small in $X\times Y$ for $t\in[0,T]$,
then the decomposition \eqref{eq:decomp}
and \eqref{eq:decomp2} satisfying \eqref{eq:orth} exists beyond $t=T$
thanks to the continuation argument.

\begin{proposition}\emph{(\cite[Proposition~3.9]{Miz17})}
\label{prop:continuation}
There exists a $\delta_1>0$ such that if \eqref{eq:decomp},
\eqref{eq:decomp2} and \eqref{eq:orth} hold for $t\in[0,T)$ and
\begin{gather*}
(\tc,\tx)\in C([0,T);Y\times Y)\cap C^1((0,T);Y\times Y)\,,
\\
\sup_{t\in[0,T]}(\|v_2(t)\|_X+\|\tc(t)\|_Y)<\delta_1\,,
\quad \sup_{t\in[0,T)}\|\tx(t)\|_Y<\infty\,,
\end{gather*}
then either $T=\infty$ or $T$ is not the maximal time of
the decomposition \eqref{eq:decomp} and \eqref{eq:decomp2} satisfying \eqref{eq:orth}.
\end{proposition}
\bigskip

\section{Modulation equations}
By \cite[Lemma~3.6 and Remark~3.7]{Miz17},
$$v_2(t)\in C([0,T);X)\,,\quad 
(\tc,\tx)\in C([0,T);Y\times Y)\cap C^1((0,T);Y\times Y)\,,
$$
where $T$ is the maximal time of the decomposition
\eqref{eq:decomp} and \eqref{eq:decomp2} satisfying \eqref{eq:orth}.
Substituting \eqref{eq:decomp2} with $z=x-x(t,y)$ into \eqref{eq:tv1}
and  \eqref{eq:decomp} and \eqref{eq:decomp2} into \eqref{eq:KPII},
we have
\begin{equation}
  \label{eq:v1}
\pd_tv_1-2c\pd_zv_1+\pd_z^3v_1+3\pd_{z}^{-1}\pd_y^2v_1
=\pd_z(N_{1,1}+N_{1,2})+N_{1,3}\,,
\end{equation}
where $N_{1,1}=-3v_1^2$, $N_{1,2}=\{x_t-2c-3(x_y)^2\}v_1$
and $N_{1,3}=6\pd_y(x_yv_1)-3x_{yy}v_1$, and
\begin{equation}
  \label{eq:v2}
\left\{
  \begin{aligned}
& \pd_tv_2=\mL_{c}v_2+\ell+\pd_z(N_{2,1}+N_{2,2}+N_{2,4})+N_{2,3}\,,
\\ &
v_2(0)=v_{2,*}\,,
  \end{aligned}
\right.
\end{equation}
where  
$\mL_cv=-\pd_z(\pd_z^2-2c+6\varphi_c)v-3\pd_z^{-1}\pd_y^2$,
$\ell=\sum_{k=1}^2\ell_k$, $\ell_k=\sum_{j=1}^3\ell_{kj}$ $(k=1$, $2)$,
$\tpsi_c(z)=\psi_{c,L}(z+3t)$ and
\begin{align*}
\ell_{11}=&(x_t-2c-3(x_y)^2)\varphi'_c-(c_t-6c_yx_y)\pd_c\varphi_c\,,\quad
\ell_{12}=3x_{yy}\varphi_c\,,\\
\ell_{13}=& 3c_{yy}\int_z^\infty\pd_c\varphi_c(z_1)\,dz_1
+3(c_y)^2\int_z^\infty \pd_c^2\varphi_{c}(z_1)\,dz_1\,,\\
\ell_{21}=& (c_t-6c_yx_y)\pd_c\tpsi_c-(x_t-4-3(x_y)^2)\tpsi_c'\,,\\
\ell_{22}=&(\pd_z^3-\pd_z)\tpsi_c-3\pd_z(\tpsi_c^2)
+6\pd_z(\varphi_c\tpsi_c)-3x_{yy}\tpsi_c\,,\\
\ell_{23}=&-3c_{yy}\int_z^\infty\pd_c\tpsi_c(z_1)\,dz_1
-3(c_y)^2\int_z^\infty \pd_c^2\tpsi_c(z_1)\,dz_1\,,
\end{align*}
\begin{gather*}
N_{2,1}=-3(2v_1v_2+v_2^2)\,,\quad  N_{2,2}=\{x_t-2c-3(x_y)^2\}v_2+6\tpsi_cv_2\,,
\\ N_{2,3}=6\pd_y(x_yv_2)-3x_{yy}v_2\,,\quad N_{2,4}=6(\tpsi_c-\varphi_c)v_1\,.
\end{gather*}
\par
Differentiating \eqref{eq:orth} with respect to $t$ and substituting
\eqref{eq:v2} into the resulting equation, we have
 in $L^2(-\eta_0,\eta_0)$
\begin{equation}
\label{eq:orth_t}
\begin{split}
& \frac{d}{dt}\int_{\R^2}v_2(t,z,y)g_k^*(z,\eta,c(t,y))
e^{-i y\eta}\,dzdy
\\=& \int_{\R^2} \ell g_k^*(z,\eta,c(t,y))e^{-i y\eta}\,dzdy
+\sum_{j=1}^6II^j_k(t,\eta)=0\,,
\end{split}  
\end{equation}
where
\begin{align*}
II^1_k=& \int_{\R^2} v_2(t,z,y)\mL_{c(t,y)}^*(g_k^*(z,\eta,c(t,y))
e^{-iy\eta})\, dzdy\,,\\
II^2_k=& -\int_{\R^2} N_{2,1}\pd_zg_k^*(z,\eta,c(t,y))
e^{-i y\eta}\, dzdy\,,\\
II^3_k=& \int_{\R^2} N_{2,3}g_k^*(z,\eta,c(t,y))e^{-iy\eta}\,dzdy
\\ & +6\int_{\R^2} v_2(t,z,y)c_y(t,y)x_y(t,y)
\pd_cg_k^*(z,\eta,c(t,y))e^{-i y\eta}\, dzdy\,,\\
II^4_k=& \int_{\R^2} v_2(t,z,y)\left(c_t-6c_yx_y\right)(t,y)
\pd_cg_k^*(z,\eta,c(t,y))e^{-i y\eta}\, dz y\,,\\
II^5_k=& -\int_{\R^2} N_{2,2}\pd_zg_k^*(z,\eta,c(t,y))e^{-i y\eta}\,dzdy\,,
\quad
II^6_k= -\int_{\R^2} N_{2,4}\pd_zg_k^*(z,\eta,c(t,y))e^{-i y\eta}\,dzdy\,.
\end{align*}
Using the fact that $g_1^*(z,\eta,c)\simeq\varphi_c(z)$ and
$g_2^*(z,\eta,c)\simeq(c/2)^{3/2}\int_{-\infty}^z\pd_c\varphi_c$ for $\eta\simeq0$,
we derive the modulation equations for $c(t,y)$ and $x(t,y)$
(see \cite[Section~4]{Miz17}). 
\par
To write down the modulation equation, let us introduce several notations.
Let $R^j$, $\wR^j$, $\wS_j$, $\bS_j$, $\widetilde{\mathcal{A}}_1(t)$, 
$B_j$ and $\wC_j$  be the same as those in \cite[pp.~165--168]{Miz17}
except for the definitions of $R^4$ and $R^5$.
We move a part of  $R^4$ into $R^5$.
See \eqref{eq:defR4} and \eqref{eq:defR5} in Appendix~\ref{ap:r}.

Note that 
\begin{equation}
\label{eq:bdef}
b(t,\cdot):=\frac{1}{3}\wP_1\left\{\sqrt{2}c(t,\cdot)^{3/2}-4\right\}
=\tc(t,\cdot)+O(\tc^2)\,,
\end{equation}
where $\wP_1 f=\mF_\eta^{-1}\mathbf{1}_{[-\eta_0,\eta_0]}\mF_yf$ and
$\mathbf{1}_{[-\eta_0,\eta_0]}$ is a characteristic function of $[-\eta_0,\eta_0]$.
We make use of \eqref{eq:bdef} to translate the nonlinear term $c_yx_y$
into a divergence form.
\par

Now let us introduce localized norms of $v(t)$.
Let $p_\a(z)=1+\tanh\a z$ and $\|v\|_{W(t)}=\|p_\a(z+3t+L)^{1/2}v\|_{L^2(\R^2)}$.
Assuming the smallness of the following quantities,
we can derive modulation equations of $b(t,y)$ and $x(t,y)$ for $t\in[0,T]$.
Let $0\le T\le\infty$ and 
\begin{align*}
& \bM_{c,x}(T)=\sum_{k=0,\,1}\sup_{t\in[0,T]}\bigl\{\la t\ra^{(2k+1)1/4}
(\|\pd_y^k\tc(t)\|_Y+\|\pd_y^{k+1}x(t)\|_Y)
%  \\ & \phantom{\bM_{c,x}(T)=\sum}
       +\la t\ra(\|\pd_y^2\tc(t)\|_Y+\|\pd_y^3x(t)\|_Y)\bigr\}\,,
\\ & \bM_v(T)=\sup_{t\in[0,T]}\|v(t)\|_{L^2}\,,       
\\ & \bM_1(T)=\sup_{t\in[0,T]}\{\la t\ra^2\|v_1(t)\|_{W(t)}
 +\la t\ra\|(1+z_+)v_1(t)\|_{W(t)}\}
+\|\mathcal{E}(v_1)^{1/2}\|_{L^2(0,T;W(t))}\,,
\\ & 
\bM_1'(\infty)=\sup_{t\ge0}\|\mathcal{E}(\tv_1(t))^{1/2}\|_{L^2(\R^2)}\,, \quad
\bM_2(T)=\sup_{0\le t\le T}\la t \ra^{3/4}\|v_2(t)\|_X
+\|\mathcal{E}(v_2)^{1/2}\|_{L^2(0,T:X)}\,,
\end{align*}
where $\mathcal{E}(v)=(\pd_xv)^2+(\pd_x^{-1}\pd_yv)^2+v^2$.
We remark that by an anisotropic Sobolev inequality (see e.g. \cite{Besov}),
\begin{equation}
  \label{eq:Sobolev}
\|v\|_{L^2(\R^2)}+\|v\|_{L^6(\R^2)}\lesssim \|\mathcal{E}(v)^{1/2}\|_{L^2(\R^2)}\,.  
\end{equation}
We can prove the following result exactly in the same way as
\cite[Proposition~3.9]{Miz17}).
\begin{proposition}
\label{prop:modulation}
There exists a $\delta_2>0$ such that if 
$\bM_{c,x}(T)+\bM_2(T)+\eta_0+e^{-\a L}<\delta_2$ for a $T\ge0$, then  
\begin{gather}
  \label{eq:modeq}
\begin{pmatrix}b_t \\ \tx_t\end{pmatrix}=\mathcal{A}(t)
\begin{pmatrix}b \\ \tx\end{pmatrix}\,+\sum_{i=1}^6 \cN^i\,,\\
  \label{eq:modeq-init}
b(0,\cdot)=b_*\,,\quad x(0,\cdot)=x_*\,,
\end{gather}
where $b_*=4/3\wP_1\{(c_*/2)^{3/2}-1\}$,
$\mathcal{A}(t)=
\mathcal{A}_*+B_4^{-1}\widetilde{\mathcal{A}}_1(t)
+\pd_y^4\mathcal{A}_1(t)+\pd_y^2\mathcal{A}_2(t)$,
\begin{align*}
& \mathcal{A}_1(t)=-B_4^{-1}(\wS_1B_1^{-1}B_2+\wS_0)\,,\quad
\mathcal{A}_2(t)=B_4^{-1}\wS^3B_1^{-1}B_2\,,
\\  
\cN^1=& \wP_1\begin{pmatrix}
6(b\tx_y)_y\\ 2(\tilde{c}-b)+3(\tx_y)^2
\end{pmatrix}\,,\quad \cN^2=\cN^{2a}+\cN^{2b}\,,\\
 \cN^{2a}=& B_3^{-1}\left(
\wP_1R^7_1\mathbf{e_1}+\wR^1+\wR^3\right)\,,
\quad  \cN^{2b}=B_3^{-1}\wP_1 R^7_2\mathbf{e_2}\,,
\quad \mathbf{e}_1=\begin{pmatrix}1 \\ 0\end{pmatrix}\,,
\quad \mathbf{e}_2=\begin{pmatrix}0 \\ 1 \end{pmatrix}\,,
\\ \cN^3=& B_3^{-1}\pd_y(\wR^2+\wR^4)\,,\quad
\cN^4=(B_3^{-1}-B_4^{-1})(B_2-\pd_y^2\wS_0)\pd_y
\begin{pmatrix} b_y \\ x_y\end{pmatrix}\,,
\\ \cN^5=& (B_3^{-1}-B_4^{-1})\widetilde{\mathcal{A}}_1(t)
\begin{pmatrix} b \\ \tx \end{pmatrix}\,,
\quad \cN^6= B_3^{-1}R^{v_1}\,.
\end{align*}
\end{proposition}
We remark that $\cN^6$ equals to $\cN^5$ in \cite{Miz17}
and that $\mathcal{A}(t)+\cN^5$ equals to $\mathcal{A}(t)$ in \cite{Miz17}.
The other terms are exactly the same.
\par

To apply \eqref{eq:decayE1} and \eqref{eq:fund-asymp1} to \eqref{eq:modeq}
in Section~\ref{sec:phase},  we need to decompose $b_*$ into a sum
of an integrable function and a function that belongs to $\pd_y^2Y_1$.
Note that $\wP_1L^1(\R)\subset Y_1\subset Y$ and $Y\subset \cap_{k\ge0}H^k(\R)$.
\begin{lemma}
  \label{lem:modeq-init-decomp}
There exist $\obu{b}\in L^1(\R)$ and $\oc{b}\in Y_1$ such that
$b_*=\wP_1\obu{b}+\pd_y^2\oc{b}$ and
$$\bigl\|\obu{b}\bigr\|_{L^1(\R)}+\bigl\|\oc{b}\bigr\|_{Y_1}
\lesssim \left\|\la x\ra\la y\ra v_0\right\|_{L^2(\R^2)}\,.$$
\end{lemma}
\begin{proof}
Since $b_*-\tc_*=\frac{4}{3}\wP_1\{(c_*/2)^{3/2}-1-\tc_*\}$ and
$\|(c_*/2)^{3/2}-1-c_*\|_{L^1(\R)} \lesssim
\|\tc_*\|_Y^2  \lesssim  \|\la x\ra v_0\|_{L^2(\R^2)}^2$,
it suffices to show that there exist $\obu{c}\in L^1(\R)$ and
$\oc{c}\in Y_1$ such that $\tc_*=\obu{c}+\pd_y^2\oc{c}$ and
$$\bigl\|\obu{c}\bigr\|_{L^1(\R)}+\bigl\|\oc{c}\bigr\|_{Y_1}
\lesssim \left\|\la x\ra\la y\ra v_0 \right\|_{L^2(\R^2)}\,.$$
\par
Let
\begin{gather*}
F_{10}[u,\tc,\gamma,L](\eta):=\frac{1}{2}\mathbf{1}_{[-\eta_0,\eta_0]}(\eta)
\int_{\R^2} \left\{u(x,y)+\Phi[\tc,\gamma](x,y)\right\}
\varphi_{c(y)}(x-\gamma(y))e^{-iy\eta}\,dxdy\,,
\\
F_{11}[u,\tc,\gamma,L](\eta):=\mathbf{1}_{[-\eta_0,\eta_0]}(\eta)
\int_{\R^2} \left\{u(x,y)+\Phi[\tc,\gamma](x,y)\right\}
g_{k1}^*(x-\gamma(y),\eta,c(y))e^{-iy\eta}\,dxdy\,
\end{gather*}
where $c(y)=2+\tc(y)$ and
$\Phi[\tc,\gamma](x,y)=\varphi(x)-\varphi_{c(y)}(x-\gamma(y))+\psi_{c(y),L}(x-\gamma(y))$. Then
$$
F_1[u,\tc,\gamma,L](\eta)=F_{10}[u,\tc,\gamma,L](\eta)
+\eta^2F_{11}[u,\tc,\gamma,L](\eta)\,,$$
and we can prove 
$$\left\|F_{11}[u,\tc,\gamma,L]\right\|_{Z_1}
\lesssim \|u\|_{X_1}+\|\tc\|_{Y_1}+\|\gamma\|_{Y_1}
+\|u\|_X(\|\tc\|_{Y_1}+\|\gamma\|_{Y_1})$$
in exactly the same way as the proof of \cite[Lemma~5.1]{Miz15}.
\par
Let $w_0(x,y)=\varphi_{c_1(y)}(x)-\varphi(x)$.
Since $F_1[w_0,\tc_*,x_*,L]=0$ and $(w_0,\tc_*,x_*)\in X_1\times Y_1\times Y_1$,
\begin{align}
    & F_{10}[w_0,\tc_*,x_*,L](\eta)=-\eta^2F_{11}[w_0,\tc_*,x_*,L](\eta)\,,
\\  & F_{11}[w_0,\tc_*,x_*,L]\in Z_1\,.
  \end{align}
\par
Let
\begin{align*}
& \Phi_0(x,y)=-\tc_*(y)\{\pd_c\varphi(x)-\psi(x+L)\}+x_*(y)\varphi'(x)\,,
\\ &
\Psi_1(x,y)=\frac12\left\{\varphi_{c_*(y)}(x-x_*(y))-\varphi(x)\right\}\,.
\end{align*}
Then $\mF_\eta^{-1}F_{10}[w_0,\tc_*,x_*,L](y)
=(2\pi)^{1/2}\wP_1(J_0+J_1+J_2+J_3$),
where
\begin{align*}
J_0= & \frac12\int_\R \Phi_0(x,y)\varphi(x)\,dx=
\left(-1+\frac12\int_\R \varphi(x)\psi(x+L)\,dx\right)\tc_*\,,
\\
J_1=& \frac12\int_\R w_0(x,y)\varphi_{c_*(y)}(x-x_*(y))\,dx\,,
\quad
J_2= \int_\R \Phi[\tc_*,x_*](x,y)\Psi_1(x,y)\,dx\,,
\\
J_3=& \frac12\int_\R \left\{\Phi[\tc_*,x_*](x,y)-\Phi_0(x,y)\right\}\varphi(x)\,dx\,,
\end{align*}
and
\begin{gather*}
\|J_1\|_{L^1(\R)} \lesssim  \|w_0\|_{L^1(\R^2)}
 \lesssim  \left\|\la y\ra(c_1-2)\right\|_{L^2(\R)}\,,
\\
\|J_2\|_{L^1(\R)}\lesssim  \|\Phi\|_{L^2(\R^2)}\|\Psi_1\|_{L^2(\R^2)}
\lesssim  (\|\tc_*\|_Y+\|x_*\|_Y)^2\,,
\\
\|J_3\|_{L^1(\R)}\lesssim \|\Phi-\Phi_0\|_{L^1(\R^2)}
\lesssim (\|\tc_*\|_Y+\|x_*\|_Y)^2\,.
\end{gather*}
Combining the above with Lemma~\ref{lem:nonzeromean1} and
\eqref{eq:cx*bound}, we obtain Lemma~\ref{lem:modeq-init-decomp}.
\end{proof}

\bigskip

\section{\`A priori estimates for modulation parameters.}
In this section, we will estimate  $\bM_{c,x}(T)$
assuming smallness of $\bM_1(T)$, $\bM_2(T)$, $\bM_v(T)$, $\eta_0$
and $e^{-\a L}$.
\begin{lemma}
  \label{lem:Mcx-bound}
There exist positive constants $\delta_3$ and $C$ such that if
$\bM_{c,x}(T)+\bM_1(T)+\bM_2(T)+\eta_0+e^{-\a L}\le \delta_3$, then
\begin{equation}
  \label{eq:Mcx-bound}
\bM_{c,x}(T)\le C\| \la x\ra(\la x\ra+\la y\ra) v_0\|_{L^2}
+C(\bM_1(T)+\bM_2(T)^2)\,.
\end{equation}
\end{lemma}
To prove Lemma~\ref{lem:Mcx-bound}, we need the following.
\begin{claim}
\label{cl:cx_t-bound}
Let $\delta_2$ be as in proposition~\ref{prop:modulation}. Suppose
$\bM_{c,x}(T)+\bM_1(T)+\bM_2(T)+\eta_0+e^{-\a L}<\delta_2$ for a $T\ge0$.
Then for $t\in[0,T]$,
$$\|c_t\|_Y+\|x_t-2c-3(x_y)^2\|_Y\lesssim 
(\bM_{c,x}(T)+\bM_1(T)+\bM_2(T)^2)\la t\ra^{-3/4}\,.$$
\end{claim}
\begin{proof}
In view of \eqref{eq:modeq},
  \begin{align*}
& \|c_t(t)\|_Y+\|x_t-2c-3(x_y)^2\|_Y \lesssim I+\sum_{2\le i\le 6}\|\cN^i\|_Y\,,
\\ & I=\|b_t-c_t\|_Y+\|\widetilde{\mathcal{A}}_1(t)(b,\tx)\|_Y
+\|b_{yy}\|_Y+\|x_{yy}\|_Y+\|(b\tx_y)_y\|_Y+\|(I-\wP_1)x_y^2\|_Y\,,
  \end{align*}
and it follows from Claim~\ref{cl:akbound}, \cite[Claim~D.6]{Miz15}
and the definition of $\bM_{c,x}(T)$ that
\begin{equation*}
I\lesssim \bM_{c,x}(T)\la t\ra^{-3/4}\quad\text{for $t\in[0,T]$.}  
\end{equation*}
See the proof of Lemma~5.2 in \cite{Miz17}. 
Following the line of \cite[Chapter~7]{Miz15} and using \eqref{eq:N2a-v1},
we can prove that for $t\in[0,T]$,
\begin{equation*}
\sum_{i=2}^5 \|\cN_i(t)\|_Y\lesssim
(\bM_{c,x}(T)+\bM_1(T)+\bM_2(t))^2\la t\ra^{-1} \,.
\end{equation*}
Combining the above with \eqref{eq:N6-est}, we have Claim~\ref{cl:cx_t-bound}.
\end{proof}
\par

To deal with $E_1\cN^6$,
we  decompose $\chi(D_y)B_3^{-1}$ and $\chi(D_y)B_4^{-1}$
into a sum of operators that belong to $B(L^1(\R))$ and operators
that belong to $\pd_y^2B(Y_1)$.
Since
\begin{align*}
& B_3=B_1+\wC_1+\pd_y^2(\bS_1+\bS_2)-\bS_3-\bS_4-\bS_5\,,
\\ & 
B_4=B_1+\pd_y^2\wS_1-\wS_3=B_3|_{\tc=0\,,\,v_2=0}\,,
\end{align*}
we have
\begin{gather}
\label{eq:B4B3expr}
B_4^{-1}=\obu{B}_4-\pd_y^2\oc{B}_{14}\,,\quad
B_3^{-1}-B_4^{-1}=\obu{B}_{34}-\pd_y^2\oc{B}_{34}\,,
\\
\obu{B}_4=(B_1-\wS_{31})^{-1}\,,
\quad \oc{B}_{14}=B_4^{-1}(\wS_1+\wS_{32})\obu{B}_4\,,
\\ \label{eq:B34expr}
\obu{B}_{34}=-\obu{B}_4(\wC_1+\wS_{31}-\bS_{31}-\bS_{41}-\bS_{51})B_3^{-1}\,,
\\ 
\oc{B}_{34}=\oc{B}_{14}(B_4-B_3)B_3^{-1}+\obu{B}_4\left(
\bS_1-\wS_1+\bS_2+\bS_{32}-\wS_{32}+\bS_{42}+\bS_{52}\right)\,
\end{gather}
where $\wS_j$ and $\bS_j$ are the same as those in \cite[p.~167]{Miz17} and
$\wS_{j1}$, $\wS_{j2}$, $\bS_{j1}$ and $\bS_{j2}$ are defined by
\eqref{eq:defS3k1}--\eqref{eq:defS6k2} and
\eqref{eq:defwS3l}--\eqref{eq:defwS5l} in Appendix~\ref{ap:s}.
We remark that $\wS_j=\wS_{j1}-\pd_y^2\wS_{j2}$, that
$\bS_j=\bS_{j1}-\pd_y^2\bS_{j2}$  and that $\wS_{j1}$ is a
time-dependent constant multiple of $\wP_1$.
\begin{claim}
  \label{cl:b4part1}
Let $0\le T\le \infty$.
There exist positive constants $\eta_0$, $L_0$ and $C$ such that if
$|\eta|\le\eta_0$ and $L\ge L_0$, then for every $t\in[0,T]$,
\begin{gather}
  \label{eq:b4part1}
 \bigl\|\obu{B}_4\bigr\|_{B(Y)\cap B(Y_1)}
+\bigl\|\oc{B}_{14}\bigr\|_{B(Y)\cap B(Y_1)}\le C\,,
\\ \label{eq:b4part2}
\|\chi(D_y)\obu{B}_4\|_{B(L^1)}\le C\,,
\\ \label{eq:b4part2'}
\bigl\|\obu{B}_4-B_1^{-1}\bigr\|_{B(Y_1)}
+\bigl\|\chi(D_y)(\obu{B}_4-B_1^{-1})\bigr\|_{B(L^1)}
\le Ce^{-\a(3t+L)}\,.
\end{gather}
Moreover, if $\bM_{c,x}(T)+\bM_2(T)\le \delta$ is sufficiently small,
then there exists a positive constant $C_1$ such that for $t\ge0$,
  \begin{align}
\label{eq:b4part3}
& \bigl\|\obu{B}_{34}\bigr\|_{B(Y,Y_1)}+\bigl\|\oc{B}_{34}\bigr\|_{B(Y,Y_1)}
  \le C_1(\bM_{c,x}(T)+\bM_2(T))\la t\ra^{-1/4}\,,
\\ &
\label{eq:b4part4}
 \bigl\|\chi(D_y)\obu{B}_{34}\bigr\|_{B(Y,L^1)} \le
 C_1(\bM_{c,x}(T)+\bM_2(T))\la t\ra^{-1/4}\,,
\\ &
\label{eq:b4part5}
 \bigl\|[\pd_y,\obu{B}_{34}]\bigr\|_{B(Y,Y_1)}
+ \bigl\|\chi(D_y)[\pd_y,\obu{B}_{34}]\bigr\|_{B(Y,L^1)} \le
 C_1(\bM_{c,x}(T)+\bM_2(T))\la t\ra^{-3/4}\,.
  \end{align}
\end{claim}
\begin{proof}
By \cite[Claim~6.3]{Miz15} and \cite[Claim~B.1]{Miz15},
$$
\sup_{t\ge0}\|B_4^{-1}\|_{B(Y)\cap B(Y_1)}=O(1)\,,\quad
\|\wS_1\|_{B(Y)\cap B(Y_1)}=O(1)\,.$$
Combining the above with \eqref{eq:clwS3-1} and \eqref{eq:clwS3-2},
we have \eqref{eq:b4part1}.
\par
Next, we will prove \eqref{eq:b4part2}.
Since $\chi(\eta)\chi_1(\eta)=\chi(\eta)$ and
$[\chi_1(D_y),\wS_{31}]=0$,
\begin{equation}
\label{eq:obu-exp}
\chi(D_y)\obu{B}_4
=\sum_{n\ge0}\chi(D_y)(B_1^{-1}\chi_1(D_y)\wS_{31})^nB^{-1}\,.  
\end{equation}
Hence it follows from \eqref{eq:clwS3-1} that
$$  \|\chi(D_y)\obu{B}_4\|_{B(L^1)} \lesssim 
\|\check{\chi}\|_{L^1}\sum_{n\ge0}\|B_1^{-1}\chi_1(D_y)\wS_{31}\|_{B(L^1(\R))}^n
=O(1)\,.$$
We can prove \eqref{eq:b4part2'} in the same way.
\par
By \cite[Claims~6.1 and 6.2]{Miz15}, we have
\begin{equation*}
\sup_{t\in[0,T]}\|B_3^{-1}\|_{B(Y)}<\infty\quad\text{and}\quad
\|\wC_1\|_{B(Y,Y_1)}\lesssim \bM_{c,x}(T)\la t\ra^{-1/4}
\quad\text{for $t\in[0,T]$.}
\end{equation*}
It follows from Claims~6.1, B.1--B.3 in \cite{Miz15} that
\begin{equation*}
\|B_3-B_4\|_{B(Y,Y_1)}+\|\bS_1-\wS_1\|_{B(Y,Y_1)}
\lesssim (\bM_{c,x}(T)+\bM_2(T))\la t\ra^{-1/4}
\quad\text{ for $t\in[0,T]$.}
\end{equation*}
Combining the above with \eqref{eq:b4part1} and Claim~\ref{cl:wS3},
we have \eqref{eq:b4part3}.
\par
Since
$$\chi(D_y)\obu{B}_{34}=-\chi(D_y)\obu{B}_4
\chi_1(D_y)\left(\wC_1+\wS_{31}-\sum_{3\le j\le 5}
\bS_{j1}\right)B_3^{-1}\,,
$$
we have \eqref{eq:b4part4} from Claim~\ref{cl:wS3}.
Using the fact that
$\|[\pd_y,\wC_1]\|_{B(Y,Y_1)}+\|\chi_1(D_y)[\pd_y,\wC_1]\|_{B(Y,L^2)}\lesssim \bM_{c,x}(T)\la t\ra^{-3/4}$
(see \cite[Claim~B.7]{Miz15}), we can prove
\eqref{eq:b4part5} in the same way as \eqref{eq:b4part3} and
\eqref{eq:b4part4}.
This completes the proof of Claim~\ref{cl:b4part1}.
\end{proof}

Now we are in position to prove Lemma~\ref{lem:Mcx-bound}.
\begin{proof}[Proof of Lemma~\ref{lem:Mcx-bound}]
Let $A(t)=\diag(1,\pd_y)\mathcal{A}(t)\diag(1,\pd_y^{-1})$ and $U(t,s)$ be
the semigroup generated by $A(t)$. Lemma~4.2 in \cite{Miz15} implies that 
there exists a $C=C(\eta_0,k)>0$ such that
\begin{gather}
\label{eq:BB1}
\|\pd_y^kU(t,s)f\|_Y\le C\la t-s\ra^{-k/2}\|f\|_Y
\quad\text{for $t\ge s\ge0$,}
\\ \label{eq:BB2}
\|\pd_y^kU(t,s)f\|_Y\le C\la t-s\ra^{-(2k+1)/4}\|f\|_{Y_1}
\quad\text{for $t\ge s\ge0$,}
\end{gather}
provided $\eta_0$ is sufficiently small.
\par
Multiplying \eqref{eq:modeq} by $\diag(1,\pd_y)$ from the left,
we have
\begin{equation}
  \label{eq:modeq2}
\left\{
  \begin{aligned}
& \pd_t\begin{pmatrix}b \\ x_y\end{pmatrix}=A(t)
\begin{pmatrix}b \\ x_y\end{pmatrix}\,+\sum_{i=1}^6\diag(1,\pd_y)\cN^i\,,
\\ & b(0,\cdot)=b_*\,,\quad \pd_yx(0,\cdot)=\pd_yx_*\,.
  \end{aligned}\right.
\end{equation}
Since $\|\cN^6\|_{Y_1}$ does not necessarily decay as $t\to\infty$,
we make use of the change of variable
\begin{gather}
\label{eq:def-k}
k(t,y)=\frac{1}{2}\wP_1
\left(\int_{\R} v_1(t,z,y)\varphi_{c(t,y)}(z)\,dz\right)\,,
\quad
S^3_{11}[\psi](t)=\frac12\int_{\R^2}\psi(z+3t+L)\varphi(z)\,dz\,,
\\
\tb(t,y)=b(t,y)+\tk(t,y)\,,\quad \tk(t,y)=(2-S^3_{11}(t))^{-1}k(t,\cdot)
\,. \notag
\end{gather}
Then
\begin{equation}
\label{eq:modeq2'}
\left\{
  \begin{aligned}
& \pd_t\begin{pmatrix}\tb \\ x_y\end{pmatrix}=A(t)
\begin{pmatrix}\tb \\ x_y\end{pmatrix}\,+\sum_{i=1}^6\diag(1,\pd_y)\cN^i
+\pd_t\tk(t)\mathbf{e_1}+A(t)\tk(t)\mathbf{e_1}\,,
\\ & \tb(0,\cdot)=b_*+\tk(0,\cdot)\,,
\quad \pd_yx(0,\cdot)=\pd_yx_*\,.
  \end{aligned}\right.  
\end{equation}
\par
By \eqref{eq:cx*bound} and \eqref{eq:BB2},
$$\left\|\pd_y^kU(t,0)
\begin{pmatrix}\tb(0,\cdot)\\ x_y(0,\cdot)\end{pmatrix}\right\|_Y
\lesssim
\la t \ra^{-(2k+1)/4}(\|\tb(0,\cdot)\|_{Y_1}+\|\pd_yx_*\|_{Y_1})\,,$$
and 
\begin{align*}
 \|\tb(0)\|_{Y_1}+\|\pd_yx_*\|_{Y_1}
\lesssim & \|b_*\|_{Y_1}+\|\la y\ra v_*\|_{L^2}+\eta_0\|x_*\|_{Y_1}
+\|\varphi_{c_*(y)}(x-x_*(y))v_*\|_{L^1(\R^2)}
\\ \lesssim &
\|\la x\ra(\la x\ra+\la y\ra) v_0\|_{L^2}\,.
\end{align*}
\par
Except for $\cN^6$, the term which includes $v_1$ in \eqref{eq:modeq2}
and needs to be treated differently
from \cite[(6.16)]{Miz15} is $\cN^{2a}$ because $\wR^1$ includes $R^4$
and $R^4$ includes the inverse Fourier transform of
$\mathbf{1}_{[-\eta_0,\eta_0]}(\eta)II^2_k(t,\eta)$. 
But thanks to the smallness of $\bM_1(T)$, 
\begin{equation}
  \label{eq:N2a-v1}
\begin{split}
\|II^2_k\|_{L^\infty[-\eta_0,\eta_0]} =&
\sup_{\eta\in[-\eta_0,\eta_0]}\left|\int_{\R^2} N_{2,1}
\pd_zg_k^*(z,\eta,c(t,y))e^{-iy\eta}\,dzdy\right|
\\ \lesssim & (\|p_\a(z)v_1\|_{L^2}\|v_2\|_X+\|v_2\|_X^2)
\sup _{c\in[2-\delta,2+\delta]\,,\eta\in[-\eta_0,\eta_0]}
\|e^{-2\a z}g_k^*(z,\eta,c)\|_{L^\infty_z}
\\ \lesssim & (\bM_1(T)+\bM_2(T))^2\la t \ra^{-3/2}\,,
\end{split}  
\end{equation}
and $\|R^4\|_{Y_1}$ decays at the rate as in \cite[Claim~D.5]{Miz15}.
Using \eqref{eq:BB1}, \eqref{eq:BB2} and \eqref{eq:N2a-v1},
we can prove that for $t\in[0,T]$ and $k=0$, $1$, $2$,
\begin{align*}
& \sum_{i=1}^5\int_0^t \|\pd_y^kU(t,s)\diag(1,\pd_y)\cN^j(s)\|_Y\,ds
 \lesssim  (\bM_{c,x}(T)+\bM_1(T)+\bM_2(T))^2\la t\ra^{-\min\{1,(2k+1)/4\}}\,.
\end{align*}
in the same way as \cite[Chapter~7]{Miz15}.
\par

Since $\diag(1,\pd_y)\cN^6=E_1\cN^6+\pd_yE_2\cN^6$,
it follows from \eqref{eq:BB1}
\begin{equation}
\label{eq:fNk}
\fN^k:=
\left\|\int_0^t \|\pd_y^kU(t,s)\pd_yE_2\cN^6(s)\|_Y\,ds\right\|_Y
\lesssim \int_0^t \la t-s\ra^{-(k+1)/2}\|\cN^6(s)\|_Y\,ds\,.
\end{equation}
Using the fact that
\begin{align*}
& \sup_{\eta\in[-\eta_0,\eta_0]}
\left(|\varphi_{c(t,y)}(z)\pd_zg_k^*(z,\eta,c(t,y))
-\varphi(z)\pd_zg_k^*(z,\eta)|
+|\tpsi_{c(t,y)}(z)\pd_zg_k^*(z,\eta,c(t,y))|\right)
  \\ & \lesssim   e^{-2\a|z|}|\tc(t,y)|\,,
\end{align*}
we see that
\begin{equation}
  \label{eq:Rv1-est}
\begin{split}
& \|R^{v_1}(t)\|_Y
\\ \lesssim  & 
\sum_{k=1,2}\left\|\int_{\R^2}
\varphi(z)v_1(t,z,y)\pd_zg_k^*(z,\eta)e^{-iy\eta}
\,dzdy\right\|_{L^2(-\eta_0,\eta_0)}
+\|\tc\|_Y\|e^{-\a|z|}v_1(t)\|_{L^2(\R^2)}
\\ \lesssim &  \|e^{-\a|z|}v_1(t)\|_{L^2(\R^2)}
\lesssim \la t \ra^{-2}\bM_1(T)
\quad\text{$t\in[0,T]$,}
\end{split}  
\end{equation}
and that
\begin{gather}
  \label{eq:N6-est}
\|\cN^6(t)\|_Y\lesssim  \la t \ra^{-2}\bM_1(T)
\quad\text{for $t\in[0,T]$,}
\end{gather}
follows from the boundedness of $B_3^{-1}$ (see \cite[Claim~4.5]{Miz17}).
Combining \eqref{eq:fNk} and \eqref{eq:N6-est}, we have
$\fN^k\lesssim \bM_1(T)\la t\ra^{-(k+1)/2}$ for $t\in[0,T]$ and $k=0$, $1$, $2$.
\par
Now we investigate
\begin{equation}
  \label{eq:cN6}
E_1\cN^6=E_1B_3^{-1}R^{v_1}=E_1\{\obu{B}_4+\obu{B}_{34}
-\pd_y^2(\oc{B}_{14}+\oc{B}_{34})\}R^{v_1}
\end{equation}
more precisely. We remark that
$\|\cN^6\|_{Y_1}$ cannot be expected to decay like $\|\cN^6\|_Y$ as $t\to\infty$
because
$$\left\|\int_\R\varphi(z)v_1(t,z,y)\,dz\right\|_{L^1(\R_y)}$$
does not necessarily decay as $t\to\infty$.
By \eqref{eq:Rv1-est} and Claim~\ref{cl:b4part1},
\begin{gather}
\label{eq:B4Rv1}
\bigl\|\obu{B}_4R^{v_1}\bigr\|_Y+\bigl\|(\oc{B}_{14}+\oc{B}_{34})R^{v_1}\bigr\|_Y
 \lesssim \bM_1(T)\la t\ra^{-2}\,,
\\
  \label{eq:B34Rv1}
\bigl\|\chi(D_y)\obu{B}_{34}R^{v_1}\bigr\|_{L^1}
+\bigl\|\obu{B}_{34}R^{v_1}\bigr\|_{Y_1}\lesssim 
\bM_1(T)(\bM_{c,x}(T)+\bM_2(T))\la t\ra^{-9/4}\,.
\end{gather}
\par
In view of \eqref{eq:BB1}, \eqref{eq:BB2}
and \eqref{eq:cN6}--\eqref{eq:B34Rv1}, we see that for $k=0$, $1$, $2$ and $t\in[0,T]$,
\begin{align*}
& \left\|\int_0^t \pd_y^kU(t,s)
E_1\{\obu{B}_{34}-\pd_y^2(\oc{B}_{14}+\oc{B}_{34})\}R^{v_1}\,ds\right\|_Y
\\ \lesssim & 
\int_0^t \la t-s\ra^{-(2k+1)/4}\bigl\|\obu{B}_{34}R^{v_1}\bigr\|_{Y_1}\,ds
+ \int_0^t \la t-s\ra^{-(k+2)/2}\bigl\|(\oc{B}_{14}+\oc{B}_{34})\}R^{v_1}\bigr\|_Y\,ds
\\ \lesssim & \bM_1(T)\la t\ra^{-(2k+1)/4}\,,
\end{align*}
and that $E_1\obu{B}_4R^{v_1}$ is the hazardous part of $E_1\cN^6$.
\par
The worst part of $E_1\obu{B}_4R^{v_1}$ can be expressed as
a time derivative of a decaying function as in \cite{Miz17}.
The operator $B_1-\wS_{31}$ and its inverse $\obu{B}_4$ are lower triangular
on $Y\times Y$ and
\begin{equation}
\label{eq:bad-part1}
E_1\obu{B}_4R^{v_1}=(2-S^3_{11}[\psi])^{-1}R^{v_1}_1\mathbf{e_1}\,.
\end{equation}
In view of \cite[pp.175--176]{Miz17},
\begin{equation}
  \label{eq:Rv1'}
R^{v_1}_1=S^7_1[\pd_c\varphi_c](c_t)-S^7_1[\varphi_c'](x_t-2c-3(x_y)^2)
-k_t+\obu{R}_{v_1}+\pd_y\oc{R}_{v_1}\,,
\end{equation}
where 
\begin{gather*}
S^7_1[q_c](f)(t,y)
=\frac{1}{2}\wP_1\left(\int_{\R} v_1(t,z,y)f(y)q_{c(t,y)}(z)\,dz\right)\,,
\end{gather*}
and $\obu{R}_{v_1}$ and $\oc{R}_{v_1}$ are chosen such that
$\obu{R}_{v_1}+\pd_y\oc{R}_{v_1}=R^{v_1}_{11}+\pd_yR^{v_1}_{12}$
and that
$$\chi(D_y)\obu{R}_{v_1}\in L^1(0,\infty;L^1(\R))\,.$$
We give the precise definitions of $\obu{R}_{v_1}$ and $\oc{R}_{v_1}$ later.
\par
The term $\pd_t\tk\mathbf{e_1}$ in \eqref{eq:modeq2'}
cancels out with a bad part of $E_1\obu{B}_4R^{v_1}$ which comes
from $-k_t$ in \eqref{eq:Rv1'}.
In fact,
$$\pd_t\tk(t,y)-(2-S^3_{11}[\psi](t))^{-1}\pd_tk(t,y)
=(2-S^3_{11}[\psi](t))^{-2}\pd_tS^3_{11}[\psi](t)k(t,y)\,,$$
and by the definition,
\begin{equation}
  \label{eq:S311-d}
|S^3_{11}[\psi](t)|+\left|\pd_tS^3_{11}[\psi](t)\right|
\lesssim e^{-2(3t+L)}\quad\text{for $t\ge0$.}
\end{equation}
Combining Claim~\ref{cl:k-growth} and \eqref{eq:S311-d} with \eqref{eq:BB2},
we have for $t\in[0,T]$ and $k\ge0$,
\begin{align*}
& \left\|\int_0^t \pd_y^kU(t,s)\left\{
  \pd_t\tk(s,\cdot)-(2-S^3_{11}[\psi](s))^{-1}\pd_tk(s,\cdot)\right\}
  \mathbf{e_1}\,ds\right\|_Y
\\ \lesssim & \bM_1(T)\int_0^t \la t-s\ra^{-(2k+1)/4}\la s\ra e^{-2(3s+L)}\,ds
\lesssim \bM_1(T)\la t\ra^{-(2k+1)/4}\,.
\end{align*}

Next, we will investigate $\obu{R}_{v_1}$ and $\oc{R}_{v_1}$.
We write $II^6_{13}$ in \cite[p.175]{Miz17} as
$II^6_{13}=II^6_{131}+\eta^2II^6_{132}$,
\begin{align*}
& II^6_{131}=3\int_{\R^2} v_1(t,z,y)\tpsi_{c(t,y)}(z)\varphi_{c(t,y)}(z)e^{-i y\eta}\,dzdy\,,
\\ &
II^6_{132}=6\int_{\R^2} v_1(t,z,y)\tpsi_{c(t,y)}(z)
\pd_zg_{11}^*(z,\eta,c(t,y))e^{-i y\eta}\,dzdy\,,
\\ &
g_{k1}^*(z,\eta,c)=\frac{g_k^*(z,\eta,c)-g_k^*(z,0,c)}{\eta^2}\,,
\end{align*}
because $\pd_zg_1^*(z,0,c(t,y))=\frac12\varphi_{c(t,y)}(z)$ and let
\begin{align*}
\obu{R}_{v_1}=& \frac{1}{2\pi}\int_{-\eta_0}^{\eta_0}
\left\{II^6_{111}(t,\eta)-II^6_{131}(t,\eta)\right\}e^{iy\eta}\,d\eta
=R^{v_1}_{11}
-\frac{\pd_y^2}{2\pi}\int_{-\eta_0}^{\eta_0} II^6_{132}(t,\eta)e^{iy\eta}\,d\eta\,,
\\ 
\oc{R}_{v_1} =
& \frac{1}{2\pi}\int_{-\eta_0}^{\eta_0}\left\{II^6_{112}(t,\eta)
-i\eta II^6_{12}(t,\eta)+i\eta II^6_{132}(t,\eta)\right\}e^{i y\eta}\,d\eta\,.
\end{align*}
Then
\begin{multline*}
\obu{R}_{v_1}=
 \frac32\wP_1\int_\R v_1(t,z,y)^2\varphi_{c(t,y)}'(z)\,dz
 -3\wP_1\int_\R v_1(t,z,y)\tpsi_{c(t,y)}(z)\varphi_{c(t,y)}'(z)\,dz
 \\
 +\frac32\wP_1\left[\int_\R v_1(t,z,y)\left\{
     c_{yy}(t,y)\int_{-\infty}^z \pd_c\varphi_{c(t,y)}(z_1)\,dz_1+
   c_y(t,y)^2 \int_{-\infty}^z \pd_c^2\varphi_{c(t,y)}(z_1)\,dz_1   \right\}\,dz\right]
\\
 -\frac32\wP_1\int_\R v_1(t,z,y)\left\{x_{yy}(t,y)\varphi_{c(t,y)}(z)
+2(c_yx_y)(t,y)\pd_c\varphi_{c(t,y)}(z)\right\}\,dz\,,
\end{multline*}
$\oc{R}_{v_1}=\oc{R}_{v_1,1}+\pd_y\oc{R}_{v_1,2}$ and
\begin{multline*}
\oc{R}_{v_1,1}=
-\frac{3}{2}\wP_1\int_\R v_1(t,z,y) c_y(t,y)
\left(\int_{-\infty}^z \pd_c\varphi_{c(t,y)}(z_1)\,dz_1\right)\,dz
\\ +3\wP_1\int_\R v_1(t,z,y)x_y(t,y)\varphi_{c(t,y)}(z)\,dz\,,
\end{multline*}
\begin{multline*}
\oc{R}_{v_1,2}=
\frac32 \wP_1\int_\R v_1(t,z,y)
\left(\int_{-\infty}^z \varphi_{c(t,y)}(z_1)\,dz_1\right)\,dz
\\ 
-\frac{1}{2\pi}
\int_{-\eta_0}^{\eta_0}\left\{II^6_{12}(t,\eta)-II^6_{132}(t,\eta)\right\}
e^{iy\eta}\,d\eta\,,
\end{multline*}
$$II^6_{12}(t,\eta)=6\int_{\R^2} v_1(t,z,y)\varphi_{c(t,y)}(z)
\pd_zg_{11}^*(z,\eta,c(t,y))e^{-i y\eta}\,dzdy\,,$$
and we have
\begin{gather}
\label{eq:obuRv1}
\bigl\|\obu{R}_{v_1}\bigr\|_{Y_1}+\bigl\|\chi(D_y)\obu{R}_{v_1}\bigr\|_{L^1(\R)}
\lesssim \bM_1(T)(\bM_{c,x}(T)+\bM_1(T))\la t\ra^{-3/2}\,,
\\
\label{eq:ocRv1}
\bigl\|\oc{R}_{v_1,1}\bigr\|_{Y_1} \lesssim \bM_{c,x}(T)\bM_1(T)
\la t\ra^{-7/4}\,,
\quad
\bigl\|\oc{R}_{v_1,2}\bigr\|_Y \lesssim \bM_1(T)\la t\ra^{-1}\,. 
\end{gather}
Combining the above with \eqref{eq:BB1} and \eqref{eq:BB2}, we have
\begin{align*}
& \left\|\int_0^t \pd_y^kU(t,s) \bigl(2-S^3_{11}[\psi](s)\bigr)^{-1}
\bigl(\obu{R}_{v_1}+\pd_y\oc{R}_{v_1}\bigr)\,ds\right\|_Y
 \lesssim 
\int_0^t \la t-s\ra^{-(2k+1)/4}\bigl\|\obu{R}_{v_1}\bigr\|_{Y_1}\,ds
\\ & +
\int_0^t \la t-s\ra^{-(2k+3)/2}\bigl\|\oc{R}_{v_1,,1}\bigr\|_{Y_1}\,ds
+\int_0^t \la t-s\ra^{-(k+4)/2}\bigl\|\oc{R}_{v_1,2}\bigr\|_Y\,ds
\\ \lesssim & \bM_1(T)\la t\ra^{-\min\{1,(2k+1)/4\}}\quad
\text{for $k=0$, $1$, $2$ and $t\in[0,T]$.}
\end{align*}
\par
Next, we will estimate $S^7_1[\pd_c\varphi_c](c_t)$ and
$S^7_1[\varphi_c'](x_t-2c-3(x_y)^2)$. 
By Claim~\ref{cl:cx_t-bound},
\begin{equation}
  \label{eq:S7-bound}
  \begin{split}
& \|S^7_1[\pd_c\varphi_c](c_t)\|_{Y_1}
+\|\chi(D_y)S^7_1[\pd_c\varphi_c](c_t)\|_{L^1}
\\ & +\|S^7_1[\varphi_c'](x_t-2c-3(x_y)^2)\|_{Y_1}
+\|\chi(D_y)S^7_1[\varphi_c'](x_t-2c-3(x_y)^2)\|_{L^1}
\\ \lesssim & 
\bM_1(T)
(\bM_{c,x}(T)+\bM_1(T)+\bM_2(T)^2)\la t\ra^{-11/4}\,.
\end{split}
\end{equation}

\par

Finally, we will estimate $A(t)\tk\mathbf{e_1}$.
Since $\mathcal{A}_1(t)E_2=O$ and $[\mathcal{A}_i(t),\pd_y]=O$ for $i=1$, $2$,
\begin{align*}
&  A(t)=A_*+\diag(\pd_y^3,\pd_y^4)\mathcal{A}_1(t)\diag(\pd_y,1)
+\diag(\pd_y,\pd_y^2)\mathcal{A}_2(t)\diag(\pd_y,1)+A_3(t)\,,
\\ & 
A_*=\begin{pmatrix} 3\pd_y^2 & 8\pd_y \\ (2-\mu_3\pd_y^2)\pd_y & \pd_y^2
\end{pmatrix}\,,
\quad
A_3(t)=\diag(1,\pd_y)B_4^{-1}\widetilde{\mathcal{A}}_1(t)E_1\,,
\\ &
\sup_{t\ge0}\|\pd_y^{-1}(A(t)-A_3(t))\|_{B(Y)}<\infty\,,
\quad \|A_3(t)\|_{B(Y_1)}\lesssim e^{-\a(3t+L)}\,.
\end{align*}
Combining the above with \eqref{eq:BB1}, \eqref{eq:BB2} and 
Claims~\ref{cl:k-decay} and \ref{cl:k-growth},
we have for $k=0$, $1$, $2$,
\begin{align*}
& \left\|\int_0^t \pd_y^k  U(t,s)A(s)\tk(s)\,ds\right\|_Y
 \le \int_0^t \|\pd_y^{k+1}U(t,s)\|_{B(Y)}
\left\|\pd_y^{-1}(A(s)-A_3(s))\right\|_{B(Y)}\|\tk(s)\|_Y\,ds
\\  & \phantom{\left\|\int_0^t \pd_y^k  U(t,s)A(s)\tk(s)\,ds\right\|_Y
 \le \int_0^t \|\pd_y^{k+1}U}
+\int_0^t \|\pd_y^kU(t,s)\|_{B(Y_1,Y)}\|A_3(s)\|_{B(Y_1)}\|\tk(s)\|_{L^1}\,ds
 \\ \lesssim &
\bM_1(T)\left\{\int_0^t \la t-s\ra^{-(k+1)/2}\la s\ra^{-2}\,ds
+\int_0^t \la t-s\ra^{-(2k+1)/4}\la s\ra e^{-\a(3s+L)}\,ds\right\}
\\ \lesssim &
\bM_1(T)\la t\ra^{-(2k+1)/4}\,.
\end{align*}
This completes the proof of Lemma~\ref{lem:Mcx-bound}.
\end{proof}
\bigskip

\section{The $L^2(\R^2)$ estimate}
\label{sec:L2norm}
In this section, we will estimate $\bM_v(T)$
assuming smallness of $\bM_{c,x}(T)$, $\bM_1(T)$ and $\bM_2(T)$.

\begin{lemma}
  \label{lem:M4-bound}
Let $\delta_3$ be the same as in Lemma~\ref{lem:Mcx-bound}.
Suppose that $\bM_{c,x}(T)+\bM_1(T)+\bM_2(T)+\eta_0+e^{-\a L}\le \delta_3$.
Then there exists a positive constant $C$ such that 
$$\bM_v(T)\le  C(\|v_0\|_{L^2(\R^2)}+\bM_{c,x}(T)+\bM_1(T)+\bM_2(T))\,.$$
\end{lemma}

To prove Lemma~\ref{lem:M4-bound}, we use a variant of the
$L^2$ conservation law on $v$ as in \cite{Miz15,Miz17}. 
\begin{lemma}
  \label{lem:L2conserve}\emph{(\cite[Lemma~6.2]{Miz17})}
Let $0\le T\le\infty$. Let $\tv_1$ be a solution of \eqref{eq:tv1} and 
$v_2$ be a solution of \eqref{eq:v2}.
Suppose that $(v_2(t),c(t),\gamma(t))$ satisfy \eqref{eq:decomp},
\eqref{eq:decomp2} and \eqref{eq:orth}.
Then
$$Q(t,v):=\int_{\R^2}\left\{v(t,z,y)^2
-2\psi_{c(t,y),L}(z+3t)v(t,z,y)\right\}\,dzdy$$
satisfies for $t\in[0,T]$,
\begin{equation}
  \label{eq:L2-conserve}
\begin{split}
Q(t,v)=& Q(0,v)+
2\int_0^t\int_{\R^2}\left(\ell_{11}+\ell_{12}+6\varphi_{c(s,y)}'(z)\tpsi_{c(s,y)}(z)
\right)v(s,z,y)\,dzdyds
\\ & -2\int_0^t\int_{\R^2} \ell\psi_{c(t,y),L}(z+3s)\,dzdyds
-6\int_0^t\int_{\R^2}\varphi_{c(s,y)}'(z)v(s,z,y)^2\,dzdyds
\\ &
-6\int_0^t\int_{\R^2}(\pd_z^{-1}\pd_yv)(s,z,y)c_y(s,y)
\pd_c\varphi_{c(t,y)}(z)\,dzdy\,.
\end{split}
\end{equation}
\end{lemma}
\begin{proof}[Proof of Lemma~\ref{lem:M4-bound}]
We can estimate the right hand side of \eqref{eq:L2-conserve}
in exactly the same way as in the proof of \cite[Lemma~8.1]{Miz15}
except for the last term.
By the definition, we have for $t\in[0,T]$,
\begin{align*}
  \left|\int_{\R^2}(\pd_z^{-1}\pd_yv)(s,z,y)
c_y(s,y)\pd_c\varphi_{c(t,y)}(z)\,dzdy\right|
\lesssim & \|e^{-\a|z|}\pd_z^{-1}\pd_yv\|_{L^2(0,T;L^2(\R^2))}\|c_y\|_{L^2(0,T;Y)}
\\ \lesssim & \bM_{c,x}(T)(\bM_1(T)+\bM_2(T))\,.
\end{align*}
and
\begin{align*}
& Q(t,v)+8\|\psi\|_{L^2}^2\|\sqrt{c(t)}-\sqrt{c_0}\|_{L^2(\R)}^2
\\ \lesssim & 
\|v_0\|_{L^2(\R^2)}^2
+(\bM_1(T)+\bM_2(T)+\bM_{c,x}(T))^2\int_0^t \la s\ra^{-5/4}\,dt
\\ \lesssim & \|v_0\|_{L^2(\R^2)}^2+(\bM_1(T)+\bM_2(T)+\bM_{c,x}(T))^2\,.
\end{align*}
Combining the above with the fact that
$Q(t,v)=\|v(t)\|_{L^2}^2+O(\|\tc(t)\|_Y\|v(t)\|_{L^2})$,
we have Lemma~\ref{lem:M4-bound}.
Thus we complete the proof.
\end{proof}
\bigskip

\section{Estimates for small solutions for the KP-II equation}
\label{sec:v1}
In this section, we will give upper bounds of $\bM_1(T)$
and $\bM_1'(\infty)$. First, we will prove decay estimates for $v_1$
assuming that $v_0(x,y)$ is polynomially localized as $x\to\infty$.
\begin{lemma}
  \label{lem:v1-a} Let$\tv_1$ be a solution of \eqref{eq:tv1}.
There exist positive constants $C$ and $\delta_4$ such that if
$\|\la x\ra^2v_0\|_{L^2}+\bM_{c,x}(T)+\bM_1(T)+\bM_2(T)<\delta_4$, then
$\bM_1(T) \le C\|\la x\ra^2 v_0\|_{L^2}$ for $t\in[0,T]$.
\end{lemma}
To prove Lemma~\ref{lem:v1-a}, we make use of the virial identity
for the KP-II equation that was shown in \cite{dBM}.
Let $u$ be a solution of \eqref{eq:KPII} with $u(0)\in L^2(\R^2)$ and 
$$I(t)=\int_{\R^2}p_\a(x-x(t))u(t,x,y)^2\,dxdy\,.$$
Suppose that $\inf_{t\ge0}x'(t)>0$. There exist positive constants $\a_0$
and $\delta$ such that if $\a\in(0,\a_0)$ and $\|v_0\|_{L^2}<\delta$, then
\begin{equation}
  \label{eq:virial}
I(t)+ \int_0^t\int_{\R^2}p_\a'(x-x(s))\mathcal{E}(u)(s,x,y)\,dxdyds
\lesssim I(0)\,.
\end{equation}
See e.g. \cite[Lemma~5.3]{MT} for the proof.
\par
If $u(0)$ is small in $L^2(\R^2)$ and polynomially localized, we can prove
time decay estimates by using \eqref{eq:virial}.

\begin{lemma}
  \label{lem:virial-0}
Let $u(t)$ be a solution of \eqref{eq:KPII}.
Let $0\le T\le \infty$ and let $x(t)$ be a $C^1$ function satisfying
$x(0)=0$ and $\inf_{t\in[0,T]}\dot{x}(t)> c_1$ for a $c_1>0$.
Suppose that $(1+x_+)^\rho u(0)\in L^2(\R^2)$ for a $\rho\ge0$.
Then there exist positive constants $\a_0$ and $\delta$ such that if
$\a\in(0,\a_0)$ and $\|u(0)\|_{L^2(\R^2)}<\delta$, then
\begin{gather}
  \label{eq:v1-decay}
\int_\R p_\a(x-x(t))u(t,x,y)^2\,dxdy
\lesssim \la t \ra^{-2\rho}\|(1+x_+)^\rho u(0)\|_{L^2(\R^2)}^2\,,
\\ \label{eq:v1-decay2}
\int_0^T\int_\R p_\a(x-x(t))\mathcal{E}(u)(t,x,y)\,dxdydt
\lesssim \|(1+x_+)^{1/2}p_\a(x)^{1/2}u(0)\|_{L^2(\R^2)}^2\,,
\\ \label{eq:v1-decay1.5}
\int_\R (1+x_+)^{\rho_1}p_\a(x)u(t,x+x(t),y)^2\,dxdy
\lesssim \la t \ra^{-(\rho_2-\rho_1)}\|(1+x_+)^{\rho_2} u(0)\|_{L^2(\R^2)}^2\,,
\end{gather}
where $\rho_1$ and $\rho_2$ are positive constants satisfying $\rho_2>\rho_1>0$.
\end{lemma}
\begin{proof}
We can prove \eqref{eq:v1-decay} in the same way as in \cite[Section~14.1]{MPQ}.
Since $\min(1,e^{2\a x})\le p_\a(x)\le 2\min(1,e^{2\a x})$ and
$p_\a'(x)=\a\sech^2\a x=O(e^{-2\a|x|})$, it follows that for $x\le0$,
$$\sum_{j\ge0}p_\a(x-j)\lesssim 
\begin{cases}
  \sum_{j\ge0}e^{-2\a(|x|+j)} \lesssim p_\a(x)
\quad & \text{for $x\le0$,}
\\  \sum_{0\le j\le [x]} 1+\sum_{j\ge [x]+1} e^{-2\a|x-j|}
\lesssim  1+x \quad & \text{for $x\ge0$.}
\end{cases}$$
Similarly, we have $p_\a(x)\lesssim \sum_{j\ge0}p_\a'(x-j)$.
Hence it follows from \eqref{eq:virial} that for $t\in[0,T]$,
\begin{align*}
& \int_0^t\int_{\R^2}p_\a(x-x(s))\mathcal{E}(u)(s,x,y)\,dxdyds
\\ \lesssim &
\sum_{j=0}^\infty \int_0^t\int_{\R^2}p_\a'(x-x(s)-j)
\mathcal{E}(u)(s,x,y)\,dxdyds
\\ \lesssim  &
\sum_{j=0}^\infty \int_{\R^2}p_\a(x-j)u(0,x,y)^2\,dxdy
 \lesssim 
\int_{\R^2}(1+x_+)p_\a(x)u(0,x,y)^2\,dxdy\,.
\end{align*}
\par

Finally, we will prove \eqref{eq:v1-decay1.5}.
Let $c_1$ and $c_2$ be constants satisfying
$0<c_1<c_2\le\inf_{0\le t\le T}\dot{x}(t)$ and let 
$q_\ell(x)=(1+x_+)^{\rho_\ell}p_\a(x)$ for $\ell=1$, $2$.
 Since
$$q_\ell(x)\simeq \sum_{j\ge0}(1+j)^{\rho_\ell-1}p_\a(x-j)\,,$$
it follows from \eqref{eq:virial} that for $t\in[0,T]$,
\begin{equation}
\label{eq:W-virial}
\int_{\R^2}q_2(x-c_1t)u(t,x,y)^2\,dxdy
\lesssim \int_{\R^2}q_2(x)u(0,x,y)^2\,,
\end{equation}
provided $\|u(0)\|_{L^2}$ is sufficiently small.
Combining \eqref{eq:W-virial} with the fact that
\begin{equation*}
q_1(x-x(t)) \le q_1(x-c_2t) \lesssim \la t\ra^{\rho_1-\rho_2}q_2(x-c_1t)\,,
\end{equation*}
we have \eqref{eq:v1-decay1.5}.
Thus we complete the proof.
\end{proof}

Now we are in position to prove Lemma~\ref{lem:v1-a}.
 \begin{proof}[Proof of Lemma~\ref{lem:v1-a}]
By Claim~\ref{cl:cx_t-bound}, there exists a $c_1>0$ such that
$x_t(t,y)\ge c_1$ for every $t\in[0,T]$ and $y\in\R$.
Hence it follows from Lemmas~\ref{lem:nonzeromean1} and \ref{lem:virial-0}
that
$\bM_1(T) \lesssim \|(1+x_+)^2v_*\|_{L^2(\R^2)}
\lesssim \|\la x\ra^2v_0\|_{L^2(\R^2)}$.
Thus we complete the proof.  
\end{proof}

The scattering result by Hadac, Herr and Koch (\cite{HHK}) 
which uses $U^p$ and $V^p$ spaces introduced by
\cite{Koch-Tataru05, Koch-Tataru07} implies
that higher order Sobolev norms of solutions to \eqref{eq:tv1}
remain small for all the time.

\begin{lemma}
  \label{lem:scattering}
Let $\tv_1(t)$ be a solution of \eqref{eq:tv1}.
There exists positive constants $\delta_5$ and $C$ such that if
$\left\|\la x\ra^2v_0\right\|_{H^1(\R^2)}\le \delta_5$,
then $\bM_1'(\infty)\le C\left\|\la x\ra^2v_0\right\|_{H^1(\R^2)}$
for every $t\in\R$.
\end{lemma}

\begin{proof}
It follows from \cite[Proposition~3.1 and Theorem~3.2]{HHK} that
\begin{align*}
\|\pd_x\tv_1(t)\|_{L^2(\R^2)}
+\left\||D_x|^{-1/2}\la D_y\ra^{1/2}\tv_1(t)\right\|_{L^2(\R^2)}
\lesssim &
\|\pd_xv_*\|_{L^2(\R^2)}+\left\||D_x|^{-1/2}\la D_y\ra^{1/2}v_*\right\|_{L^2(\R^2)}
\\ \lesssim & \|\mathcal{E}(v_*)^{1/2}\|_{L^2(\R^2)}\,.
\end{align*}
See e.g. \cite[Section~7.2]{Miz17} for an explanation.
Combining the above with 
the $L^2$-conservation law $\|\tv_1(t)\|_{L^2(\R^2)}=\|v_*\|_{L^2(\R^2)}$
and the Sobolev inequality \eqref{eq:Sobolev}, we have
\begin{align*}
\|\tv_1(t)\|_{L^3(\R^2)} \lesssim &
\left\||D_x|^{1/2}\tv_1(t)\right\|_{L^2(\R^2)}
+\left\||D_x|^{-1/2}\la D_y\ra^{1/2}\tv_1(t)\right\|_{L^2(\R^2)}
\lesssim  \|\mathcal{E}(v_*)^{1/2}\|_{L^2(\R^2)}\,.
\end{align*}
Let 
$$H(u)=\frac12\int_{\R^2}\left\{
(\pd_xu)^2-3(\pd_x^{-1}\pd_yu)^2-2u^3\right\}\,dxdy\,.$$
Since $H(u)$ is the Hamiltonian of the KP-II equation and $\tv_1$
is a solution of \eqref{eq:tv1} satisfying
$\tv_1\in C(\R;H^1(\R^2))$ and $\pd_x^{-1}\pd_y\tv_1\in C(\R;L^2(\R^2))$,
\begin{align*}
3\|\pd_x^{-1}\pd_y\tv_1(t)\|_{L^2(\R^2)}^2
\le & -2H(\tv_1(t))+\|\pd_x\tv_1(t)\|_{L^2(\R^2)}^2+2\|\tv_1(t)\|_{L^2(\R^2)}^3
\\ =& -2H(v_*)+\|\pd_x\tv_1(t)\|_{L^2(\R^2)}^2+2\|\tv_1(t)\|_{L^2(\R^2)}^3
\lesssim \|\mathcal{E}(v_*)^{1/2}\|_{L^2(\R^2)}^2\,.
\end{align*}
Combining the above with Lemma~\ref{lem:nonzeromean1},
we have Lemma~\ref{lem:scattering}.
\end{proof}
\bigskip

\section{Decay estimates for the exponentially localized part of perturbations}
In this section, we will estimate $v_2(t)$ following the line of \cite{Miz15}. 
\begin{lemma}
  \label{lem:exp-bound}
Let $\eta_0$ be a small positive number and $\a\in (\nu(\eta_0),2)$.
Suppose that $\bM_1'(\infty)$ is sufficiently small.
Then there exist positive constants $\delta_6$ and $C$
such that if   $\bM_2(T)+\bM_v(T)\le \delta_6$,
\begin{equation}
  \label{eq:M4-bound}
\bM_2(T)\le C\left(
\left\|\la x\ra v_0\right \|_{L^2(\R^2)}+\bM_{c,x}(T)+\bM_1(T)\right)\,.
\end{equation}
\end{lemma}
First, we estimate the low frequency part of $v_2(t)$ assuming the smallness of
$\bM_{c,x}(T)$, $\bM_2(t)$ and $\bM_v(T)$.
\begin{lemma}
  \label{lem:nonresonant-ylow}
Let $\eta_0$, $\a$ and $M$ be positive constants satisfying $\nu(\eta_0)<\a<2$ and $\nu(M)>\a$.
Suppose that $v_2(t)$ is a solution of \eqref{eq:v2}.
Then there exist positive constants $b_1$, $\delta_6$ and $C$ such that
if $\bM_{c,x}(T)+\bM_v(T)+\bM_1(T)+\bM_2(T)<\delta_6$, then for $t\in[0,T]$,
\begin{equation}
\label{eq:nonresonant-ylow}
\begin{split}
 \left\|P_1(0,M)v_2(t,\cdot)\right\|_X \le &  Ce^{-bt}\|v_{2,*}\|_X \\
& +C\left\{\bM_{c,x}(T)+\bM_1(T)+\bM_2(T)(\bM_2(T)+\bM_v(T))\right\}
\la t\ra^{-3/4}\,.
\end{split}
\end{equation}
\end{lemma}
\begin{proof}
Let $\tv_2(t)=P_2(\eta_0,M)v_2(t)$ and
$N_{2,2}'=\{2\tc(t,y)+6(\varphi(z)-\varphi_{c(t,y)}(z))\}v_2(t,z,y)$.
Applying Proposition~\ref{prop:semigroup-est} to \eqref{eq:v2}, we have
\begin{equation}
  \label{eq:v1n-est}
  \begin{split}
& \|\tv_2(t)\|_X\lesssim e^{-bt}\|v_{2,*}\|_X 
+\int_0^t e^{-b'(t-s)}(t-s)^{-3/4}\|e^{\a z}P_2N_{2,1}(s)\|_{L^1_zL^2_y}\,ds
\\ & \enskip +\int_0^t e^{-b'(t-s)}(t-s)^{-1/2}
(\|N_{2,2}(s)\|_X+\|N_{2,2}'(s)\|_X+\|N_{2,4}\|_X)\,d s
\\ & \qquad  +\int_0^t e^{-b(t-s)}(\|\ell(s)\|_X+\|P_2N_{2,3}(s)\|_X)\,ds\,,
  \end{split}
\end{equation}
where we abbreviate $P_2(\eta_0,M)$ as $P_2$.
It follows from \cite[Claim~9.1]{Miz15} that for $t\in[0,T]$,
\begin{equation}
  \label{eq:N21-est}
\begin{split}
\|e^{\a z}P_2N_{2,1}\|_{L^1_zL^2_y}\lesssim
& \sqrt{M}(\|v_1\|_{L^2}+\|v_2\|_{L^2})\|v_2\|_X
\\ \lesssim & \sqrt{M}(\bM_1(T)+\bM_v(T))\bM_2(T)\la t\ra^{-3/4}\,.
\end{split}
\end{equation}
By the definitions and Claim~\ref{cl:cx_t-bound},
\begin{equation}
  \label{eq:ell-est}
\begin{split}
\|\ell_1\|_X\lesssim & (\bM_{c,x}(T)+\bM_1(T)+\bM_2(T)^2)\la t\ra^{-3/4}\,,
\\ 
\|\ell_2\|_X\lesssim & e^{-\a(3t+L)}(\bM_{c,x}(T)+\bM_1(T)+\bM_2(T)^2)\,,
\end{split}
\end{equation}
and
\begin{equation}
  \label{eq:N22-est}
\begin{split}
\|N_{2,2}\|_X \lesssim &
(\|x_t-2c-3(x_y)^2\|_{L^\infty}+\|\tc\|_{L^\infty})\|v_2\|_X\,,
\\ \lesssim & (\bM_{c,x}(T)+\bM_1(T)+\bM_2(T)^2)\bM_2(T)\la t\ra^{-5/4}\,.
\end{split}
\end{equation}
in the same way as (8.6) and (8.7) in \cite{Miz17}.
Since
$$\|\tc(t)\|_{L^\infty}+\sum_{k=1,2}\|\pd_y^kx(t)\|_{L^\infty}
\lesssim \bM_{c,x}(T)\la t\ra^{-1/2} \quad\text{for $t\in[0,T]$,}$$
\begin{equation}
\label{eq:N2,2'-est}
\|N_{2,2}'\|_X +\|P_2N_{2,3}\|_X \lesssim
  (\|\tc\|_{L^\infty}+\|x_y\|_{L^\infty}+\|x_{yy}\|_{L^\infty})\|v_2\|_X  
 \lesssim  
\bM_{c,x}(T)\bM_2(T)\la t\ra^{-5/4}\,.
\end{equation}
Here we use the fact that $\|\pd_yP_2\|_{B(X)}\lesssim M$.
Since $|e^{\a z}\{\varphi_c(z)-\tpsi_c(z)\}|\lesssim p_\a(z+3t+L)$, we have
\begin{equation}
\label{eq:N2,4-est}
\|N_{2,4}\|_X \lesssim \bM_1(T)\la t\ra^{-2}\quad\text{for $t\in[0,T]$.}  
\end{equation}
Combining \eqref{eq:v1n-est}--\eqref{eq:N2,4-est},
we have for $t\in[0,T]$,
$$\|\tv_2(t)\|_X \lesssim  e^{-bt}\|v_{2,*}\|_X
+ \{\bM_{c,x}(T)+\bM_1(T)+(\bM_v(T)+\bM_2(T))\bM_2(T)\}\la t\ra^{-3/4}\,.$$
\par
As long as $v_2(t)$ satisfies the orthogonality condition \eqref{eq:orth} and
$\tc(t,y)$ remains small, we have
$$\|\tv_2(t)\|_X\lesssim \|P_1(0,M)v_2(t)\|_X\lesssim \|\tv_2(t)\|_X$$
in exactly the same way as the proof of lemma~9.2 in \cite{Miz15}.
Thus we have \eqref{eq:nonresonant-ylow}.
This completes the proof of lemma~\ref{lem:nonresonant-ylow}.
\end{proof}
Using a virial identity, we can estimate
the exponentially weighted norm of $v_2(t)$ for high frequencies
in $y$ in the same way as \cite[Lemma~8.3]{Miz17}.  
\begin{lemma}
\label{lem:virial}
Let $\a\in(0,2)$ and $v_2(t)$ be a solution of \eqref{eq:v2}.
Suppose $\bM_1'(\infty)$ is sufficiently small.
Then there exist positive constants $\delta_6$ and $M_1$ such that
if $\bM_{c,x}(T)+\bM_1(T)+\bM_2(T)+\bM_v(T)<\delta_6$ and $M\ge M_1$,
then for $t\in[0,T]$,
\begin{align*}
& \|v_2(t)\|_X^2 \lesssim
e^{-M\a t}\|v_{2,*}\|_X^2
 +\int_0^t e^{-M\a(t-s)}\left(\|\ell(s)\|_X^2
+\|P_1(0,M)v_2(s)\|_X^2+\|N_{2,4}(s)\|_X^2\right)\,ds\,,
\end{align*}
\begin{equation*}
\|\mathcal{E}(v_2)^{1/2}\|_{L^2(0,T;X)}\lesssim
\|v_{2,*}\|_X+\|\ell\|_{L^2(0,T;X)}+\|P_1(0,M)v_2\|_{L^2(0,T);X)}
+\|N_{2,4}\|_{L^2(0,T;X)}\,.  
\end{equation*}
\end{lemma}

Now we are in position to prove Lemma~\ref{lem:exp-bound}.
\begin{proof}[Proof of Lemma~\ref{lem:exp-bound}]
Combining Lemmas~\ref{lem:nonresonant-ylow} and \ref{lem:virial},
\eqref{eq:ell-est} and \eqref{eq:N2,4-est} with \eqref{eq:v2*-bound}, we have
\begin{align*}
\bM_2(T)\lesssim &
\|v_{2,*}\|_X+\bM_{c,x}(T)+\bM_1(T)+\bM_2(T)(\bM_2(T)+\bM_v(T))
\\ \lesssim & \|\la x\ra v_0\|_{L^2(\R^2)}+\bM_1(T)+\bM_2(T)(\bM_2(T)+\bM_v(T))\,.
\end{align*}
Thus we obtain \eqref{eq:M4-bound} provided $\bM_2(T)$ and $\bM_v(T)$ are
sufficiently small.
This completes the proof of Lemma~\ref{lem:exp-bound}.
\end{proof}

\bigskip
\section{Large time behavior of the phase shift of line solitons}
\label{sec:phase}
In this section, we will prove Theorem~\ref{thm:phase}.
To begin with, we remark that $\bM_{c,x}(T)$, $\bM_1(T)$, $\bM_2(T)$ and
$\bM_v(T)$ remain small for every $T\in[0,\infty]$ provided
the initial perturbation $v_0$ is sufficiently small.
Combining Proposition~\ref{prop:continuation}, Lemmas~\ref{lem:Mcx-bound},
\ref{lem:M4-bound}, \ref{lem:v1-a}, \ref{lem:scattering}
and \ref{lem:exp-bound}, we have the following.
\begin{proposition}
\label{prop:poly}  
There exist positive constants $\eps_0$ and $C$ such that if 
$\eps:=\|\la x\ra(\la x\ra+\la y\ra)v_0\|_{H^1(\R^2)}<\eps_0$,
then
$\bM_{c,x}(\infty)+\bM_1(\infty)+\bM_2(\infty)+\bM_v(\infty)+\bM_1'(\infty)
\le C\eps$.
\end{proposition}

When we estimate the $L^\infty$ norm of $\tx$ by applying 
Lemma~\ref{lem:fund-sol} to \eqref{eq:modeq}, two terms $\cN^6$ and
$B_4^{-1}\widetilde{\mathcal{A}}_1(t){}^t(b,\tx)$ are hazardous because
they do not necessarily belong to $L^1(\R)$.
\par

We introduce a change of variable to eliminate $E_1\obu{B}_4k_t\mathbf{e_1}$
and a bad part of $B_4^{-1}\widetilde{\mathcal{A}}_1(t)$.
Let
\begin{gather*}
\begin{pmatrix} \ta_3(t,D_y) & 0 \\ \ta_4(t,D_y) & 0\end{pmatrix}
:=B_4^{-1}\widetilde{\mathcal{A}}_1(t)\,,\enskip
\ta_{31}(t)=\ta_3(t,0)\,,\enskip
\ta_{32}(t,\eta)=\frac{\ta_3(t,\eta)-\ta_3(t,0)}{\eta^2}\,,
\end{gather*}
and $\gamma(t)=e^{-\int_0^t \ta_{31}(s)\,ds}$.
Note that $\ta_3(t,\eta)$ is even in $\eta$ because 
$g_k^*(z,\eta)$ thus the symbols of $B_4$ and $\widetilde{\mathcal{A}}_1(t)$
are even in $\eta$. By \cite[Claim~6.3]{Miz17},
the operator $B_4^{-1}$ is uniformly bounded in $B(Y)$ and we can prove
that for $t\ge0$,
\begin{equation}
  \label{eq:a3-a4}
|\ta_{31}(t)|+|\ta_{31}'(t)| +\|\ta_{32}(t,D_y)\|_{B(Y)}+\|\ta_3(t,D_y)\|_{B(Y)}
  +\|\ta_4(t,D_y)\|_{B(Y)}
\lesssim e^{-\a(3t+L)}
\end{equation}
in exactly the same way as \cite[Claim~D.3]{Miz15}.
We need to replace $e^{-\a(4t+L)}$ in \cite[Claim~D.3]{Miz15} by
$e^{-\a(3t+L)}$ because $\tpsi_c(z)=\psi_{c,L}(z+3t)$ in our paper
whereas $\tpsi_c(z)=\psi_{c,L}(z+4t)$ in \cite{Miz15}.
By the definitions of $\wS_0$, $\wS_1$ and $\wS_3$
(see \cite[pp.40--41]{Miz15}) and  \cite[(A1), (A6)]{Miz17},
\begin{equation}
  \label{eq:A1-A3-est}
\|\mathcal{A}_1(t)\|_{B(Y) \cap B(Y_1)}=O(1)\,,
\quad \|\mathcal{A}_2(t)\|_{B(Y) \cap B(Y_1)}=O(e^{-\a(3t+L)})\,.
\end{equation}
By  \eqref{eq:a3-a4},
$$0<\inf_{t\ge0}\gamma(t)\le \sup_{t\ge0}\gamma(t)<\infty\,,
\quad \lim_{t\to\infty}\gamma(t)>0\,.$$
Let $\mathbf{k}(t,y)=\gamma(t)\tk(t,y)\mathbf{e_1}$,
$\tb(t,y)=b(t,y)+\tk(t,y)$  and
\begin{equation}
  \label{eq:tbdef}
\bb(t,y)={}^t\!(b_1(t,y),b_2(t,y))=\gamma(t){}^t\!(\tb(t,y),\tx(t,y))\,.
\end{equation}
By Claim~ \ref{cl:k-decay} and \eqref{eq:S311-d},
\begin{equation}
  \label{eq:b-b1}
  \begin{split}
\|\pd_y^kb_1(t)\|_Y \lesssim 
 & \|\pd_y^kb(t)\|_Y +\|k(t)\|_Y
\\ \lesssim & \la t\ra^{-(2k+1)/4}
\bM_{c,x}(\infty)+\bM_1(\infty)\la t\ra^{-2} \quad\text{for $k\ge0$,}
  \end{split}
\end{equation}
\begin{gather}
  \label{eq:tx-b2}
\|\pd_y^kb_2(t)\|_Y \lesssim 
 \|\pd_y^k\tx(t)\|_Y 
 \lesssim \la t\ra^{-(2k-1)/4}\bM_{c,x}(\infty)
 \quad\text{for $k\ge1$,}
 \\
 \label{eq:tx-b2'}
\|\tx(t)\|_{L^\infty} \lesssim  \|b_2(t)\|_{L^\infty} +\|k(t)\|_Y
\lesssim  \|b_2(t)\|_{L^\infty}+\bM_1(\infty)\la t\ra^{-2}\,.
\end{gather}
Note that $\|\pd_y^kk(t)\|_Y\lesssim \eta_0^k\|k(t)\|_Y$ for any $k\ge1$.

\par
Substituting \eqref{eq:tbdef} into \eqref{eq:modeq} and using
\eqref{eq:cN6}, \eqref{eq:bad-part1} and \eqref{eq:Rv1'}, we have
\begin{equation}
  \label{eq:modeq3}
\pd_t\bb
=\mathcal{A}_*\bb+\gamma\left\{
\sum_{i=1}^5\cN^i+\obu{\cN}_6+\pd_y\oc{\cN}_6+\obu{\cN}_7
+\pd_y^2\oc{\cN}_7\right\}\,,
\end{equation}
where $\obu{\cN}_6=\sum_{1\le j\le 3}\obu{\cN}_{6j}$,
$\oc{\cN}_6=\oc{\cN}_{61}+\pd_y\oc{\cN}_{62}$,
\begin{align*}
\obu{\cN}_{61}=& \gamma^{-1}\pd_t\{\gamma(2-S^3_{11}[\psi])^{-1}\}k\mathbf{e_1}\,,
\quad
\obu{\cN}_{62}=E_2\obu{B}_4R^{v_1}-2E_{21}\mathbf{k}\,,
\quad  E_{21}=\begin{pmatrix} 0 & 0 \\  1 & 0\end{pmatrix}\,,
\\
\obu{\cN}_{63}=& (2-S^3_{11}[\psi])^{-1}\left\{\obu{R}_{v_1}
    +S^7_1[\pd_c\varphi_c](c_t)-S^7_1[\varphi_c'](x_t-2c-3(x_y)^2)\right\}
     +\obu{B}_{34}R^{v_1}\,,
\end{align*}
\begin{align*}
\oc{\cN}_{61}=& \obu{B}_4\oc{R}_{v_1,1}\mathbf{e_1}\,,\quad
\oc{\cN}_{62}=\left\{\obu{B}_4\oc{R}_{v_1,2}\mathbf{e_1}
  -(\oc{B}_{14}+\oc{B}_{34})R^{v_1}\right\}\,,
\\
\obu{\cN}_7=& \{\ta_4(t,D_y)-\ta_{31}(t)\}b(t,\cdot)\mathbf{e_2}\,,
\\
\oc{\cN}_7=&
\left(\pd_y^2\mathcal{A}_1+\mathcal{A}_2\right)
(b(t,\cdot)\mathbf{e_1}+\tx(t,\cdot)\mathbf{e_2})
-\ta_{32}(t,D_y)b(t,\cdot)\mathbf{e_1}
-\pd_y^{-2}(\mathcal{A}_*-2E_{21})\mathbf{k}(t,\cdot)\,.
\end{align*}
\par

Now we start to prove Theorem~\ref{thm:phase}.
\begin{proof}[Proof of Theorem~\ref{thm:phase}]
Using the variation of constants formula, we can translate \eqref{eq:modeq3}
into
\begin{equation}
  \label{eq:formula-bb}
  \begin{split}
  \bb(t)=& e^{t\mathcal{A}_*}\bb(0)
\\ & +\int_0^t e^{(t-s)\mathcal{A}_*}\gamma(s)
  \left(\sum_{i=1}^5\cN^i(s)+\obu{\cN}_6(s)+\pd_y\oc{\cN}_6(s)
    +\obu{\cN}_7(s)+\pd_y^2\oc{\cN}_7(s)\right)\,ds\,.      
  \end{split}
\end{equation}
Now we will estimate the $L^\infty$-norm of the right hand side of
\eqref{eq:formula-bb}.
By \eqref{eq:modeq-init} and \eqref{eq:tbdef}, 
\begin{align*}
&b_1(0,y)=b_*(y)+\frac12\left(2-S^3_{11}[\psi](0)\right)^{-1}
\wP_1\left(\int_\R v_*(x,y)\varphi_{c_*(y)}(x-x_*(y))\,dx\right)\,,
\\ & b_2(0,y)=x_*(y)\,,
\end{align*}
and it follows from Lemmas~\ref{lem:fund-sol} and \ref{lem:modeq-init-decomp}
that
\begin{gather}
\left\| e^{t\mathcal{A}_*}\bb(0)\right\|_{L^\infty(\R)}
\lesssim  \bigl\|\obu{b}\bigr\|_{L^1}+\bigl\|\oc{b}\bigr\|_{Y_1}
+\|x_*\|_{Y_1}+\|v_*\|_{L^1(\R^2)}
\lesssim  \eps\,,
\notag \\
\label{eq:phase-hom}
\left \|\mathbf{e_2}\cdot e^{t\mathcal{A}_*}\bb(0)-
\frac12H_{2t}*W_{4t}*b_1(0)\right\|_{L^\infty}
\lesssim \la t\ra^{-1/2}\eps\,,
\end{gather}
where $\eps=\bigl\|\la x\ra\la y\ra v_0\bigr\|_{L^2(\R^2)}$.
\par
Next, we will estimate $\cN^1$.
Let $\obu{\cN}_1=\wP_1\{2(\tc-b)+3(x_y)^2\}\mathbf{e_2}$ and
$\oc{\cN}_1=6\wP_1(b\tx_y)\mathbf{e_1}$. Then
$E_1\cN_1=\pd_y\oc{\cN}$, $E_2\cN_1=\obu{\cN}_1$,
$\cN^1=\diag(\pd_y,1)(\obu{\cN}_1+\oc{\cN}_1)$,  and
$$III_1(t):=\|\obu{\cN}_1\|_{Y_1}+\|\oc{\cN}_1\|_{Y_1}
\lesssim \bM_{c,x}(\infty)^2\la t\ra^{-1/2}\,.$$
By \eqref{eq:profile} and the fact that
\begin{equation}
  \label{eq:b-tc}
b-\tc=\frac{4}{3}\wP_1
\left\{\left(\frac{c}{2}\right)^{3/2}-1-\frac{3}{4}\tc\right\}
=\frac18\wP_1\tc^2+O(\tc^3)
\end{equation}
(see \cite[Claim~D.6]{Miz15}),
\begin{gather*}
III_2(t):=\left\|bx_y-2\{u_B^+(t,\cdot+4t)^2-u_B^-(t,\cdot-4t)^2\}\right\|_{L^1}
\lesssim \eps^2\delta(t)\la t\ra^{-1/2}\,,
\\
III_3(t):= \left\|2(\tc-b)+3(x_y)^2-2\{u_B^+(t,\cdot+4t)^2
+u_B^-(t,\cdot-4t)^2\}\right\|_{L^1}
\lesssim \eps^2\delta(t)\la t\ra^{-1/2}\,,
\end{gather*}
where $\delta(t)$ is a functions that tends to $0$ as $t\to\infty$.
Note that 
$\|u_B^+(t,\cdot+4t)u_B^-(t,\cdot-4t)\|_{L^1}=O(\eps^2t^{-1/2}e^{-8t})$.
By Lemma~\ref{lem:fund-sol} and \cite[Claim~4.1]{Miz15},
\begin{align*}
\Bigl\| \int_0^t e^{(t-s)\mathcal{A}_*}\gamma(s)\cN^1(s)\,ds\Bigr\|_{L^\infty}
\lesssim &
\int_0^t \la t-s\ra^{-1/2}III_1(s)\,ds
\\ \lesssim & \bM_{c,x}(\infty)\int_0^t\la t-s\ra^{-1/2}\la s\ra^{-1/2}
\lesssim \bM_{c,x}(\infty)^2\,,
\end{align*}
\begin{align*}
& \Bigl\| \int_0^t e^{(t-s)\mathcal{A}_*}\gamma(s)\bigl\{\cN^1(s)\,ds
\\ &
\phantom{\Bigl\| \int_0^t e^{(t-s)\mathcal{A}_*}\gamma(s)}
-\sum_\pm H_{2t}(\cdot \pm4(t-s))
*\{4u_B^\pm(s,\cdot\pm 4s)^2-2u_B^\mp(s,\cdot\mp 4s)^2\bigr\}\,ds\mathbf{e_2}
\Bigr\|_{L^\infty}
\\ 
\lesssim &
\int_0^t \la t-s\ra^{-1} III_1(s)\,ds
+\int_0^t \la t-s\ra^{-1/2}\left(III_2(s)+III_3(s)\right)\,ds
\\ \lesssim & \eps^2 \la t\ra^{-1/2}\log(t+2)
+\eps^2\int_0^t (t-s)^{-1/2}s^{-1/2}\delta(s)\,ds
\to0\quad \text{as $t\to\infty$.}
\end{align*}
\par

For $y$ satisfying $\min\{|y-4t|\,,\,|y+4t|\}\ge 2\delta(t)^{-1/2}\sqrt{t} $,
\begin{align*}
& \int_0^t \int_\R H_{2(t-s)}(y-y_1\pm4t)
\left\{H_{2s}(y_1)^2+H_{2s}(y_1\mp8s)^2\right\}\,dy_1ds
\\ \lesssim &
\int_0^t (t-s)^{-1/2}s^{-1}e^{-\delta(t)^{-1}/8}
\left(\int_{|y_1|\le \delta(t)^{-1/2}\sqrt{t}} e^{-y_1^2/8s}\,dy_1\right)\,ds
\\ & 
+\int_0^t (t-s)^{-1/2}s^{-1}\left(
\int_{|y_1|\ge \delta(t)^{-1/2}\sqrt{t}} e^{-y_1^2/8s}\,dy_1\right)\,ds
\\\lesssim & \exp\bigl(-\delta(t)^{-1}/8\bigr)\to 0\quad
\text{as $t\to\infty$.}
\end{align*}
Combining the above with the fact that $|u_B^\pm(s,y)|\lesssim H_{2s}(y)$,
we obtain
\begin{equation*}
\lim_{t\to\infty}\left\|  \int_0^t e^{(t-s)\mathcal{A}_*}\gamma(s)\cN^1(s)\,ds
\right\|_{L^\infty(|y\pm4t|\ge \delta t)}=0\,.
\end{equation*}
\par
The other terms can be decomposed as
\begin{equation}
  \label{eq:nonl-classify}
\sum_{i=2}^5\cN^i(t)
+\obu{\cN}_6+\pd_y\oc{\cN}_6+\obu{\cN}_7+\pd_y^2\oc{\cN}_7
=\obu{\cN}_a+\obu{\cN}_b+\pd_y(\oc{\cN}_a+\oc{\cN}_b)
+\pd_y^2\oc{\cN}_c\,,  
\end{equation}
such that $\obu{\cN}_b=E_2\obu{\cN}_b$ and
\begin{gather}
  \label{eq:cNa-est}
 \|\chi(D_y)E_1\obu{\cN}_a\|_{L^1(\R)}
+\|\pd_y^{-1}(I-\chi(D_y))E_1\obu{\cN}_a\|_{Y_1}
 \lesssim  \left(e^{-\a L}\eps+\eps^2\right)\la t \ra^{-3/2}\,,
\\
\label{eq:cNa2-est}
\bigl\|E_2\obu{\cN}_a\bigr\|_{Y_1}+\bigl\|\oc{\cN}_a\bigr\|_{Y_1}\lesssim 
\eps^2\la t\ra^{-1}\,,
\\
\label{eq:cNbc-est}
\bigl\|\obu{\cN}_b\bigr\|_Y+\bigl\|\oc{\cN}_b\bigr\|_Y\lesssim 
\eps \la t\ra^{-7/4}\,,
\quad
\bigl\|\oc{\cN}_c\bigr\|_Y\lesssim \eps\la t\ra^{-1}\,.
\end{gather}
Hence it follows from Proposition~\ref{prop:poly} and
Lemma~\ref{lem:fund-sol} that
\begin{align}
\label{eq:phase-Na1}
& \sup_{t\ge0} \left\| \int_0^t e^{(t-s)\mathcal{A}_*}\gamma(s)
\chi(D_y)E_1\obu{\cN}_a(s)\,ds\right\|_{L^\infty} 
\lesssim e^{-\a L}\eps+\eps^2\,,
\\ &
\label{eq:phase-Na2}
\left\| \int_0^t e^{(t-s)\mathcal{A}_*}\gamma(s)
(I-\chi(D_y))E_1\obu{\cN}_a(s)\,ds\right\|_{L^\infty} 
\lesssim (e^{-\a L}\eps+\eps^2) \la t\ra^{-1/2}\,,
\end{align}
\begin{equation}
  \label{eq:phase-Nbc}
  \begin{split}
\Bigl\| \int_0^t e^{(t-s)\mathcal{A}_*}\gamma(s)
\bigl[E_2\obu{\cN}_a(s)+&
\obu{\cN}_b (s)+\pd_y\{\oc{\cN}_a(s)+\oc{\cN}_b(s)\}
\\ & +\pd_y^2\oc{\cN}_c(s)\bigr]\,ds\Bigr\|_{L^\infty} 
\lesssim \eps \la t\ra^{-1/4}\,.    
  \end{split}
\end{equation}
By Lemma~\ref{cl:phase-lim}, \eqref{eq:fund-asymp1} and  \eqref{eq:cNa-est},
that for any $\delta>0$,
\begin{gather}
\label{eq:phase-Na3}
\lim_{t\to\infty}\sup_{|y|\le (4-\delta)t}
\left| \int_0^t e^{(t-s)\mathcal{A}_*}\gamma(s)
\chi(D_y)E_1\obu{\cN}_a(s)\,ds-h_a\mathbf{e_2}\right|=0\,,
\\ \label{eq:phase-Na4}
h_a=\frac12\int_0^\infty\int_\R \gamma(s)\mathbf{e_1}
\cdot\obu{\cN}_a(s,y)\,dyds\,,
\quad |h_a|\lesssim \eps e^{-\a L}+\eps^2\,,
\\ \notag
\lim_{t\to\infty}\sup_{|y|\ge (4+\delta)t}
\left| \int_0^t e^{(t-s)\mathcal{A}_*}\gamma(s)
\chi(D_y)E_1\obu{\cN}_a(s)\,ds\right|=0\,.  
\end{gather}
\par
Now, we will prove \eqref{eq:nonl-classify}--\eqref{eq:cNbc-est}.
First, we will estimate $\cN^2$.
As in \cite[Claim~D.7]{Miz15},
\begin{align}
& \label{eq:R71-est}
\|\wP_1R^7_1\|_{Y_1}+\|\chi(D_y)R^7_1\|_{L^1}
\lesssim \bM_{c,x}(\infty)^2\la t\ra^{-3/2}\,,
\\ &  \label{eq:R72-est}
\|\wP_1R^7_2\|_{Y_1}\lesssim \bM_{c,x}(\infty)^2\la t\ra^{-1}\,,
\quad \|\wP_1R^7_2\|_Y\lesssim \bM_{c,x}(\infty)^2\la t\ra^{-5/4}\,,
\end{align}
where 
\begin{align*}
R^7_1=& \left\{4\sqrt{2}c^{3/2}-16-12b\right\}x_{yy}
-6(2b_y-(2c)^{1/2}c_y)x_y-3c^{-1}(c_y)^2\,,\\
R^7_2=& 6\left\{\left(\frac{c}{2}\right)^{3/2}-1\right\}x_{yy}
+3\left(\frac{c}{2}\right)^{1/2}c_yx_y-3(bx_y)_y+\mu_2\frac{2}{c}(c_y)^2
\\ & + \frac{3}{2}(c^2-4)(I-\wP_1)(x_y)^2\,.
\end{align*}  
Let 
\begin{gather*}
\wR^{11}=R^{31}+R^{41}+R^{61}+\wS_{41}\begin{pmatrix}  0 \\ 2\tc\end{pmatrix}\,,
\quad 
\wR^{12}=R^{32}+R^{42}+R^{62}+\wS_{42}\begin{pmatrix}  0 \\ 2\tc\end{pmatrix}\,,
\\
\wR^{31}=R^{91}+R^{11,1}\,,\quad \wR^{32}=R^{92}+R^{11,2}\,.
\end{gather*}
Then $\wR^1=\wR^{11}-\pd_y^2\wR^{12}$ and $\wR^3=\wR^{31}-\pd_y^2\wR^{32}$.
See Appendix~\ref{ap:r} for the definitions of $R^{j1}$ and $R^{j2}$
and see \eqref{eq:defS3k1}--\eqref{eq:defS4k2} and \eqref{eq:defwS4l}
for the definitions of $\wS_{4\ell}$ ($\ell=1$, $2$).
\par
By Claims~\ref{cl:S3}, \ref{cl:S4}, \ref{cl:R3}--\ref{cl:R9} and \ref{cl:R11},
\begin{gather}
\label{eq:wR11-est}
\|\chi(D_y)\wR^{11}\|_{L^1(\R)}+\|\wR^{11}\|_{Y_1}+\|\wR^{12}\|_{Y_1}
\lesssim (\bM_{c,x}(\infty)^2+\bM_2(\infty)^2+\bM_1(\infty)\bM_2(\infty))
\la t\ra^{-3/2}\,,
  \\ \label{eq:wR31-est}
\|\chi(D_y)\wR^{31}\|_{L^1(\R)}+\|\wR^{31}\|_{Y_1}+\|\wR^{32}\|_{Y_1}
\lesssim
\bM_{c,x}(\infty)(\bM_{c,x}(\infty)+\bM_2(\infty))\la t\ra^{-3/2}\,.
\end{gather}
Let $\obu{\cN}_2=\obu{\cN}_{21}+\obu{\cN}_{22}+\obu{\cN}_{23}$ and
\begin{align*}
& \obu{\cN}_{21}=\obu{B}_4(\wP_1R^7_1\mathbf{e_1}+\wR^{11}+\wR^{31})
+\obu{B}_{34}(\wP_1R^7+\wR^1+\wR^3)\,,
\\ & 
\obu{\cN}_{22}=B_1^{-1}\wP_1R^7_2\mathbf{e_2}\,,\quad 
\obu{\cN}_{23}=\bigl(\obu{B}_4-B_1^{-1}\bigr)\wP_1R^7_2\mathbf{e_2}\,,
\\ &
\oc{\cN}_2=\obu{B}_4(\wR^{12}+\wR^{32})
+(\oc{B}_{14}+\oc{B}_{34})(\wP_1R^7+\wR^1+\wR^3)\,.
\end{align*}
Since $B_3^{-1}=\obu{B}_4+\obu{B}_{34}-\pd_y^2(\oc{B}_{14}+\oc{B}_{34})$
and $[\obu{B}_4,\pd_y]=O$, we have $\cN^2=\obu{\cN}_2-\pd_y^2\oc{\cN}_2$.
By \eqref{eq:R71-est}--\eqref{eq:wR31-est} and Claim~\ref{cl:b4part1}, 
\begin{align*}
 & \|\obu{\cN}_{21}\|_{Y_1}+\|\chi(D_y)\obu{\cN}_{21}\|_{L^1(\R)}
\lesssim (\bM_{c,x}(\infty)+\bM_2(\infty))^2\la t\ra^{-3/2}\,,
\\ &
\|\oc{\cN}_2\|_{Y_1}\lesssim (\bM_{c,x}(\infty)+\bM_2(\infty))^2\la t\ra^{-1}\,.
\end{align*}
By \eqref{eq:b4part2'},  \eqref{eq:R72-est} and the fact that
$B_1^{-1}\mathbf{e_2}=\frac12\mathbf{e_2}$,
\begin{gather*}
\|\obu{\cN}_{22}\|_{Y_1}\lesssim \bM_{c,x}(\infty)^2\la t\ra^{-1}\,,
\quad \obu{\cN}_{22}=E_2\obu{\cN}_{22}\,,
\\
\|\obu{\cN}_{23}\|_{Y_1}+\|\chi(D_y)\obu{\cN}_{23}\|_{L^1(\R)}
\lesssim e^{-\a(3t+L)}\la t\ra^{-1}\bM_{c,x}(\infty)^2\,.  
\end{gather*}
\par

Next we will estimate $\cN^3$. Let
$$\obu{\cN}_3=[\obu{B}_{34},\pd_y](\wR^2+\wR^4)\,,\quad
\oc{\cN}_3=
(B_4^{-1}+\obu{B}_{34}-\pd_y\oc{B}_{34}\pd_y)(\wR^2+\wR^4)\,.$$
Then $\cN^3=\obu{\cN}_3+\pd_y\oc{\cN}_3$.
By Claims~\ref{cl:R2} and \ref{cl:R4-R5},
\begin{equation}
\label{eq:wR2}
\|\wR^2\|_{Y_1}\lesssim \bM_{c,x}(\infty)
(\bM_{c,x}(\infty)+\bM_2(\infty))\la t\ra^{-1}\,,
\quad
\|\wR^2\|_Y\lesssim \bM_{c,x}(\infty)
(\bM_{c,x}(\infty)+\bM_2(\infty))\la t\ra^{-5/4}\,.
\end{equation}
By Claims~\ref{cl:R9} and \ref{cl:R10},
\begin{equation}
\label{eq:wR4}
\|\wR^4\|_{Y_1} \lesssim \bM_{c,x}(\infty)^2\la t\ra^{-1}
\,,\quad
\|\wR^4\|_Y \lesssim \bM_{c,x}(\infty)^2\la t\ra^{-5/4}\,.  
\end{equation}
Combining \eqref{eq:wR2} and \eqref{eq:wR4} with Claim~\ref{cl:b4part1},
we have
\begin{align*}
& \|\obu{\cN}_3\|_{Y_1}+\|\chi(D_y)\obu{\cN}_3\|_{L^1}
\lesssim \bM_{c,x}(\infty)(\bM_{c,x}(\infty)+\bM_2(\infty))^2\la t\ra^{-2}\,,
\\ &
\|\oc{\cN}_3\|_{Y_1}\lesssim \bM_{c,x}(\infty)(\bM_{c,x}(\infty)+\bM_2(\infty))
\la t\ra^{-1}\,.
\end{align*}
\par
Next, we will estimate $\cN^4$.
Let $n_{41}=(B_2-\pd_y^2\wS_0)b_y\mathbf{e_1}$,
$n_{42}=(B_2-\pd_y^2\wS_0)x_{yy}\mathbf{e_2}$ and
\begin{gather*}
\obu{\cN}_{41}=[\obu{B}_{34},\pd_y]n_{41}\,,\quad
\obu{\cN}_{42}=E_2\obu{B}_{34}n_{42}\,,\quad
\obu{\cN}_{43}=E_1\obu{B}_{34}n_{42}\,,
\\
\oc{\cN}_{41}=\obu{B}_{34}n_{41}\,,\quad
\oc{\cN}_{42}=\oc{B}_{34}(\pd_yn_{41}+n_{42})  
\end{gather*}
By the definitions, $E_2\obu{\cN}_{42}=\obu{\cN}_{42}$ and
$$\cN^4=(\obu{B}_{34}-\pd_y^2\oc{B}_{34})(\pd_yn_{41}+n_{42})
=\obu{\cN}_{41}+\obu{\cN}_{42}+\obu{\cN}_{43}
+\pd_y\oc{\cN}_{41}-\pd_y^2\oc{\cN}_{42}\,.
$$
Since $\|\wS_0\|_{B(Y)}=O(1)$ by \cite[Claim~B.1]{Miz15},
we have $\|n_{41}\|_Y+\|n_{42}\|_Y\lesssim \bM_{c,x}(\infty)\la t\ra^{-3/4}$.
It follows from Claim~\ref{cl:b4part1} that
\begin{gather*}
\|\chi(D_y)\obu{\cN}_{41}\|_{L^1}+\|\obu{\cN}_{41}\|_{Y_1} 
\lesssim \bM_{c,x}(\infty)(\bM_{c,x}(\infty)+\bM_2(\infty))\la t\ra^{-3/2}\,,
\\
\|\obu{\cN}_{42}\|_{Y_1}+\|\oc{\cN}_{41}\|_{Y_1}+\|\oc{\cN}_{42}\|_{Y_1}
\lesssim \bM_{c,x}(\infty)(\bM_{c,x}(\infty)+\bM_2(\infty))\la t\ra^{-1}\,.
\end{gather*}
\par
Since $E_1B_1^{-1}\wC_1=O$,
\begin{equation*}
\obu{\cN}_{43}=E_1
\left(\obu{B}_{34}+B_1^{-1}\wC_1B_3^{-1}\right)n_{42}\,.  
\end{equation*}
By Claim~\ref{cl:wS3} and \eqref{eq:B34expr},
\begin{align*}
\|\obu{B}_{34}+\obu{B}_4\wC_1B_3^{-1}\|_{B(Y,Y_1)}
=&
\|\obu{B}_4(\wS_{31}-\bS_{31}-\bS_{41}-\bS_{51})B_3^{-1}\|_{B(Y,Y_1)}
\\ \lesssim &
(\bM_{c,x}(\infty)+\bM_2(\infty))\la t\ra^{-3/4}\,.
\end{align*}
By \eqref{eq:b4part2'} and the above,
\begin{equation}
  \label{eq:B34-1}
\begin{split}
& \|\obu{B}_{34}+B_1^{-1}\wC_1B_3^{-1}\|_{B(Y,Y_1)}
  \\ \le &
\|\obu{B}_{34}+\obu{B}_4\wC_1B_3^{-1}\|_{B(Y,Y_1)}
+\|B_1^{-1}-\obu{B}_4\|_{B(Y_1)}\|\wC_1B_3^{-1}\|_{B(Y,Y_1)}
\\  \lesssim & 
\la t\ra^{-3/4}(\bM_{c,x}(\infty)+\bM_2(\infty))\,.
\end{split}
\end{equation}
Using Claim~\ref{cl:b4part1} and \eqref{eq:obu-exp}, we can prove
\begin{equation}
  \label{eq:B34-2}
\|\chi(D_y)\obu{B}_{34}+\chi(D_y)B_1^{-1}\wC_1B_3^{-1}\|_{B(Y,L^1)}
 \lesssim 
\la t\ra^{-3/4}(\bM_{c,x}(\infty)+\bM_2(\infty))
\end{equation}
in the same way.
By \eqref{eq:B34-1} and \eqref{eq:B34-2},
\begin{equation}
\label{eq:obuN42}
 \bigl\|\obu{\cN}_{43}\bigr\|_{Y_1}
+\bigl\|\chi(D_y)\obu{\cN}_{43}\bigr\|_{L^1}
\lesssim
\la t\ra^{-3/2}\bM_{c,x}(\infty)(\bM_{c,x}(\infty)+\bM_2(\infty))\,.
\end{equation}
\par
Secondly, we will estimate $\cN^5$.
By \eqref{eq:B4B3expr},
$$\cN^5=\obu{\cN}_5-\pd_y^2\oc{\cN}_5\,,\quad
\obu{\cN}_5=\obu{B}_{34}\widetilde{\mathcal{A}}_1(t)
\begin{pmatrix}b \\ \tx\end{pmatrix}\,,
\quad
\oc{\cN}_5=\oc{B}_{34}\widetilde{\mathcal{A}}_1(t)
\begin{pmatrix}b \\ \tx\end{pmatrix}\,.$$
Since $\|\widetilde{\mathcal{A}}_1(t){}^t(b,\tx)\|_Y\lesssim \bM_{c,x}(\infty)
e^{-\a(3t+L)}$, it follows from Claim~\ref{cl:b4part1} that
\begin{equation}
 \label{eq:oN5} 
\bigl\|\obu{\cN}_5\bigr\|_{Y_1}+\bigl\|\chi(D_y)\obu{\cN}_5\bigr\|_{L^1}
+\bigl\|\oc{\cN}_5\bigr\|_{Y_1}
\lesssim e^{-\a(3t+L)}\la t\ra^{-1/4}\bM_{c,x}(\infty)
(\bM_{c,x}(\infty)+\bM_2(\infty))\,.
\end{equation}
\par
Next, we will estimate $\obu{\cN}_6$ and $\oc{\cN}_6$.
By Claim~\ref{cl:k-growth}, \eqref{eq:S311-d} and \eqref{eq:a3-a4},
\begin{align*}
 \bigl\|\obu{\cN}_{61}\bigr\|_{Y_1}+\bigl\|\chi(D_y)\obu{\cN}_{61}\bigr\|_{L^1}
\lesssim & e^{-\a(3t+L)}\la t\ra \eps\,.
\end{align*}
We see that $\obu{\cN}_{62}=E_2\obu{\cN}_{62}$ and that
\begin{align*}
\bigl\|\obu{\cN}_{62}\bigr\|_Y
\lesssim & \|R^{v_1}\|_Y+\|k\|_Y \lesssim  \bM_1(\infty)\la t\ra^{-2}
\end{align*}
follows from Claim~\ref{cl:k-decay} and \eqref{eq:Rv1-est}.
By \eqref{eq:B34Rv1}, \eqref{eq:obuRv1} and \eqref{eq:S7-bound},
\begin{align*}
\bigl\|\chi(D_y)\obu{\cN}_{63}\bigr\|_{L^1}+\bigl\|\obu{\cN}_{63}\bigr\|_{Y_1}
 \lesssim & 
\bM_1(\infty)(\bM_{c,x}(\infty)+\bM_1(\infty)+\bM_2(\infty))\la t\ra^{-3/2}\,.
\end{align*}
By Claims~\ref{cl:b4part1}, \eqref{eq:B4Rv1} and \eqref{eq:ocRv1},
\begin{align*}
\|\oc{\cN}_{61}\|_Y \lesssim &  \|\oc{R}_{v_1,1}\|_Y
 \lesssim 
\bM_{c,x}(\infty)\bM_1(\infty)\la t\ra^{-7/4}\,,
\end{align*}
\begin{align*}
\|\oc{\cN}_{62}\|_Y \lesssim
 & \|\oc{R}_{v_1,2}\|_Y +\|(\oc{B}_{14}+\oc{B}_{34})R^{v_1}\|_Y
\lesssim \bM_1(\infty)\la t\ra^{-1}\,. 
\end{align*}
\par
Finally, we will estimate $\obu{\cN}_7$ and $\oc{\cN}_7$.
By the definition and \eqref{eq:a3-a4},
we have $\obu{\cN}_7=E_2\obu{\cN}_7$ and 
$$\|\obu{\cN}_7(t)\|_Y \lesssim 
\left(\bM_{c,x}(\infty) +\bM_1(\infty)\right)
e^{-\a(3t+L)}\la t\ra^{-1/4}\,.$$
In view of Claim~\ref{cl:cx_t-bound},
$$\|\tx(t)\|_Y\lesssim 
(\bM_{c,x}(\infty)+\bM_1(\infty)+\bM_2(\infty)^2)\la t\ra^{1/4}\,.$$
Combining the above with Claim~\ref{cl:k-decay},  \eqref{eq:a3-a4}
and \eqref{eq:A1-A3-est},
$$\|\oc{\cN}_7(t)\|_Y \lesssim
\left\{(\eta_0+e^{-\a L})\bM_{c,x}(\infty) +\bM_1(\infty)+\bM_2(\infty)^2\right\}
\la t\ra^{-1}\,.$$
This completes the proof of Theorem~\ref{thm:phase}.
\end{proof}

Next, we will prove Corollary~\ref{cor:instability2}.
\begin{proof}[Proof of Corollary~\ref{cor:instability2}]
  Let $\zeta\in C_0^\infty(-\eta_0,\eta_0)$ such that $\zeta(0)=1$ and let
$$u(0,x,y)=\varphi_{2+c_*(y)}(x)-\psi_{2+\tc_*(y),L}(x)\,,\quad
\tc_*(y)=\eps (\mathcal{F}_\eta^{-1}\zeta)(y)\,.$$
Then it follows from \cite[Lemma~5.2]{Miz15} that
$$\tc(0,y)=\tc_*(y)\,,\quad \tx(0,y)\equiv0\,,\quad
b_1(0,y)=b_*(y)\,,\quad  v_*=v_{2,*}=0\,.$$
Since $\|b_*-\tc_*\|_{L^1}\lesssim \|\tc_*(0)\|_Y^2$, we see that
$b_*\in L^1(\R)$ and that
\begin{align*}
\int_\R b_1(0,y)\,dy\ge & \int_\R \tc_*(y)\,dy- \|b_*-\tc_*\|_{L^1}
=\frac{\eps}{\sqrt{2\pi}}+O(\eps^2)\,.
\end{align*}
If $\eps$ and $e^{-\a L}$ are sufficiently small, then
it follows from
\eqref{eq:phase-hom}, \eqref{eq:phase-Na1}--\eqref{eq:phase-Na4}
and the above that
\begin{equation}
 \label{eq:phase-lb}
h \gtrsim \inf_{t\ge0}\tx(t,0)\gtrsim
\inf_{t\ge0}b_2(t,0) \gtrsim \eps\,,
\end{equation}
where $h$ is a constant in \eqref{eq:phase-lim}.
Corollary~\ref{cor:instability2} follows immediately from
\eqref{eq:phase-lb} and Theorem~\ref{thm:phase}. Thus we complete the proof.
\end{proof}

\bigskip

\section{Behavior of  the local amplitude and the local inclination
of line solitons}
In this section, we will prove Theorem~\ref{thm:burgers} following a compactness argument in \cite{Karch}.
\par
Let $\bb(t,\cdot)=\gamma(t)\mathcal{P}_*(D_y)e^{4t\sigma_3\pd_y}\bd(t,\cdot)$
and 
$$\Pi_*(\eta)=\frac{1}{4i}
\begin{pmatrix}
8i & 8i \\ \eta+i\omega(\eta) & \eta-i\omega(\eta)
\end{pmatrix}=
\begin{pmatrix}
  1 & 0 \\ 0 & i\eta
\end{pmatrix}
\mathcal{P}_*(\eta),.$$
Then \eqref{eq:modeq3} is translated to
\begin{equation}
  \label{eq:modeq4}
\pd_t \bd= \{2\pd_y^2I+\pd_y\tilde{\omega}(D_y)\sigma_3\}\bd
+\widetilde{\cN}_a+\pd_y(\widetilde{\cN}+\widetilde{\cN}')
+\pd_y^2\widetilde{\cN}''\,,
\end{equation}
where $\sigma_3=\diag(1,-1)$, $\widetilde{\cN}_a=e^{-4t\sigma_3\pd_y}\Pi_*(D_y)^{-1}E_1\chi(D_y)\obu{\cN}_a$ and 
 \begin{align*}
&  \widetilde{\cN}=e^{-4t\sigma_3\pd_y}\Pi_*(D_y)^{-1}
\begin{pmatrix} 6bx_y \\ 2(\tc-b)+3(x_y)^2\end{pmatrix}\,,
\\ &
\widetilde{\cN}'=e^{-4t\sigma_3\pd_y}\Pi_*(D_y)^{-1}\left\{
\pd_y^{-1}(I-\chi(D_y))E_1\obu{\cN}_a+
E_2\obu{\cN}_a+\obu{\cN}_b+
\diag(1,\pd_y)\left(\oc{\cN}_a+\oc{\cN}_b\right)\right\}\,,
\\ &
\widetilde{\cN}''=e^{-4t\sigma_3\pd_y}\Pi_*(D_y)^{-1}\diag(1,\pd_y)\oc{\cN}_c\,.
\end{align*}
Note that $\chi(\eta)=1$ for $\eta\in[-\eta_0/4,\eta_0/4]$
and that $\diag(1,\pd_y)\obu{\cN}_b=\pd_y\obu{\cN}_b$.
\par
By \eqref{eq:omega-approx}, we have for $\eta\in[-\eta_0,\eta_0]$,
\begin{gather}
\label{eq:Pi-est}
\left|\Pi_*(\eta)
-\begin{pmatrix}2 & 2\\ 1 & -1 \end{pmatrix}\right|
+\left|\Pi_*(\eta)^{-1}-\frac{1}{4}
\begin{pmatrix}1 & 2 \\ 1 & -2 \end{pmatrix}\right|
\lesssim |\eta|\,.
\end{gather}
If $\eta_0$ is sufficiently small, then $\Pi_*(D_y)$, $\Pi_*^{-1}(D_y)\in B(Y)$
and it follows from Claim~\ref{cl:k-decay} and the definitions of $\bb$
and $\bd$ that
\begin{equation}
  \label{eq:bx-bb} 
  \begin{split}
\left\|\begin{pmatrix}  b(t,\cdot) \\ x_y(t,\cdot)\end{pmatrix}
-\begin{pmatrix}2 & 2 \\ 1 & -1 \end{pmatrix}e^{4t\sigma_3\pd_y}\bd(t,\cdot)
\right\|_Y
& \lesssim \|k(t,\cdot)\|_Y+\|\pd_y\bd(t,\cdot)\|_Y
 \lesssim  \eps\la t\ra^{-3/4}\,.
  \end{split}
\end{equation}
\par
To investigate the asymptotic behavior of solutions,
we consider the rescaled solution
$\bd_\lambda(t,y)=\lambda\bd(\lambda^2t,\lambda y)$.
We will show that for any $t_1$ and $t_2$ satisfying $0<t_1<t_2<\infty$,
\begin{equation}
\label{eq:bd-c}
\lim_{\lambda\to\infty}\sup_{t\in[t_1,t_2]}
\|\bd_\lambda(t,y)-\bd_\infty(t,y)\|_{L^2(\R)}=0\,,
\end{equation}
where $\bd_\infty(t,y)={}^t(d_{\infty,+}(t,y), d_{\infty,-}(t,y))$ 
and $d_{\infty,\pm}(t,y)$ are self-similar
solutions of Burgers equations 

\begin{equation}
\label{eq:Burgers}
  \left\{
    \begin{aligned}
& \pd_td_+=2\pd_y^2d_++4\pd_y(d_+^2)\,,  \\
& \pd_td_-=2\pd_y^2d_--4\pd_y(d_-)^2\,. 
    \end{aligned}
\right.
\end{equation}
satisfying
\begin{equation}
\label{eq:ss}
\lambda\bd_\infty(\lambda^2t,\lambda y)=\bd_\infty(t,y)
\quad\text{for every $\lambda>0$.}  
\end{equation}
\par
First, we will show uniform boundedness of $\bd_\lambda$ with respect to
$\lambda\ge1$.
\begin{lemma}
\label{lem:tbb-bound}
Let $\eps$ be as in Theorem~\ref{thm:burgers}. 
Then there exists a positive constants $C$ such that
for any $\lambda\ge1$ and $t\in(0,\infty)$,
\begin{gather}
  \label{eq:tbb-bound1}
\sum_{k=0,1}\|\pd_y^k\bd_\lambda(t,\cdot)\|_{L^2}\le C\eps t^{-(2k+1)/4} \,,
\quad \|\pd_y^2\bd_\lambda(t,\cdot)\|_{L^2}\le C\eps\lambda^{1/2}t^{-1}\,,
\\  \label{eq:tbb-bound2}
\left\|\pd_t\bd_\lambda(t,\cdot)\right\|_{H^{-2}}
\le C(t^{-1/4}+t^{-3/2})\eps\,.
\end{gather}
\end{lemma}
\begin{proof}
The proof follows the line of the proof of \cite[Lemma~12.1]{Miz15}.
By Proposition~\ref{prop:poly} and \eqref{eq:bx-bb},
\begin{equation}
\label{eq:bb-bound}
\sum_{k=0,1}\la t\ra^{(2k+1)/4}\|\pd_y^k\bd(t)\|_Y
+\la t\ra\|\pd_y^2\bd(t)\|_Y\lesssim \eps
\quad\text{for every $t\ge0$,}  
\end{equation}
and \eqref{eq:tbb-bound1} follows immediately from \eqref{eq:bb-bound}.
\par
Next, we will prove \eqref{eq:tbb-bound2}.
By \eqref{eq:modeq4},
\begin{align*}
\pd_t\bd_\lambda
=& 
2\pd_y^2\bd_\lambda+\lambda\sigma_3\pd_y\tilde{\omega}(\lambda^{-1}D_y)\bd_\lambda
+\widetilde{\cN}_{a,\lambda}
+\pd_y(\widetilde{\cN}_\lambda+\widetilde{\cN}'_\lambda)
+\pd_y^2\widetilde{\cN}''_\lambda\,,   
  \end{align*}
where $\widetilde{\cN}_{a,\lambda}(t,y)
=\lambda^3\widetilde{\cN}_a(\lambda^2t,\lambda y)$ and
$$
\widetilde{\cN}_\lambda(t,y)
=\lambda^2\widetilde{\cN}(\lambda^2t,\lambda y)\,,
\enskip
\widetilde{\cN}'_\lambda(t,y)
=\lambda^2\widetilde{\cN}'_\lambda(\lambda^2t,\lambda y)\,,
\enskip
\widetilde{\cN}''_\lambda(t,y)
=\lambda\widetilde{\cN}''_\lambda(\lambda^2t,\lambda y)\,.$$
In view of \eqref{eq:omega-approx} and \eqref{eq:tbb-bound1},
\begin{align*}
\|\pd_y^2\bd_\lambda(t,\cdot)\|_{H^{-2}}+\|\lambda\pd_y\tilde{\omega}(\lambda^{-1}D_y)\bd_\lambda(t,\cdot)\|_{H^{-2}}
 \lesssim \|\bd_\lambda(t,\cdot)\|_{L^2}
\lesssim & \eps t^{-1/4}\,.
\end{align*}
\par
Using \eqref{eq:cNa-est}--\eqref{eq:cNbc-est}, \eqref{eq:Pi-est}
and the fact that $Y_1\subset Y$, we have
\begin{equation}
\label{eq:wcNa-est}
\|\widetilde{\cN}_a\|_{L^1}\lesssim 
\left(e^{-\a L}\eps+\eps^2\right)\la t\ra^{-3/2}\,,
\end{equation}
\begin{equation*}
\|\widetilde{\cN}\|_Y\lesssim \eps^2\la t\ra^{-3/4}\,,
\quad
\|\widetilde{\cN}'\|_Y\lesssim \eps^2\la t\ra^{-1}\,,
\quad
\|\widetilde{\cN}''\|_Y\lesssim \eps\la t\ra^{-1}\,,
\end{equation*}
and 
\begin{align}
& \label{eq:wcNa2}
\|\widetilde{\cN}_{a,\lambda}(t,\cdot)\|_{L^1}
=\lambda^2\|\widetilde{\cN}_a(\lambda^2t,\cdot)\|_Y
\lesssim 
\left(e^{-\a L}\eps+\eps^2\right)\lambda^{-1}t^{-3/2}\,,
\\ & \label{eq:wcN2}
 \|\widetilde{\cN}_\lambda(t,\cdot)\|_{L^2}
=  \lambda^{3/2}\|\widetilde{\cN}(\lambda^2t,\cdot)\|_Y
\lesssim \eps^2t^{-3/4}\,,
\\ & \label{eq:wcN'2}
\|\widetilde{\cN}'_\lambda(t,\cdot)\|_{L^2}
=\lambda^{3/2}\|\widetilde{\cN}_1'(\lambda^2t,\cdot)\|_Y
\lesssim \eps\lambda^{3/2}(1+\lambda^2t)^{-1}
\lesssim \eps\lambda^{-1/4}t^{-7/8}\,,
\\ & \label{eq:wcN''2}
\|\widetilde{\cN}''_\lambda(t,\cdot)\|_{L^2}
=\lambda^{1/2}\|\widetilde{\cN}''(\lambda^2t,\cdot)\|_Y
\lesssim \eps\lambda^{1/2}(1+\lambda^2t)^{-1}
\lesssim \eps\lambda^{-1/2}t^{-1/2}\,.
\end{align}
Combining the above, we have \eqref{eq:tbb-bound2}.
Thus we complete the proof.
\end{proof}
Using standard compactness argument along with the Aubin-Lions lemma,
we have the following.
\begin{corollary}
\label{cor:tbb-bound}
There exists a sequence $\{\lambda_n\}_{n\ge1}$ satisfying
$\lim_{n\to\infty}\lambda_n=\infty$ and $\bd_\infty(t,y)$ such that
\begin{align*}
& \bd_{\lambda_n}(t,\cdot) \to \bd_\infty(t,\cdot)
\quad\text{weakly star in $L^\infty_{loc}((0,\infty);H^1(\R))$,}\\
& \pd_t\bd_{\lambda_n}(t,\cdot)\to \pd_t\bd_\infty(t,\cdot)
\quad\text{weakly star in $L^\infty_{loc}((0,\infty);H^{-2}(\R))$,}\\
\end{align*}
\begin{equation}
  \label{eq:condition}
 \sup_{t>0}t^{1/4}\|\bd_\infty(t)\|_{L^2}\le C\eps\,,
\end{equation}
where $C$ is a constant given in Lemma~\ref{lem:tbb-bound}.
Moreover, for any $R>0$ and $t_1$, $t_2$ with $0<t_1\le t_2<\infty$,
\begin{equation}
  \label{eq:aubli}
\lim_{n\to\infty}\sup_{t\in[t_1,t_2]}
\|\bd_{\lambda_n}(t,\cdot)-\bd_\infty(t,\cdot)\|_{L^2(|y|\le R)}=0\,.  
\end{equation}
\end{corollary}

Next, we will show that $\bd_\infty(t,y)$ tends
to a constant multiple of the delta function as $t\downarrow0$.  To
find initial data of $\bd(0,\cdot)$, we transform \eqref{eq:modeq4}
into a conservative system.  Let
$$ \tbd(t,y)=\begin{pmatrix} \td_+(t,y)\\ \td_-(t,y)\end{pmatrix}
:=\bd(t,y)-\bar{\bd}(t,y)\,,\quad
\bar{\bd}(t,y)=-\int_t^\infty \widetilde{\cN}_a(s,\cdot)\,ds\,.
$$
Then
\begin{equation}
  \label{eq:modeq5}
\pd_t\tbd=2\pd_y^2\tbd+\pd_y(\widetilde{\cN}+\widetilde{\cN}')
+\pd_y^2(\widetilde{\cN}''+\widetilde{\cN}''')\,,
\end{equation}
where $\widetilde{\cN}'''=2\bar{\bd}+\pd_y^{-1}\tilde{\omega}(D_y)\sigma_3\bd$.
\begin{lemma}
  \label{lem:Burgers-ini}
\begin{equation}
  \label{eq:B-ini}
\lim_{t\downarrow0}\int_\R \bd_\infty(t,y)h(y)\,dy
=h(0)\int_\R\tbd(0,y)\,dy
\quad\text{for any $h\in H^2(\R)$.}
\end{equation}
\end{lemma}
\begin{proof}
Let $\tbd_\lambda(t,y)=\lambda\tbd(\lambda^2t,\lambda y)$
and $\bar{\bd}_\lambda(t,y)=\lambda\bar{\bd}(\lambda^2t,\lambda y)$.
By \eqref{eq:wcNa-est}, \eqref{eq:wcNa2} and the fact that
$\|\bar{\bd}(t,\cdot)\|_Y\lesssim \|\bar{\bd}(t,\cdot)\|_{L^1}$,
\begin{gather}
\label{eq:barbb-est}
\bigl\|\bar{\bd}(t)\bigr\|_{L^1}+\bigl\|\bar{\bd}(t)\bigr\|_Y
\lesssim  \left(e^{-\a L}\eps+\eps^2\right)\la t \ra^{-1/2}\,,
\\
  \label{eq:barbb-est2}
\|\bar{\bd}_\lambda(t,\cdot)\|_{L^2}=\lambda^{1/2}\|\bar{\bd}(\lambda^2t,\cdot)\|_Y
\lesssim  \left(e^{-\a L}\eps+\eps^2\right)\lambda^{-1/2}t^{-1/2}\,.
\end{gather}  
Hence the limiting profile of $\bd_\lambda(t)$ and $\tbd_\lambda(t)$
as $\lambda\to\infty$ are the same for every $t>0$.
\par

By \eqref{eq:modeq5},
\begin{equation*}
\pd_t\tbd_\lambda
=2\pd_y^2\tbd_\lambda
+\pd_y(\widetilde{\cN}_\lambda+\widetilde{\cN}'_\lambda)
+\pd_y^2(\widetilde{\cN}''_\lambda+\widetilde{\cN}'''_\lambda)\,,
\end{equation*}
where $\widetilde{\cN}'''_\lambda=2\bar{\bd}_\lambda+
\lambda\pd_y^{-1}\tilde{\omega}(\lambda^{-1}D_y)\bd_\lambda$.
By  Lemma~\ref{lem:tbb-bound} and \eqref{eq:omega-approx},
\begin{equation}
  \label{eq:wcN'''-est}
\left\|2\tbd_\lambda+\widetilde{\cN}'''_\lambda\right\|_{L^2}
\lesssim  \|\bd_\lambda\|_{L^2}\lesssim  \eps t^{-1/4}\,.
\end{equation}
Combining the above with \eqref{eq:wcN2}--\eqref{eq:wcN''2},
we have
$$
\sup_{\lambda\ge1}
\|\pd_t\tbd_\lambda(t,\cdot)\|_{H^{-2}}
\lesssim \eps(t^{-1/4}+t^{-7/8})\,.$$
Thus for $t>s>0$ and $h\in H^2(\R)$,
$$
\left|\int_\R \tbd_\lambda(t,y)h(y)\,dy
-\int_\R \tbd_\lambda(s,y)h(y)\,dy\right|
\le C\left\{(t-s)^{3/4}+(t-s)^{1/8}\right\}\,,$$
where $C$ is a constant independent of $\lambda$.
Passing to the limit as $s\downarrow0$ in the above, we obtain for $t>0$,
\begin{equation}
  \label{eq:tbd-0}
  \left|\int_\R\tbd_\lambda(t,y)h(y)\,dy
-\int_\R\tbd_\lambda(0,y)h(y)\,dy\right|\le C(t^{3/4}+t^{1/8})\,.
\end{equation}
Since $\tbd(0,\cdot)\in L^1(\R)+\pd_yY_1$,
it follows from Lebesgue's dominated convergence theorem that
as $\lambda\to\infty$,
\begin{align*}
\int_\R \tbd_\lambda(0,y)h(y)\,dy
=&  \int_\R \mF_y\tbd(0,\lambda^{-1}\eta)\mF_y^{-1}h(\eta)\,d\eta
\to  \sqrt{2\pi}\mF_y\tbd(0,0)h(0)\,.
\end{align*}
On the other hand,
Corollary~\ref{cor:tbb-bound} and \eqref{eq:barbb-est2} imply that
for any $t>0$ and $h\in L^2(\R)$,
$$
\lim_{n\to\infty}
\int_\R \tbd_{\lambda_n}(t,y)h(y)\,dy=
\int_\R \bd_\infty(t,y)h(y)\,dy\,.$$
This completes the proof of Lemma~\ref{lem:Burgers-ini}.
\end{proof}

Now we will improve \eqref{eq:aubli} to show \eqref{eq:bd-c}.
\begin{lemma}
\label{lem:outer-region}
Suppose that $\eps$ is sufficiently small.
Then for every $t_1$ and $t_2$ satisfying $0<t_1\le t_2<\infty$,
there exist a positive constant $C$ and a function
$\tilde{\delta}(R)$ satisfying $\lim_{R\to\infty}\tilde{\delta}(R)=0$
such that
$$\sup_{t\in[t_1,t_2]}\left\|\bd_\lambda(t,\cdot)\right\|_{L^2(|y|\ge R)}
\le C(\tilde{\delta}(R)+\lambda^{-1/4})
\quad\text{for  $\lambda\ge 1$.}$$
\end{lemma}

\begin{proof}
Let $\zeta$ be a smooth function such that $\zeta(y)=0$ if $|y|\le1/2$ and $\zeta(y)=1$ if $|y|\ge 1$
and  $\zeta_R(y)=\zeta(y/R)$.
Multiplying \eqref{eq:modeq5} by $\zeta_R$, we have
\begin{equation}
  (\pd_t-2\pd_y^2)(\zeta_R\tbd_\lambda)=
\pd_y\{\zeta_R(\widetilde{\cN}_\lambda+\widetilde{\cN}_\lambda')\}
+\pd_y^2\{\zeta_R(\widetilde{\cN}_\lambda''+\widetilde{\cN}'''_\lambda)\}
-\widetilde{\cN}_R\,,
\end{equation}
where $\widetilde{\cN}_R=\widetilde{\cN}_{R,1}+\widetilde{\cN}_{R,2}$,
$\widetilde{\cN}_{R,1}=[\pd_y,\zeta_R](\widetilde{\cN}_\lambda+\widetilde{\cN}_\lambda')$
and $\widetilde{\cN}_{R,2}=[\pd_y^2,\zeta_R](2\tbd_\lambda+\widetilde{\cN}_\lambda''+\widetilde{\cN}_\lambda''')$.
Using the variation of constants formula, we have
$$
\zeta_R\tbd_\lambda(t)=e^{2t\pd_y^2}\zeta_R\tbd(0)
+\sum_{j=1}^6IV_j\,,
$$
\begin{align*}
& IV_1=\int_0^t e^{2(t-\tau)\pd_y^2}
\pd_y(\zeta_R\widetilde{\cN}_\lambda(\tau))\,d\tau\,,
\quad 
IV_2=\int_0^t e^{2(t-\tau)\pd_y^2}
\pd_y(\zeta_R\widetilde{\cN}_\lambda'(\tau))\,d\tau\,,
\\ & 
IV_3=\int_0^t e^{2(t-\tau)\pd_y^2}
\pd_y^2(\zeta_R\widetilde{\cN}_\lambda''(\tau))\,d\tau\,,
\quad
IV_4=\int_0^t e^{2(t-\tau)\pd_y^2}
\pd_y^2(\zeta_R\widetilde{\cN}_\lambda'''(\tau))\,d\tau\,,
\\ &
IV_5=-\int_0^t e^{2(t-\tau)\pd_y^2} \widetilde{\cN}_{R,1}(\tau)\,d\tau\,,
\quad
IV_6=-\int_0^t e^{2(t-\tau)\pd_y^2} \widetilde{\cN}_{R,2}(\tau)\,d\tau\,.
\end{align*}
\par
By Lemma~\ref{lem:modeq-init-decomp}, \eqref{eq:cx*bound}, \eqref{eq:modeq-init} and \eqref{eq:barbb-est},
we can decompose $\tbd(0)$ as 
\begin{equation}
  \label{eq:d0}
\tbd(0)=\obu{\bd}_0+\pd_y\oc{\bd}_0\,,
\quad \bigl\|\obu{\bd}_0\bigr\|_{L^1(\R)}+\bigl\|\oc{\bd}_0\bigr\|_{Y_1}
\lesssim \eps\,.
\end{equation}
Let $\obu{\bd}_{0,\lambda}(y)=\lambda\obu{\bd}_0(\lambda y)$ and $\oc{\bd}_{0,\lambda}(y)=\oc{\bd}_0(\lambda y)$.
Then $\tbd_\lambda(0,y)=\obu{\bd}_{0,\lambda}(y)+\pd_y\oc{\bd}_{0,\lambda}(y)$ and
\begin{align*}
\bigl\|e^{2t\pd_y^2}\zeta_R\tbd_\lambda(0)\bigr\|_{L^2}
\lesssim & 
t^{-1/4}\bigl\|\zeta_R\obu{\bd}_{0,\lambda}\bigr\|_{L^1}+t^{-1/2}\bigl\|\oc{\bd}_{0,\lambda}\bigr\|_{L^2}
+\bigl\|[\pd_y,\zeta_R]\oc{\bd}_{0,\lambda}\bigr\|_{L^2}
\\ \lesssim & 
t^{-1/4}\bigl\|\obu{\bd}_0\bigr\|_{L^1(|y|\ge \lambda R)}
+\{R^{-1}+(t\lambda)^{-1/2}\}\bigl\|\oc{\bd}_0\bigr\|_{L^2}\,.
\end{align*}
\par
By Lemma~\ref{lem:tbb-bound} and \eqref{eq:barbb-est2},
\begin{align*}
\|IV_1\|_{L^2}\lesssim &
\int_0^t (t-\tau)^{-3/4} \|\zeta_R\bd_\lambda(\tau)\|_{L^2}
\|\bd_\lambda(\tau)\|_{L^2}\,d\tau
\\ \lesssim & \eps\int_0^t(t-\tau)^{-3/4}\tau^{-1/4}
\|\zeta_R\tbd_\lambda(\tau)\|_{L^2}\,d\tau
+\eps^2\lambda^{-1/2}\int_0^t(t-\tau)^{-3/4}\tau^{-3/4}\,d\tau
\\ \lesssim & \eps\int_0^t(t-\tau)^{-3/4}\tau^{-1/4}
\|\zeta_R\tbd_\lambda(\tau)\|_{L^2}\,d\tau
+\eps^2\lambda^{-1/2}t^{-1/2}\,.
\end{align*}
By \eqref{eq:wcN'2},
\begin{align*}
\|IV_2\|_{L^2}\lesssim & \int_0^t (t-\tau)^{-1/2}
\|\zeta_R\widetilde{\cN}_\lambda'(\tau)\|_{L^2}\,d\tau
\lesssim 
\eps\lambda^{-1/4}\int_0^t (t-\tau)^{-1/2}\tau^{-7/8}\,d\tau 
\lesssim \eps\lambda^{-1/4}t^{-3/8}\,.
\end{align*}
Since $\mF_y\widetilde{\cN}_\lambda''(t,\eta)=0$ for $\eta\not\in[-\lambda\eta_0,\lambda\eta_0]$ ,
it follows from \eqref{eq:wcN''2} that
$$\|\widetilde{\cN}_\lambda''(\tau,\cdot)\|_{H^{1/4}} \lesssim \lambda^{1/4}\|\widetilde{\cN}_\lambda''(\tau,\cdot)\|_{L^2}
\lesssim \eps\lambda^{-1/4}\tau^{-1/2}\,,$$
\begin{align*}
\|IV_3\|_{L^2}\lesssim & \eps\lambda^{-1/4}\int_0^t (t-\tau)^{-7/8}\tau^{-1/2}\,ds \lesssim \eps\lambda^{-1/4}t^{-3/8}\,.
\end{align*}
\par
Using Lemma~\ref{lem:tbb-bound}, \eqref{eq:omega-approx} and
\eqref{eq:barbb-est2}, we have
\begin{align*}
\|\widetilde{\cN}_\lambda'''\|_{H^{1/4}} \lesssim & 
\|\bar{\bd}_\lambda\|_{H^{1/4}}+
\lambda\|\pd_y^{-1}\tilde{\omega}(\lambda^{-1}D_y)\bd_\lambda\|_{H^{1/4}}
\\ \lesssim &
\lambda^{1/4}\|\bar{\bd}_\lambda\|_{L^2}+\lambda^{-1/4}\|\bd_\lambda\|_{H^{1/2}}
\lesssim \eps\lambda^{-1/4}t^{-1/2}\,,
 \end{align*}
and
\begin{align*}
\|IV_4\|_{L^2} \lesssim & \int_0^t (t-\tau)^{-7/8}
\left\|\zeta_R\widetilde{\cN}'''_\lambda(\tau)\right\|_{H^{1/4}}\,d\tau
\\ \lesssim  &
\eps\lambda^{-1/4} \int_0^t (t-\tau)^{-7/8}\tau^{-1/2} \,d\tau
\lesssim \eps\lambda^{-1/4}t^{-3/8}\,.
\end{align*}

By \eqref{eq:wcN2} and \eqref{eq:wcN'2},
\begin{align*}
  \|IV_5\|_{L^2}\lesssim  \int_0^t \|\pd_y\zeta_R\|_{L^\infty}
\left(\left\|\widetilde{\cN}_\lambda(\tau)\right\|_{L^2}+\left\|\widetilde{\cN}'_\lambda(\tau)\right\|_{L^2}\right)\,d\tau
\lesssim & \frac{\eps}{R}(t^{1/4}+t^{1/8})\,.
\end{align*}
By \eqref{eq:tbb-bound1}, \eqref{eq:wcN''2}, \eqref{eq:barbb-est2}
and \eqref{eq:wcN'''-est},
\begin{align*}
\|IV_6\|_{L^2} \lesssim &
\int_0^t \{\|\pd_y^2\zeta_r\|_{L^\infty}+(t-\tau)^{-1/2}\|\pd_y\zeta_R\|_{L^\infty}\}
\left(\left\|\widetilde{\cN}''_\lambda(\tau)\right\|_{L^2}+
\left\|2\tbd_\lambda(\tau)+\widetilde{\cN}'''_\lambda(\tau)\right\|_{L^2}\right)
\,d\tau
\\ \lesssim & \frac{\eps}{R}
\int_0^t \{1+(t-\tau)^{-1/2}\}\{(\lambda\tau)^{-1/2}+\tau^{-1/4}\}\,d\tau
 \lesssim \frac{\eps}{R}\la t \ra^{3/4}\,.
\end{align*}
Combining the above, we have for $t\in(0,t_2)$,
\begin{equation*}
\|\zeta_R\bd_\lambda(t)\|_{L^2}
\lesssim 
t^{-1/4}\bigl\|\obu{\bd}_0\bigr\|_{L^1(|y|\ge \lambda R)}
+\frac{\eps}{R}+\eps\lambda^{-1/4}t^{-1/2}
+\eps\int_0^t(t-\tau)^{-3/4}\tau^{-1/4}\|\zeta_R\tbd_\lambda(\tau)\|\,d\tau\,,
\end{equation*}
and if $\eps$ is sufficiently small, 
\begin{equation*}
\sup_{t\in(0,t_2)}t^{1/2}\|\zeta_R\bd_\lambda(t)\|_{L^2}
\lesssim C(t_1,t_2)
\left(\bigl\|\obu{\bd}_0\bigr\|_{L^1(|y|\ge \lambda R)}
+\frac{\eps}{R}+\eps\lambda^{-1/4}\right)\,,
\end{equation*}
where $C(t_2)$ is a constant depending only on $t_2$.
Thus we complete the proof.
\end{proof}

Now we are in position to prove Theorem~\ref{thm:burgers}
\begin{proof}[Proof of Theorem~\ref{thm:burgers}]
Corollary~\ref{cor:tbb-bound} and
Lemma~\ref{lem:outer-region} imply 
$$\lim_{n\to\infty}\sup_{t\in[t_1,t_2]}
\|\bd_{\lambda_n}(t,y)-\bd_\infty(t,y)\|_{L^2(\R)}=0\,,$$
and that $\bd_\infty(t,y)$ is a solutions of \eqref{eq:Burgers}
satisfying $\|\bd_\infty(t,\cdot)\|_{L^2}\le C\eps t^{-1/4}$ for every $t>0$.
\par
Let $m_\pm\in(-2\sqrt{2},2\sqrt{2})$ be constants satisfying
$$\frac{1}{2}\log\left(\frac{2\sqrt{2}\pm m_\pm }{2\sqrt{2}\mp m_\pm}\right)
=\int_\R \td_\pm(0,y)\,dy\,.$$
Then for every $h\in H^1(\R)$,
$$\lim_{t\downarrow0}\int_\R u_B^\pm(t,y)h(y)\,dy
=h(0)\int_\R \td_\pm(0,y)\,dy\,.$$

If $\eps$ is sufficiently small, then solutions of \eqref{eq:Burgers}
satisfying \eqref{eq:condition} and \eqref{eq:B-ini} are unique
(see e.g. \cite[pp.74--75]{Miz15}). Hence it follows that
\begin{equation}
\label{eq:ss-sol}
\bd_\infty(t,y)=\begin{pmatrix}u_B^+(t,y+4t) \\ u_B^-(t,y-4t)\end{pmatrix}\,,
\end{equation}
and that $\bd_\infty(t,y)$ satisfies \eqref{eq:ss}.
Thanks to the uniqueness of the limiting profile $\bd_\infty(t,y)$,
we have \eqref{eq:bd-c}.
\par
By \eqref{eq:bd-c} and \eqref{eq:ss},
\begin{equation}
\label{eq:bd-profile}
  t^{1/4}\|\bd(t,\cdot)-\bd_\infty(t,\cdot)\|_{L^2(\R)}
=\|\bd_{\sqrt{t}}(1,\cdot)-\bd_\infty(1,\cdot)\|_{L^2(\R)}
\to0 \quad\text{as $t\to\infty$,} 
\end{equation}
and Theorem~\ref{thm:burgers} follows immediately 
from \eqref{eq:bx-bb}, \eqref{eq:ss-sol} and \eqref{eq:bd-profile}.
This completes the proof of Theorem~\ref{thm:burgers}.
\end{proof}
\bigskip

\appendix
\section{Operator norms of $S^j_k$}
\label{ap:s}
First, we recall the definitions of operators $S^j_k$ and $\wS^j$
used in \cite{Miz15,Miz17}.
For $q_c=\varphi_c$, $\varphi_c'$, $\pd_c\varphi_c$
and $\pd_z^{-1}\pd_c^m\varphi_c(z)=-\int_z^\infty \pd_c^m\varphi_c(z_1)\,dz_1$
($m\ge1$), let $S_k^1[q_c]$  and $S_k^2[q_c]$ be operators defined by
\begin{align*}
& S_k^1[q_c](f)(t,y)=\frac{1}{2\pi} \int_{-\eta_0}^{\eta_0}\int_{\R^2}
f(y_1)q_2(z)g_{k1}^*(z,\eta,2)e^{i(y-y_1)\eta}\,dy_1dzd\eta\,,\\
& S_k^2[q_c](f)(t,y)=\frac{1}{2\pi} \int_{-\eta_0}^{\eta_0}\int_{\R^2}
f(y_1)\tc(t,y_1)g_{k2}^*(z,\eta,c(t,y_1))
e^{i(y-y_1)\eta}\,d y_1dzd\eta\,,
\end{align*}
where 
\begin{align*}
& \delta q_c(z)=\frac{q_c(z)-q_2(z)}{c-2}\,,\quad
g_{k2}^*(z,\eta,c)=g_{k1}^*(z,\eta,2)\delta q_c(z)+
\frac{g_{k1}^*(z,\eta,c)-g_{k1}^*(z,\eta,2)}{c-2}q_c(z)\,,
\end{align*}
\begin{gather*}
\wS_0=3\begin{pmatrix}-S^1_1[\pd_z^{-1}\pd_c\varphi_c] & S^1_1[\varphi_c]
\\ -S^1_2[\pd_z^{-1}\pd_c\varphi_c] & S^1_2[\varphi_c]\end{pmatrix}\,,
\quad
\wS_j=\begin{pmatrix}-S^j_1[\pd_c\varphi_c]  & S^j_1[\varphi_c']
\\ -S^j_2[\pd_c\varphi_c]  & S^j_2[\varphi_c']\end{pmatrix}
\enskip\text{for $j=1$, $2$.}  
\end{gather*}
Let $S^3_k[p]$ and $S^4_k[p]$ be operators defined by
\begin{equation*}
S^3_k[p](f)(t,y)=\frac{1}{2\pi}\int_{-\eta_0}^{\eta_0}\int_{\R^2}
f(y_1)p(z+3t+L)g_k^*(z,\eta)e^{i(y-y_1)\eta}\,d y_1dzd\eta\,,  
\end{equation*}
\begin{multline*}
S^4_k[p](f)(t,y)=\frac{1}{2\pi}\int_{-\eta_0}^{\eta_0}\int_{\R^2}
f(y_1)\tc(t,y_1)p(z+3t+L)\\ \times g_{k3}^*(z,\eta,c(t,y_1))
e^{i(y-y_1)\eta}\,d y_1dzd\eta\,,
\end{multline*}
where $g_{k3}^*(z,\eta,c)=(c-2)^{-1}(g_k^*(z,\eta,c)-g_k^*(z,\eta))$ and
$p(z)\in C_0^\infty(\R)$.
Let $S^5_k$ and $S^6_k$ be operators defined by
\begin{gather*}
S^5_k(f)(t,y)=\frac{1}{2\pi}\int_{-\eta_0}^{\eta_0}
\int_{\R^2} v_2(t,z,y_1)f(y_1)\pd_cg_k^*(z,\eta,c(t,y_1))
e^{i(y-y_1)\eta}\,dzdy_1d\eta\,,\\
S^6_k(f)(t,y)=-\frac{1}{2\pi}\int_{-\eta_0}^{\eta_0}\int_{\R^2}
v_2(t,z,y_1)f(y_1)\pd_zg_k^*(z,\eta,c(t,y_1))
e^{i(y-y_1)\eta}\,dzdy_1d\eta\,.
\end{gather*}
\par
Let $\wS_3=S^3_1[\psi]E_1+ S^3_2[\psi]E_{21}$,
\begin{align*}
& \wS_4=\begin{pmatrix}
S^3_1[\psi]((\sqrt{2/c}-1)\cdot)+S^4_1[\psi](\sqrt{2/c}\cdot)
& -2(S^3_1[\psi']+S^4_1[\psi'])((\sqrt{2c}-2)\cdot) 
\\ S^3_2[\psi]((\sqrt{2/c}-1)\cdot)+S^4_2[\psi](\sqrt{2/c}\cdot)
& -2(S^3_2[\psi']+S^4_2[\psi'])((\sqrt{2c}-2)\cdot)
\end{pmatrix}\,,
\\ & \wS_5=\begin{pmatrix}  S^5_1 & S^6_1 \\ S^5_2 & S^6_2 \end{pmatrix}\,,
\end{align*}
and $\bS_j=\wS_j(I+\wC_2)^{-1}$ for $1\le j\le 5$, where
$$\cC_2=\wP_1\left\{\left(c(t,\cdot)/2\right)^{1/2}-1\right\}\wP_1\,,
\quad \wC_2=\cC_2 E_1\,.$$
\par

Now we decompose the operator $S^j_k$ $(1\le j\le 6\,,\, k=1\,,\,2)$
into a sum of a time-dependent constant multiple of $\wP_1$ and
an operator which belongs to $\pd_y^2B(Y)$. Let
\begin{equation}
  \label{eq:defS3k1}
\begin{split}
  S^3_{k1}[p](t)f(y)=& \frac{1}{2\pi}\int_{-\eta_0}^{\eta_0}\int_{\R^2}
f(y_1)p(z+3t+L)g_k^*(z,0)e^{i(y-y_1)\eta}\,dy_1dzd\eta
\\ =& \left(\int_{\R^2}p(z+3t+L)g_k^*(z,0)\,dz\right)\wP_1f\,,
\end{split}
\end{equation}
\begin{equation}
  \label{eq:defS3k2}
S^3_{k2}[p](f)(t,y)=\frac{1}{2\pi}\int_{-\eta_0}^{\eta_0}\int_{\R^2}
f(y_1)p(z+3t+L)g_{k1}^*(z,\eta)e^{i(y-y_1)\eta}\,dy_1dzd\eta\,,
\end{equation}
\begin{equation}
\label{eq:defS4k1}
S^4_{k1}[p](f)(t,y)=\wP_1 
\left\{\tc(t,\cdot)f
\int_\R p(z+3t+L)g_{k3}^*(z,0,c(t,\cdot))\,dz\right\}\,,
\end{equation}
\begin{equation}
\label{eq:defS4k2}
\begin{split}
S^4_{k2}[p](f)(t,y)=\frac{1}{2\pi}\int_{-\eta_0}^{\eta_0}\int_{\R^2}
f(y_1)& \tc(t,y_1)p(z+3t+L)\\ & \times g_{k4}^*(z,\eta,c(t,y_1))
e^{i(y-y_1)\eta}\,dy_1dzd\eta\,,
\end{split}  
\end{equation}
where $g_{k4}^*(z,\eta,c)=\eta^{-2}\{g_{k3}^*(z,\eta,c)-g_{k3}^*(z,0,c)\}$
and
\begin{gather}
S^5_{k1}(f)(t,y)=\frac{1}{2\pi}\int_{-\eta_0}^{\eta_0}
\int_{\R^2} v_2(t,z,y_1)f(y_1)\pd_cg_k^*(z,0,c(t,y_1))
e^{i(y-y_1)\eta}\,dz dy_1d\eta\,,\\
S^5_{k2}(f)(t,y)=\frac{1}{2\pi}\int_{-\eta_0}^{\eta_0}
\int_{\R^2} v_2(t,z,y_1)f(y_1)\pd_cg_{k1}^*(z,\eta,c(t,y_1))
e^{i(y-y_1)\eta}\,dz dy_1d\eta\,,\\
S^6_{k1}(f)(t,y)=-\frac{1}{2\pi}\int_{-\eta_0}^{\eta_0}\int_{\R^2}
v_2(t,z,y_1)f(y_1)\pd_zg_k^*(z,0,c(t,y_1))
e^{i(y-y_1)\eta}\,dz dy_1d\eta\,,\\
\label{eq:defS6k2}
S^6_{k2}(f)(t,y)=-\frac{1}{2\pi}\int_{-\eta_0}^{\eta_0}\int_{\R^2}
v_2(t,z,y_1)f(y_1)\pd_zg_{k1}^*(z,\eta,c(t,y_1))
e^{i(y-y_1)\eta}\,dz dy_1d\eta\,.
\end{gather}
Then $S^j_k=S^j_{k1}-\pd_y^2S^j_{k2}$ for $j=3$, $4$, $5$, $6$.

\begin{claim}
  \label{cl:S3}
Let $\a\in(0,2)$. There exist positive constants $C$ and $\eta_1$
such that for $\eta\in(0,\eta_1]$, $k=1$, $2$ and $t\ge0$,
\begin{gather}
\label{eq:clS3-1}
\|\chi(D_y)S^3_{k1}[p](f)(t,\cdot)\|_{L^1}
\le Ce^{-\a(3t+L)}\|e^{\a z}p\|_{L^2}\|f\|_{L^1(\R)}\,,
\\
\label{eq:clS3-2}
\|S^3_{k1}[p](f)(t,\cdot)\|_Y +\|S^3_{k2}[p](f)(t,\cdot)\|_Y
\le Ce^{-\a(3t+L)}\|e^{\a z}p\|_{L^2}\|f\|_{L^2(\R)}\,,
\\ \label{eq:clS3-2'}
\|S^3_{k1}[p](f)(t,\cdot)\|_{Y_1}+\|S^3_{k2}[p](f)(t,\cdot)\|_{Y_1}
\le Ce^{-\a(3t+L)}\|e^{\a z}p\|_{L^2}\|\wP_1f\|_{Y_1}\,.
\end{gather}
\end{claim}
\begin{claim}
\label{cl:S4}
There exist positive constants $\eta_1$, $\delta$ and $C$ such that
if $\eta_0\in(0,\eta_1]$ and $\bM_{c,x}(T)\le \delta$,
then for $k=1$, $2$, $t\in[0,T]$ and $f\in L^2$,
  \begin{gather}
\label{eq:clS4-1}
\|\chi(D_y)S^4_{k1}[p](f)(t,\cdot)\|_{L^1(\R)} \le
Ce^{-\a(3t+L)}\|e^{\a z}p\|_{L^2}\|\tc\|_Y\|f\|_{L^2}\,,
\\ \label{eq:clS4-2}
\|S^4_{k1}[p](f)(t,\cdot)\|_{Y_1}+\|S^4_{k2}[p](f)(t,\cdot)\|_{Y_1} \le
Ce^{-\a(3t+L)}\|e^{\a z}p\|_{L^2}\|\tc\|_Y\|f\|_{L^2}\,.
\end{gather}
\end{claim}
\begin{claim}
  \label{cl:S5}
There exist positive constants $\eta_1$, $\delta$ and $C$ such that
if $\eta_0\in(0,\eta_1]$ and $\bM_{c,x}(T)\le \delta$,
then for $k=1$, $2$, $t\in[0,T]$ and $f\in L^2$,
\begin{gather}
  \label{eq:clS5-1}
\|\chi(D_y)S^5_{k1}(f)(t,\cdot)\|_{L^1(\R)}+\|\chi(D_y)S^6_{k1}(f)(t,\cdot)\|_{L^1(\R)}
\le C\|v_2(t)\|_X\|f\|_{L^2(\R)}\,,
\\ \label{eq:clS5-2}
\sum_{j=5,6}\left( \|S^j_{k1}(f)(t,\cdot)\|_{Y_1}
+\|S^j_{k2}(f)(t,\cdot)\|_{Y_1}\right)
\le C\|v_2(t)\|_X\|f\|_{L^2}\,.
\end{gather}
\end{claim}
\begin{proof}[Proof of Claims~\ref{cl:S3}--\ref{cl:S5}]
Since $\chi(D_y)\wP_1=\chi(D_y)$,
\begin{align*}
\|\chi(D_y)S^3_{k1}[p](f)(t,\cdot)\|_{L^1(\R)}=& \frac{1}{\sqrt{2\pi}}
\left|\int_\R p(z+3t+L)g_k^*(z,0)\,dz\right|
\|\check{\chi}*f\|_{L^1(\R)}
\\ \le &
\|\check{\chi}\|_{L^1(\R)}\|f\|_{L^1(\R)}
\|e^{\a z}p(z+3t+L)\|_{L^2(\R)}\|e^{-\a z}g_k^*(z,0)\|_{L^2(\R)}
\\ \lesssim & e^{-\a(3t+L)}\|f\|_{L^1(\R)}\,.
\end{align*}
Using Young's inequality, we have
\begin{align*}
 \|\chi(D_y)S^5_{k1}(f)(t,\cdot)\|_{L^1(\R)}
=&
\left\|\int_\R \check{\chi}(y-y_1)
f(y_1)\left\{\int_\R v_2(t,z,y_1)\pd_cg_k^*(z,0,c(t,y_1)\,dz\right\}\,dy_1\
\right\|_{L^1(\R)}
\\ \lesssim & 
\|\check{\chi}\|_{L^1}\|f\|_{L^2(\R)}
\left\|\int_\R v_2(t,z,\cdot)\pd_cg_k^*(z,0,c(t,\cdot))\,dz\right\|_{L^2(\R)}
\\ \lesssim &
\|f\|_{L^2(\R)}\|v_2(t)\|_X\sup_{c\in[2-\delta,2+\delta]}
\|e^{-\a z}\pd_cg_k^*(z,0,c)\|_{L^2(\R)}
\\ \lesssim & \|f\|_{L^2(\R)}\|v_2(t)\|_X\,.
\end{align*}
Similarly, we have \eqref{eq:clS4-1} and
$\|\chi(D_y)S^6_{k1}(f)(t,\cdot)\|_{L^1(\R)}\lesssim \|f\|_{L^2(\R)}\|v_2(t)\|_X$.
\par
We can prove \eqref{eq:clS3-2}, \eqref{eq:clS3-2'},
\eqref{eq:clS4-2} and \eqref{eq:clS5-2}
in exactly the same way as the proof of Claims~B.3--B.5 in \cite{Miz15}.
Thus we complete the proof.
\end{proof}

For $\ell=1$, $2$, let
\begin{align}
& \label{eq:defwS3l}
\wS_{3\ell}
=\begin{pmatrix} S^3_{1\ell}[\psi] & 0 \\ S^3_{2\ell}[\psi] & 0\end{pmatrix}\,,
\\ & \label{eq:defwS4l}
\wS_{4\ell}=
\begin{pmatrix}
S^3_{1\ell}[\psi]((\sqrt{2/c}-1)\cdot)+S^4_{1\ell}[\psi](\sqrt{2/c}\cdot)
& -2(S^3_{1\ell}[\psi']+S^4_{1\ell}[\psi'])((\sqrt{2c}-2)\cdot) 
\\ S^4_{1\ell}[\psi]((\sqrt{2/c}-1)\cdot)+S^4_{2\ell}[\psi](\sqrt{2/c}\cdot)
& -2(S^3_{2\ell}[\psi']+S^4_{2\ell}[\psi'])((\sqrt{2c}-2)\cdot)
\end{pmatrix}\,,
\\ &  \label{eq:defwS5l}
 \wS_{5\ell}=\begin{pmatrix}  S^5_{1\ell} & S^6_{1\ell} \\ S^5_{2\ell} & S^6_{2\ell} \end{pmatrix}\,,
\end{align}
and $\bS_{j\ell}=\wS_{j\ell}(1+\wC_2)^{-1}$.
Then $\wS_j=\wS_{j1}-\pd_y^2\wS_{j2}$ and $\bS_j=\bS_{j1}-\pd_y^2\bS_{j2}$
for $j=3$, $4$, $5$.
\begin{claim}
\label{cl:wS3}
There exist positive constants $\eta_1$, $\delta$ and $C$ such that
if $\eta_0\in(0,\eta_1]$ and $\bM_{c,x}(T)\le \delta$,
then for $k=1$, $2$ and $t\in[0,T]$,
  \begin{gather}
\label{eq:clwS3-C}
\|\chi(D_y)\cC_k\|_{B(L^2;L^1)}\le C\bM_{c,x}(T)\la t\ra^{-1/4}\,,\\
\label{eq:clwS3-1}
\|\chi(D_y)\wS_{31}\|_{B(L^1(\R))}+\|\wS_{31}\|_{B(Y_1)} \le Ce^{-\a(3t+L)}\,,\\
\label{eq:clwS3-1'} 
\|\chi(D_y)(\bS_{31}-\wS_{31})\|_{B(L^2(\R),L^1(\R))}
+\|\bS_{31}-\wS_{31}\|_{B(Y,Y_1)}
  \le
C\bM_{c,x}(T)\la t\ra^{-1/4}e^{-\a(3t+L)}\,,
\\ \label{eq:clwS3-2}
\sum_{k=1,2}\left(\|\wS_{3k}\|_{B(Y)\cap B(Y_1)}+\|\bS_{3k}\|_{B(Y)\cap B(Y_1)}
\right)  \le Ce^{-\a(3t+L)}\,,
\\ \label{eq:clwS4-1}
\|\chi(D_y)\wS_{41}\|_{B(L^2(\R),L^1(\R))}+\|\chi(D_y)\bS_{41}\|_{B(L^2(\R),L^1(\R))}
\le C\la t\ra^{-1/4}e^{-\a(3t+L)}\bM_{c,x}(T)\,,
\\  \label{eq:clwS4-2}
\sum_{k=1,2}\left(\|\wS_{4k}\|_{B(L^2(\R),Y_1)}+\|\bS_{4k}\|_{B(L^2(\R),Y_1)}\right)
\le  C\la t\ra^{-1/4}e^{-\a(3t+L)}\bM_{c,x}(T)\,,    
\\
\label{eq:clwS5-1}
\|\chi(D_y)\wS_{51}\|_{B(L^2(\R),L^1(\R))}+\|\chi(D_y)\bS_{51}\|_{B(L^2(\R)),L^1(\R))}
\le C\la t\ra^{-3/4}\bM_2(T)\,,
\\
\label{eq:clwS5-2}
\sum_{k=1,2}\left(\|\wS_{5k}\|_{B(L^2(\R),Y_1)}+\|\bS_{5k}\|_{B(L^2(\R),Y_1)}\right)
\le  C\la t\ra^{-3/4}\bM_2(T)\,.
\end{gather}
\end{claim}
\begin{proof}
 By the definition, we have for $f\in L^2(\R)$,
\begin{align*}
  \|\chi(D_y)\cC_1f\|_{L^1}=& 
\frac{1}{2\sqrt{2\pi}}
\left\|\check{\chi}_1*(c^2-4)\wP_1f\right\|_{L^1(\R)}
\lesssim  \la t\ra^{-1/4}\bM_{c,x}(T)\|f\|_{L^2}\,.
\end{align*}
We can prove 
$\|\chi(D_y)\cC_2f\|_{L^1}\lesssim  \la t\ra^{-1/4}\bM_{c,x}(T)\|f\|_{L^2}$
in the same way.
\par
Equation~\eqref{eq:clwS3-1} follows from Claim~\ref{cl:S3}.
Let $f\in L^2(\R)$ and $f_1=\wC_2f$. 
Since $\chi(\eta)=\chi(\eta)\chi_1(\eta)$,
\begin{align*}
 \chi(D_y)S^3_{k1}[p](f_1)=&
\frac{1}{\sqrt{2\pi}}\int_\R \chi(\eta)\hat{f_1}(\eta)
e^{iy\eta}\,d\eta \left(\int_\R p(z+3t+L)g_k^*(z,0)\,dz\right)
\\=& \chi(D_y)S_{k1}^3[p](\chi_1(D_y)f_1)\,,
\end{align*}
and it follow from Claim~\ref{cl:S3} that
$$\|\chi(D_y)S^3_{k1}[p](f_1)\|_{L^1(\R)} 
 \lesssim e^{-\a(3t+L)}\|e^{\a z}p\|_{L^2(\R)}\|\chi_1(D_y)f_1\|_{L^1(\R)}\,.$$
Combining the above and \eqref{eq:clwS3-C} with $\chi$ replaced by $\chi_1$,
we have for $k=1$, $2$ and  $t\in[0,T]$,
\begin{equation}
  \label{eq:clwS3-pf1}
\|\chi(D_y)S^3_{k1}[p](\wC_2f)\|_{L^1(\R)}\lesssim
\la t\ra^{-1/4}\bM_{c,x}(T) \|f\|_{L^2(\R)}\,.
\end{equation}
Since $\wS_{31}-\bS_{31}=\wS_{31}\wC_2(I+\wC_2)^{-1}$
and $I+\wC_2$ has a bounded inverse on $L^2(\R)$,
\eqref{eq:clwS3-1'} follows immediately
from Claim~\ref{cl:S3} and \eqref{eq:clwS3-pf1}. We can prove
\eqref{eq:clwS4-1} and \eqref{eq:clwS5-1} in the same way.
\par
Equations \eqref{eq:clwS3-2}--\eqref{eq:clwS5-2} follow
from Claims~\ref{cl:S3}, \ref{cl:S4} and \ref{cl:S5}.
\end{proof}

Let $\ell_{2,lin}$ be the linear part of $\ell_{22}+\ell_{23}$ in $\tc$
(see \cite[p.166]{Miz17}),
\begin{equation*}
\ta_k(t,D_y)\tc:=
\frac{1}{2\pi}\int_{-\eta_0}^{\eta_0}\int_{\R^2}
\ell_{2,lin}(t,z,y_1)g_k^*(z,\eta)e^{i(y-y_1)\eta}\,dy_1dzd\eta
\enskip\text{for $k=1$, $2$,}
\end{equation*}
and $\widetilde{\mathcal{A}}_1(t)=\ta_1(t,D_y)E_1+\ta_2(t,D_y)E_{21}$.
More precisely,
\begin{align*}
\ta_k(t,\eta)=& \Bigl[
\int_\R\left\{\pd_z\left(\pd_z^2-1+6\varphi(z)\right)\psi(z+3t+L)\right\}
g_k^*(z,\eta)\,dz\\
& +3\eta^2\int_\R\left(\int_z^\infty\psi(z_1+3t+L)\,dz_1\right)
g_k^*(z,\eta)\,dz\Bigr]\mathbf{1}_{[-\eta_0,\eta_0]}(\eta)\,,
\end{align*}
Let $\widetilde{\mathcal{A}}_{1j}(t)=
\ta_{1j}(t)E_1+\ta_{2j}E_{21}$ for $j=1$, $2$,
where 
\begin{align*}
& \ta_{k1}(t)
=\int_\R\left\{\pd_z\left(\pd_z^2-1+6\varphi(z)\right)\psi(z+3t+L)\right\}
g_k^*(z,0)\,dz\,,
\end{align*}
\begin{align*}
\ta_{k2}(t,\eta)=& \Bigl[
\int_\R\left\{\pd_z\left(\pd_z^2-1+6\varphi(z)\right)\psi(z+3t+L)\right\}
g_{k1}^*(z,\eta)\,dz\\
& +3\int_\R\left(\int_z^\infty\psi(z_1+3t+L)\,dz_1\right)
g_k^*(z,\eta)\,dz\Bigr]\mathbf{1}_{[-\eta_0,\eta_0]}(\eta)\,.
\end{align*}
Then $\widetilde{\mathcal{A}}_1(t)=\widetilde{\mathcal{A}}_{11}(t)
-\pd_y^2\widetilde{\mathcal{A}}_{12}(t)$
and we have the following.
\begin{claim}
\label{cl:akbound}
  There exist  positive constants $C$ and $L_0$ such that
if $L\ge L_0$, then for every $t\ge0$, 
\begin{align}
\label{eq:akbound1}
& \|\chi(D_y)\widetilde{\mathcal{A}}_{11}(t)\|_{B(L^1(\R))}
+\|\widetilde{\mathcal{A}}_{11}(t)\|_{B(Y_1)}\le Ce^{-\a(3t+L)}\,,
\\  \label{eq:akbound2}
& \|\widetilde{\mathcal{A}}_{12}(t)\|_{B(Y)}
+\|\widetilde{\mathcal{A}}_{12}(t)\|_{B(Y_1)}
\le Ce^{-\a(3t+L)}\,.
\end{align}
\end{claim}
Since $\chi(D_y)\wP_1=\chi(D_y)$, $\chi\in C_0^\infty$ and
$\check\chi$ is integrable, we have
\eqref{eq:akbound1}.
Equation~\eqref{eq:akbound2} can be shown in exactly the same way as
\cite[Claims~D.3]{Miz15}.
\bigskip

\section{Estimates of $R^j$}
\label{ap:r}
Let $R^2_k$ be as in \cite[p.~39]{Miz15},
$R^3_k=R^{31}_k-\pd_y^2R^{32}_k$ and
\begin{align*}
R^{31}_k(t,y):=&
\wP_1\int_\R (\ell_{22}+\ell_{23})g_k^*(z,0,c(t,y_1))\,dz
 -\wP_1\int_\R \ell_{2,lin}g_k^*(z,0)\,dz\,,
\end{align*}
\begin{align*}
R^{32}_k(t,y):=&
\frac{1}{2\pi}\int_{-\eta_0}^{\eta_0}\int_{\R^2}
(\ell_{22}+\ell_{23})g_{k1}^*(z,\eta,c(t,y_1))
e^{i(y-y_1)\eta}\,dzdy_1d\eta \\
& -\frac{1}{2\pi}\int_{-\eta_0}^{\eta_0} \int_{\R^2}
\ell_{2,lin}g_{k1}^*(z,\eta)e^{i(y-y_1)\eta}\,dzdy_1d\eta\,.
\end{align*}
Then we have the following.
\begin{claim}\emph{(\cite[Claim~D.1]{Miz15})}
  \label{cl:R2}
There exist positive constants $\delta$ and $C$ such that if
$\bM_{c,x}(T)\le \delta$, then for $t\in[0,T]$,
\begin{equation*}
\|R^2_k(t,\cdot)\|_{Y_1}\le C\bM_{c,x}(T)^2\la t\ra^{-1}\,,\quad
\|\pd_yR^2_k(t,\cdot)\|_{Y_1}\le C\bM_{c,x}(T)^2\la t\ra^{-5/4}\,.
\end{equation*}
\end{claim}
\begin{claim}
  \label{cl:R3}
There exist  positive constants $\delta$ and $C$ such that
if $\bM_{c,x}(T)\le \delta$, then  for $t\in[0,T]$,
\begin{align*}
& \|R^3_k(t,\cdot)\|_{Y_1}+\|R^3_{k2}(t,\cdot)\|_{Y_1} \le Ce^{-\a(3t+L)}\bM_{c,x}(T)^2\,,
\\ & \|R^3_{k1}(t,\cdot)\|_{Y_1}+\|\chi(D_y)R^3_{k1}(t,\cdot)\|_{L^1(\R)}
\le Ce^{-\a(3t+L)}\bM_{c,x}(T)^2\,.
\end{align*}
\end{claim}
We can prove Claim~\ref{cl:R3} in exactly the same way
as Claim~D.2 in \cite{Miz15}. Note that $\chi(D_y)\wP_1=\chi(D_y)$
and $\chi(D_y)\in B(L^1(\R))$.
\par
In this paper,  we slightly modify the definitions of $R^4_k$ and $R^5_k$
from \cite{Miz15,Miz17}. We move $II^1_{k1}$ into $R^5_k$ from $R^4_k$.
Let
\begin{align}
\label{eq:defR4}
R^4_k(t,y)=& \frac{1}{2\pi}\int_{-\eta_0}^{\eta_0}
\left\{II^1_{k2}(t,\eta)+II^1_{k3}(t,\eta)+II^2_k(t,\eta)+ II^3_{k1}(t,\eta)\right\}e^{iy\eta}\,d\eta\,,
\\ \label{eq:defR5}
R^5_k(t,y)=& \frac{1}{2\pi}\int_{-\eta_0}^{\eta_0}
\left\{II^1_{k1}(t,\eta)+II^3_{k2}(t,\eta)\right\}e^{iy\eta}\,d\eta\,.
\end{align}
See \cite[p.166]{Miz17} for the definitions of $II^3_{kj}$.
For the definitions of $II^1_{kj}$, replace $v(t,z,y)$ by $v_2(t,z,y)$
in $II^1_{kj}$ defined in the proof of \cite[Claim~D.5]{Miz15}.
We decompose $R^4_k$ further.
For $j$, $k$, $\ell=1$, $2$, 
\begin{align*}
& h_{jk1}(t,y)=\int_\R v_2(t,z,y)
\left(\int_{-\infty}^z\pd_c^jg_k^*(z_1,0,c(t,y))\,dz_1\right)\,dz\,,
\\ &
h_{jk2}(t,y,\eta)=\int_\R v_2(t,z,y)
\left(\int_{-\infty}^z\pd_c^jg_{k1}^*(z_1,\eta,c(t,y))\,dz_1\right)\,dz\,,
\\ & 
II^1_{k2\ell}(t,\eta)=3\int_\R c_{yy}(t,y)h_{1k\ell}(t,y)e^{-iy\eta}\,dy\,,
\quad
II^1_{k3\ell}(t,\eta)=3\int_\R c_y(t,y)^2h_{2k\ell}(t,y)e^{-iy\eta}\,dy\,,
\end{align*}
Then $II^1_{k2}+II^1_{k3}=\obu{II}_{1,k}+\eta^2\oc{II}_{1,k}$,
$\obu{II}_{1,k}=II^1_{k21}+II^1_{k31}$, $\oc{II}_{1,k}=II^1_{k22}+II^1_{k32}$ and
\begin{gather*}
\|\mF^{-1}\obu{II}_{1,k}(t,\cdot)\|_{L^1(\R)}+
\|\oc{II}_{1,k}(t,\cdot)\|_{Z_1}\lesssim \la t\ra^{-3/2}\bM_{c,x}(T)\bM_2(T)
\quad\text{for $t\in[0,T]$.}
\end{gather*}
Let
\begin{align*}
& II^2_{k1}= -\int_{\R^2} N_{2,1}\pd_zg_k^*(z,0,c(t,y))
e^{-iy\eta}\, dzdy\,,\\
II^2_{k2}=& -\int_{\R^2} N_{2,1}\pd_zg_{k1}^*(z,\eta,c(t,y))
e^{-iy\eta}\, dzdy\,,
\\ &
II^3_{k11}= -3\int_{\R^2}v_2(t,z,y)x_{yy}(t,y)
g_k^*(z,0,c(t,y))e^{-iy\eta}\,dzdy\,,
\\ &
II^3_{k12}= -3\int_{\R^2}v_2(t,z,y)x_{yy}(t,y)
g_{k1}^*(z,\eta,c(t,y))e^{-iy\eta}\,dzdy\,.
\end{align*}
Let
\begin{align*}
&
R^{41}_k(t,y)=\frac{1}{2\pi}\int_{-\eta_0}^{\eta_0}
\left\{\obu{II}^1_k(t,\eta)+II^2_{k1}(t,\eta)+ II^3_{k11}(t,\eta)\right\}
e^{iy\eta}\,d\eta\,,
\\ &
R^{42}_k(t,y)=\frac{1}{2\pi}\int_{-\eta_0}^{\eta_0}
\left\{\oc{II}^1_{k1}(t,\eta)+II^2_{k2}(t,\eta)+ II^3_{k12}(t,\eta)\right\}
e^{iy\eta}\,d\eta\,.
\end{align*}
Then $R^4_k=R^{41}_k-\pd_y^2R^{42}_k$.
Let $R^6_k=R^{61}_k-\pd_y^2R^{62}_k$ and
\begin{align*}
& R^{61}_k=-6\wP_1\left(\int_\R
\psi_{c(t,y_1),L}(z+3t)v_2(t,z,y_1)\pd_zg_k^*(z,0,c(t,y_1))\,dz\right)\,,
\\ &
R^{62}_k=-\frac{3}{\pi}\int_{-\eta_0}^{\eta_0}\int_{\R^2}
\psi_{c(t,y_1),L}(z+3t)v_2(t,z,y_1)
\pd_zg_{k1}^*(z,\eta,c(t,y_1))e^{i(y-y_1)\eta}\,dy_1dzd\eta\,.
\end{align*}
We can prove the following in the same way as \cite[Claim~D.5]{Miz15}.
\begin{claim}
  \label{cl:R4-R5}
Suppose $\a\in(0,1)$ and $\bM_{c,x}(T)\le \delta$.
If $\delta$ is sufficiently small, then there exists a positive constant $C$
such that for $t\in[0,T]$,
\begin{align*}
& \|\chi(D_y)R^{41}_k(t)\|_{L^1}+\|R^{41}_k(t)\|_{Y_1}+\|R^{42}_k(t)\|_{Y_1}
\le C\la t\ra^{-3/2}(\bM_{c,x}(T)+\bM_1(T)+\bM_2(T))\bM_2(T)\,,
\\ &
\|R^5_k(t)\|_{Y_1} \le  C\la t\ra^{-1}\bM_{c,x}(T)\bM_2(T)\,,
\quad \|R^5_k(t)\|_Y \le  C\la t\ra^{-5/4}\bM_{c,x}(T)\bM_2(T)\,,
\\ &  \label{eq:R6'}
\|\chi(D_y)R^{61}_k\|_{L^1(\R)}+\|R^{61}_k\|_{Y_1}+\|R^{62}_k\|_{Y_1}
\le C\la t\ra^{-1}e^{-\a(3t+L)}\bM_{c,x}(T)\bM_2(T)\,.
\end{align*}
\end{claim}
\par

Let $R^9=R^{91}-\pd_y^2R^{92}$ and
\begin{align*}
& R^{91}=-6\sum_{3\le j\le 5}\bS_{j1}
\{(I+\mathcal{C}_2)(c_yx_y)-(bx_y)_y\}\mathbf{e_1}\,,
\\ &
R^{92}=-6\sum_{3\le j\le 5}\bS_{j2}\{(I+\mathcal{C}_2)(c_yx_y)-(bx_y)_y\}
\mathbf{e}_1\,.
\end{align*}
Using Claims~\ref{cl:S3}--\ref{cl:S5} and boundedness of operators
$\pd_y$, $\bS_1$, $\bS_2$ and $\wC_2$
(\cite[pp.83--84]{Miz15}, \cite[Claims~6.1, B.6]{Miz15}),
we have the following.
\begin{claim}
  \label{cl:R9}
There exist positive constants $C$ and $\delta$ such that
if $\bM_{c,x}(T)\le \delta$, then for $t\in[0,T]$,
\begin{align*}
& \|R^8(t)\|_{Y_1} \le C\la t\ra^{-1}\bM_{c,x}(T)^2\,,
\quad \|R^8(t)\|_Y \le C\la t\ra^{-5/4}\bM_{c,x}(T)^2\,,
\\ &
\|\chi(D_y)R^{91}(t)\|_{L^1}+\|R^{91}(t)\|_{Y_1}+\|R^{92}(t)\|_{Y_1}
\le C(e^{-\a L}+\bM_2(T))\bM_{c,x}(T)^2\la t\ra^{-2}\,.  
\end{align*}
\end{claim}

For $R^{10}=(\pd_y^2\wS_0-B_2)(b_y-c_y)\mathbf{e_1}$,
we have the following from \cite[Claim~D.6]{Miz15}
and the fact that $\wS_0\in B(Y)\cap B(Y_1)$.
\begin{claim}
\label{cl:R10}
 There exist positive constants $C$ and $\delta$ such that
if $\bM_{c,x}(T)\le \delta$, then for $t\in[0,T]$,
\begin{equation*}
\|R^{10}(t)\|_{Y_1} \le C\la t\ra^{-1}\bM_{c,x}(T)^2\,,
\quad \|R^{10}(t)\|_Y \le C\la t\ra^{-5/4}\bM_{c,x}(T)^2\,.
\end{equation*}
\end{claim}

Let $R^{11}=R^{11,1}-\pd_y^2R^{11,2}$ and
$$R^{11,1}=\widetilde{\mathcal{A}}_{11}(t)(\tc -b)\mathbf{e_1}\,,
\quad  R^{11,2}=\widetilde{\mathcal{A}}_{12}(t)(\tc -b)\mathbf{e_1}\,.$$
\begin{claim}
  \label{cl:R11}
There exist positive constants $C$ and $\delta$ such that
if $\bM_{c,x}(T)$, then for $t\in[0,T]$,
\begin{align*}
 \|\chi(D_y)R^{11,1}\|_{L^1}+\|R^{11,1}\|_{Y_1}+\|R^{11,2}\|_{Y_1}  
\le Ce^{-\a(3t+L)}\la t\ra^{-1/2}\bM_{c,x}(T)^2\,.
\end{align*}
\end{claim}
\begin{proof}
By the definition,
$$\chi(D_y)R^{11,1}(t)= 
\chi(D_y)\{\tc(t,\cdot)-b(t,\cdot)\}
(\ta_{11}(t)\mathbf{e_1}+\ta_{21}(t)\mathbf{e_2})\,.$$
Claim~\ref{cl:akbound} and  \eqref{eq:b-tc} imply
\begin{align*}
\|\chi(D_y)R^{11,1}(t)\|_{L^1}\lesssim &
(|\ta_{11}(t)|+|\ta_{12}(t)|)\|\tc(t)\|_Y^2
 \lesssim  e^{-\a(3t+L)}\la t\ra^{-1/2}\bM_{c,x}(T)^2\,.
\end{align*}
We can prove the rest in the similar manner by using
Claim~\ref{cl:akbound}. Thus we complete the proof.
\end{proof}

\bigskip

\section{Estimates for $k(t,y)$}
\label{ap:k}
By Lemmas~\ref{lem:nonzeromean1} and \ref{lem:virial-0},
the $L^2$-norm of $k(t,y)$ decays like $t^{-2}$ as $t\to\infty$.
\begin{claim}
  \label{cl:k-decay}
Suppose that $\inf_{t\ge0\,,\,y\in\R}x_t(t,y)\ge c_1$ for a $c_1>0$.
Then there exist positive constants $\delta$ and $C$ such that if
$\|\la x\ra^2 v_0\|_{H^1(\R^2)}<\delta$, then 
$$\|k(t,y)\|_{L^2}\le C \la t \ra^{-2}\|\la x\ra^2v_0\|_{L^2(\R^2)}\,.$$
\end{claim}
Next, we will give an upper bound of the growth rate of 
$\|k(t,y)\|_{L^1}$ when $v_*(x,y)$ is polynomially localized in $\R^2$.
\begin{claim}
  \label{cl:k-growth}
Let $\tv_1$ be a solution of \eqref{eq:tv1}.
There exist positive constant $C$ and $\eps_0$ such that if
$\left\|\la x\ra(\la x\ra+\la y\ra)v_0\right \|_{H^1(\R^2)}\le \eps_0$,
then for every $t\ge0$,
    \begin{equation*}
\|\la y\ra k(t,\cdot)\|_{L^2(\R)} \le C\la t \ra
\left\|\la x\ra(\la x\ra+\la y\ra)v_0\right \|_{H^1(\R^2)}\,.
\end{equation*}
\end{claim}
\begin{proof}
 Multiplying \eqref{eq:tv1} by $2(1+y^2)\tv_1$ and integrating
 the resulting equation over $\R^2$, we have after some integration  by parts,
 \begin{align*}
\frac{d}{dt}\int_{\R^2} (1+y^2)\tv_1(t,x,y)^2\,dxdy
=& 12\int_{\R^2} y\tv_1(t,x,y)(\pd_x^{-1}\pd_y\tv_1)(t,x,y)\,dxdy\,.
 \end{align*}
By Lemmas~\ref{lem:nonzeromean1} and Lemma~\ref{lem:scattering}
and the definition of $\bM_1'(\infty)$,
\begin{align*}
\|\la y\ra \tv_1(t)\|_{L^2}\le & 
\|\la y\ra v_*\|_{L^2}+6\int_0^t \|\pd_x^{-1}\pd_y\tv_1(s)\|_{L^2}\,ds
\\ \lesssim  & \|\la x\ra\la y\ra v_0\|_{L^2(\R^2)}
+\bM_1'(\infty)t
\\ \lesssim & 
\|\la x\ra\la y\ra v_0\|_{L^2(\R^2)}+
\left(\left\|\la x\ra^2v_0\right\|_{H^1(\R^2)}
+\|\la x\ra\la y\ra v_0\|_{L^2(\R^2)}\right)t\,.
\end{align*}
Thus we complete the proof.
\end{proof}

\section*{Acknowledgment}
This research is supported by JSPS KAKENHI Grant Number 17K05332.

\end{document}